\documentclass[12pt]{amsart}
\usepackage{etex}
\usepackage{etoolbox}
\patchcmd{\thebibliography}{*}{}{}{}
\usepackage{amssymb}
\usepackage{cite}
\usepackage{booktabs}
\usepackage{url}
\usepackage{hyphenat}
\usepackage{mathtools}
\usepackage{pifont}
\usepackage[all,cmtip]{xy}
\usepackage{ifpdf}
\usepackage{enumitem}
\usepackage{xcolor}
\usepackage{graphicx}
\usepackage[lmargin=1in,rmargin=1in,tmargin=1in,bmargin=1in]{geometry}

\usepackage{mathrsfs}
\usepackage{xr-hyper}

\definecolor{darkblue}{rgb}{0,0,0.4} 
\usepackage[colorlinks=true, citecolor=darkblue, filecolor=darkblue, linkcolor=darkblue,urlcolor=darkblue]{hyperref} 
\usepackage[all]{hypcap}

\allowdisplaybreaks

\usepackage{tikz}
\usetikzlibrary{matrix,arrows,3d,positioning,calc}
\tikzset{zxplane/.style={canvas is zx plane at y=#1,very thin}}
\tikzset{xyplane/.style={canvas is xy plane at z=#1,very thin}}
\tikzset{yzplane/.style={canvas is yz plane at x=#1,very thin}}
\tikzstyle{intext}=[rectangle,fill=white,inner sep=1pt,outer sep=2pt, fill opacity=0.7,text opacity=1]
\tikzset{doubar/.style={double, double equal sign distance, -implies}}

\graphicspath{{draws/}{}}

\newcommand{\widebar}{\overline}

\newcommand{\RR}{\mathbb R}
\newcommand{\CC}{\mathbb C}
\newcommand{\ZZ}{\mathbb Z}
\newcommand{\QQ}{\mathbb Q}

\newcommand{\FF}{\mathbb F}
\newcommand{\NN}{\mathbb N}
\newcommand{\Ker}{\mathrm{ker}}
\newcommand{\Coker}{\mathrm{coker}}

\newcommand{\bD}{\mathbb{D}}

\newcommand{\co}{\nobreak\mskip2mu\mathpunct{}\nonscript
  \mkern-\thinmuskip{:}\penalty300\mskip6muplus1mu\relax}
\newcommand{\from}{\co}

\newcommand{\bdy}{\partial}
\newcommand{\into}{\hookrightarrow}

\newcommand{\lbracket}{[}
\newcommand{\rbracket}{]}

\newcommand{\spinc}{\mathfrak s}

\DeclareMathOperator{\sign}{sign}

\DeclareMathOperator{\Hom}{Hom}
\DeclareMathOperator{\Iso}{Iso}
\DeclareMathOperator{\Map}{Map}
\DeclareMathOperator{\Fun}{Fun}

\DeclareMathOperator{\Ext}{Ext}

\DeclareMathOperator{\im}{im}
\DeclareMathOperator{\spin}{spin}
\DeclareMathOperator{\pin}{pin}

\DeclareMathOperator{\Pin}{\mathit{Pin}}

\DeclareMathOperator{\ind}{ind}

\DeclareMathOperator{\ev}{ev}
\DeclareMathOperator{\gr}{gr}

\newcommand\interior{\mathrm{int}}

\renewcommand{\emptyset}{\varnothing}

\DeclareMathOperator{\SO}{\mathit{SO}}

\DeclareMathOperator{\Gr}{\mathit{Gr}}
\DeclareMathOperator{\Lag}{\mathit{Lag}}
\DeclareMathOperator{\Fr}{\mathrm{Fr}}

\theoremstyle{plain}

\numberwithin{equation}{section}
\newtheorem{theorem}[equation]{Theorem}

\newtheorem{proposition}[equation]{Proposition}
\newtheorem{lemma}[equation]{Lemma}
\newtheorem{corollary}[equation]{Corollary}

\newtheorem{conjecture}[equation]{Conjecture}

\newtheorem{convention}[equation]{Convention}
\newtheorem{definition}[equation]{Definition}

\newtheorem{hypothesis}[equation]{Hypothesis}

\theoremstyle{definition}

\theoremstyle{remark}
\newtheorem{example}[equation]{Example}
\newtheorem{remark}[equation]{Remark}

\newcommand{\HF}{\mathit{HF}}

\newcommand{\CF}{{\mathit{CF}}}

\newcommand{\x}{\mathbf x}

\newcommand\HH{\mathit{HH}}

\newcommand\Hochschild\HH

\newcommand{\cM}{\mathcal{M}}

\newcommand{\ocM}{\widebar{\cM}{}} 

\newcommand\Id{\mathbb{I}}

\newcommand{\Field}{\FF}
\newcommand{\Ring}{R}
\DeclareMathOperator{\nbd}{nbd}

\newcommand{\dbar}{\bar{\partial}}

\renewcommand{\th}{^\text{th}}

\newcommand{\Cat}{\mathscr{C}}

\DeclareMathOperator{\ob}{Ob}

\newcommand{\op}{\mathrm{op}}

\newcommand{\ol}[1]{\overline{#1}{}}
\newcommand{\wt}[1]{\widetilde{#1}{}}

\makeatletter
\newcommand\honestalg[3]{\bigl\lbracket
\begin{smallmatrix} #1\@ifempty{#3}{}{&#3} \\ #2 \end{smallmatrix}
\bigr\rbracket}

\makeatother

\newcommand{\cube}{\mathit{cube}}
\newcommand{\Morse}{\mathit{Morse}}

\newcommand{\Moduli}{\cM}
\newcommand{\pModuli}{\mathcal{N}}
\newcommand{\oModuli}{\overline{\Moduli}}
\newcommand{\opModuli}{\overline{\pModuli}}

\newcommand{\Filt}{\mathcal{F}}

\newcommand{\pt}{\mathit{pt}}

\newcommand{\Crit}{\mathit{Crit}}
\newcommand{\Hess}{\mathit{Hess}}

\newcommand{\JSpace}{\mathcal{J}}

\newcommand{\cyl}{\mathrm{cyl}}
\newcommand{\ECat}{\mathscr{E}}
\newcommand{\BCat}{\mathscr{B}}
\newcommand{\ICat}{\mathscr{I}}

\newcommand{\Complexes}{\mathsf{Kom}}
\DeclareMathOperator{\hocolim}{hocolim}
\DeclareMathOperator{\fathocolim}{focolim}
\newcommand{\KhSymp}{\mathit{Kh}_{\mathit{symp}}}

\DeclareMathOperator{\Hilb}{Hilb}

\newcommand{\ssspace}[1]{\mathcal{Y}_{#1}}

\newcommand{\eH}[1][{\ZZ/2}]{H_{#1}}
\newcommand{\eHF}[1][{\ZZ/2}]{\HF_{\!#1}}

\makeatletter

\newcommand*\wthelper[2]{%
        \hbox{\dimen@\accentfontxheight#1%
                \accentfontxheight#11.3\dimen@
                $\m@th#1\widetilde{#2}$%
                \accentfontxheight#1\dimen@
        }%
}
\newcommand*\accentfontxheight[1]{%
        \fontdimen5\ifx#1\displaystyle
                \textfont
        \else\ifx#1\textstyle
                \textfont
        \else\ifx#1\scriptstyle
                \scriptfont
        \else
                \scriptscriptfont
        \fi\fi\fi3
}
\makeatother

\makeatletter
\newlength\xvec@height%
\newlength\xvec@depth%
\newlength\xvec@width%
\newcommand{\xvec}[2][]{%
  \ifmmode%
    \settoheight{\xvec@height}{$#2$}%
    \settodepth{\xvec@depth}{$#2$}%
    \settowidth{\xvec@width}{$#2$}%
  \else%
    \settoheight{\xvec@height}{#2}%
    \settodepth{\xvec@depth}{#2}%
    \settowidth{\xvec@width}{#2}%
  \fi%
  \def\xvec@arg{#1}%
  \def\xvec@dd{:}%
  \def\xvec@d{.}%
  \raisebox{.2ex}{\raisebox{\xvec@height}{\rlap{%
    \kern.05em
    \begin{tikzpicture}[scale=1]
    \pgfsetroundcap
    \draw (.05em,0)--(\xvec@width-.05em,0);
    \draw (\xvec@width-.05em,0)--(\xvec@width-.15em, .075em);
    \draw (\xvec@width-.05em,0)--(\xvec@width-.15em,-.075em);
    \ifx\xvec@arg\xvec@d%
      \fill(\xvec@width*.45,.5ex) circle (.5pt);%
    \else\ifx\xvec@arg\xvec@dd%
      \fill(\xvec@width*.30,.5ex) circle (.5pt);%
      \fill(\xvec@width*.65,.5ex) circle (.5pt);%
    \fi\fi%
    \end{tikzpicture}%
  }}}%
  #2%
}
\makeatother

\newcommand{\SingG}[1][{\bullet}]{G_{#1}}
\newcommand{\SingX}[1][{\bullet}]{X_{#1}}

\newcommand{\ChainComplexes}{\mathsf{Kom}}
\newcommand{\Sets}{\mathsf{Sets}}

\newcommand{\chainFb}{F}

\newcommand{\Nerve}{\mathscr{N}}
\newcommand{\NerveAc}{\mathscr{N}^{\!\mathit{AC}}}
\newcommand{\NerveWAc}{\mathscr{N}^{\!\mathit{ac}}}
\newcommand{\SSets}{\mathsf{SSet}}
\newcommand{\Spaces}{\mathsf{Top}}
\newcommand{\SSetsOther}{\mathsf{Set}^{\Delta^{\op}}}

\newcommand{\CVcat}[1][\bullet]{S[#1]}
\newcommand{\sCVcat}[1][\bullet]{S^{\mathit{sm}}[#1]}

\newcommand{\MarkLine}[1][{}]{\mathit{ML}_{#1}}
\DeclareMathOperator{\gluing}{gl}
\newcommand{\oMarkLine}[1][{}]{\overline{\mathit{ML}}_{#1}}
\newcommand{\JH}{(\wt{J},\wt{H})}
\newcommand{\JHp}{(\wt{J}',\wt{H}')}
\newcommand{\Line}{\mathbb{L}}

\newcommand{\action}[2]{#1#2}

\newcommand{\GJHspace}{\mathcal{J}_G}
\newcommand{\GJHcat}{\mathcal{G\!J}}
\newcommand{\FloerFunc}{F{\scriptscriptstyle\mathit{\!\!loer}}}
\newcommand{\maxFloerFunc}{\maxmap{F}{\scriptscriptstyle\mathit{\!\!loer}}}

\newcommand{\upto}[1]{\mathfrak{m}}

\newcommand{\maxmap}[1]{\xvec{#1}}
\newcommand{\messygr}[2]{[{#1}|{#2}]}

\newcommand{\Realize}[1]{|#1|}
\newcommand{\FatRealize}[1]{\|#1\|}

\newcommand{\Nervesm}{\Nerve^{\text{sm}}}

\newcommand{\brane}{\mathrm{br}}

\newcommand{\metric}{\langle\cdot,\cdot\rangle}
\newcommand{\MPspace}{\mathcal{J}_G}
\newcommand{\MPcat}{\mathcal{G\!J}}
\newcommand{\tmetric}{\wt{\metric}}
\newcommand{\MorseFunc}{F_{\scriptscriptstyle\mathit{\!\!Morse}}}
\newcommand{\cubeF}{F_{\scriptscriptstyle\mathit{\!\!\cube}}}
\newcommand{\maxMorseFunc}{\maxmap{F}_{\scriptscriptstyle\mathit{\!\!Morse}}}
\newcommand{\spaceF}{F_{\scriptscriptstyle\mathit{\!\!sp}}}

\newcommand{\diff}{\delta}
\newcommand{\BGSection}{\mathcal{S}}

\newcommand{\Polyh}{P}

\newcommand{\inj}{\mathrm{inj}}

\newcommand{\cubemap}[1]{\maxmap{#1}}

\begin{document}
\title{A simplicial construction of G-equivariant Floer homology}

\author{Kristen Hendricks}
 \address{Mathematics Department, Rutgers University\\
   New Brunswick, NJ 08854}
  \thanks{\texttt{KH was supported by NSF grant DMS-1663778.}}
\email{\href{mailto:kristen.hendricks@rutgers.edu}{kristen.hendricks@rutgers.edu}}

\author{Robert Lipshitz}
 \address{Department of Mathematics, University of Oregon\\
   Eugene, OR 97403}
\thanks{\texttt{RL was supported by NSF grant DMS-1149800.}}
\email{\href{mailto:lipshitz@uoregon.edu}{lipshitz@uoregon.edu}}

\author{Sucharit Sarkar}
\thanks{\texttt{SS was supported by NSF grant DMS-1643401.}}
\address{Department of Mathematics, University of California\\
  Los Angeles, CA 90095}
\email{\href{mailto:sucharit@math.ucla.edu}{sucharit@math.ucla.edu}}

\date{\today}


\begin{abstract}
  For $G$ a Lie group acting on a symplectic manifold $(M,\omega)$
  preserving a pair of Lagrangians $L_0$, $L_1$, under certain
  hypotheses not including equivariant transversality we construct a
  $G$-equivariant Floer cohomology $\eHF[G](L_0,L_1)$.
\end{abstract}

\maketitle

\tableofcontents


\section{Introduction}
In the 1980s, Floer introduced a collection of semi-infinite
dimensional Morse-like
homologies\cite{Floer88:unregularized,Floer88:LagrangianHF,Floer88:instanton}. These
and later Floer-type invariants are often enriched by various kinds of
symmetries. For example, Seiberg-Witten Floer homology and cylindrical
contact homology are intrinsically $S^1$-equivariant
theories~\cite{KronheimerMrowka,Manolescu03:SW-spectrum-1,BO:equi-symp-hom},
and Fukaya categories often carry actions of mapping class groups
(e.g.,~\cite{KhS02:BraidGpAction,SeidelSmith6:Kh-symp}) and Lie
algebras (e.g.,~\cite{Seidel:dilating,Lekili:G-mflds}).

In this paper, we will focus on Lagrangian intersection Floer
cohomology in the presence of extrinsic symmetries: 
an action of a Lie group $G$ on a symplectic manifold $M$ preserving
Lagrangians $L_0$ and $L_1$. Our goal is to construct an equivariant
Floer cohomology, modeled on Borel equivariant (singular) cohomology,
and show that it enjoys good properties. The naive approach
is to consider $G$-equivariant almost complex structures and
Hamiltonian perturbations. The difficulty is that even for $G$ a discrete
group it is often impossible to achieve transversality for moduli
spaces of holomorphic curves via equivariant choices, and for $G$ a
positive-dimensional Lie group, $G$-equivariant transversality seems to
occur rarely indeed.

One way to construct equivariant Floer (co)homology without
equivariant transversality is to consider the space of
(non-equivariant) Lagrangians and build a kind of $G$-equivariant
local system over this space. There are at least two approaches to
making this rigorous. One is in terms of Floer theory on $M$ coupled to Morse
theory on $EG$; this approach is sketched by
Seidel-Smith\cite{SeidelSmith10:localization,Seidel:equi-pants} (who
focus on the case $G=\ZZ/2$). The other is to use notions of homotopy
coherence and a simplicial model for $BG$; in the case that $G$ is a
finite group and we consider Floer complexes with coefficients in
$\FF_2$ this approach was explained in our previous
paper~\cite{HLS:HEquivariant}. The goal of this paper is to extend
the homotopy coherence approach to $G$ an arbitrary compact Lie group (or, under
additional hypotheses on the action, any Lie group) and, under appropriate
hypotheses, coefficients in $\ZZ$ (or any other ring). Our main
construction / theorem is:

\begin{theorem}
  Suppose a Lie group $G$ acts on a symplectic manifold $(M,\omega)$
  by symplectomorphisms preserving Lagrangians $L_0$ and $L_1$,
  satisfying the following hypotheses:
  \begin{enumerate}
  \item For any loop of paths in $M$ from $L_0$ to $L_1$, both the
    symplectic area and the Maslov index vanish.  (See
    Point~\ref{item:J-1} in Hypothesis~\ref{hyp:Floer-defined}.)
  \item The manifold $M$ is either compact or is convex at infinity,
    and in the latter case $L_0$ and $L_1$ are compact or conical at
    infinity. (See Point~\ref{item:J-2} in
    Hypothesis~\ref{hyp:Floer-defined}.)
  \item The group $G$ is either compact or else every $G$-twisted loop
    of paths has Maslov index $0$. (See
    Section~\ref{sec:twisted-pi2}.)
  \end{enumerate}
  Then:
  \begin{enumerate}
  \item There is an associated $G$-equivariant Floer cohomology group
    $\eHF[G](L_0,L_1)$ over the ground ring $\Ring=\FF_2$.
  \item If the configuration spaces of Whitney disks in $(M,L_0,L_1)$ admit a
    $G$-orientation system (Definition~\ref{def:G-or-system}) then
    there is an associated $G$-equivariant Floer cohomology group
    $\eHF[G](L_0,L_1)$ over any ground ring $\Ring$. In particular, this
    holds whenever $L_0$ and $L_1$ admit $G$-equivariant brane data (see
    Proposition~\ref{prop:spin-or}).
  \item The equivariant Floer cohomology $\eHF[G](L_0,L_1)$ is a module
    over the group cohomology $H^*(BG;\Ring)$ of $G$ with coefficients in
    $\Ring$.
  \item If $\Ring=\Field$ is a field then there is a spectral sequence with $E^2$-page
    $H^*\bigl(BG;\HF^*(L_0,L_1;\Field)\bigr)$ converging to $
    \eHF[G](L_0,L_1; \Field)$. In particular, if $G$ is connected the $E^2$-page of this spectral sequence is 
    $H^*(BG;\Field)\otimes \HF^*(L_0,L_1;\Field)$.
  \item In the case that $G$ is a finite group and $\Ring=\FF_2$,
    $\eHF[G](L_0,L_1)$ is isomorphic to the equivariant Floer
    cohomology as defined in our previous
    paper~\cite{HLS:HEquivariant}.
  \end{enumerate}
\end{theorem}
This is a combination of Definition~\ref{def:equivariant-cohomology},
Theorems~\ref{thm:eHF-inv} and~\ref{thm:finite-same},
and Proposition~\ref{prop:sseq}.

Versions of equivariant Lagrangian intersection Floer homology have
appeared elsewhere in the literature. One of the earliest instances 
is Khovanov-Seidel's use of $\ZZ/2$-equivariant
almost complex structures to simplify computations of Floer homology
with $\Field_2$-coefficients in certain cases~\cite[Lemma
5.14]{KhS02:BraidGpAction}. There, strong hypotheses on the group
action guaranteed that one could achieve equivariant
transversality. Seidel later constructed and exploited an action of
$\ZZ/2$ on the Fukaya category, again under strong hypotheses on the
group action~\cite{SeidelBook,Seidel:quartic}.  Free
actions were further exploited by Wu~\cite{Wu:free-action}. These
constructions were extended to fairly general finite group actions
by Cho-Hong~\cite{CH:group-actions}, for coefficient rings containing
$\QQ$. The hypothesis on the coefficients is essential for their
construction, as they average the differential over a $G$'s worth of
complex structures. Consequently, their construction does not carry an
interesting action by $H^*(BG)$.

The constructions in this paper can be seen as a more homotopy-theoretic 
kind of averaging over the group action, allowing us to work
with coefficients in arbitrary rings (and, in particular,
$\Field_2$). While we will not pursue the analogy with Cho-Hong's work
further, and in particular will not attempt to construct an
equivariant Fukaya category, we will make use of some of their ideas when
orienting moduli spaces in Section~\ref{sec:G-or}.

The paper is organized as follows. The language that we use to
formulate the coherent Floer data in the construction of
$\eHF[G](L_0,L_1)$ is that of infinity
(or $(\infty,1)$) categories. We review the aspects of this subject
that are needed for the construction in
Section~\ref{sec:background}. (The exposition is intended to be a
convex combination of efficient and explicit, so while we omit some
details it does not assume any familiarity with higher categories or
algebra.) Section~\ref{sec:construction} gives the construction of
$G$-equivariant Floer cohomology, culminating in
Section~\ref{seq:Floer-coho}, which gives both the definition of $\eHF[G](L_0,L_1)$ and a
proof of its invariance. We also outline the analogous construction in
Morse theory, in
Section~\ref{sec:Morse-const}. Section~\ref{sec:computations} gives
further properties and computations of these invariants: an
identification of this construction in the discrete case with our
earlier construction~\cite{HLS:HEquivariant}
(Section~\ref{sec:discrete}); the proof that the Morse-theory analogue
gives the usual Borel equivariant cohomology
(Section~\ref{sec:comp-morse}); the spectral sequence
$H^*(BG;\HF(L_0^H,L_1;\Field))\Rightarrow\eHF[G](L_0,L_1;\Field)$
(Section~\ref{sec:sseq}); and a brief discussion of $S^1$- and
$O(2)$-equivariant symplectic Khovanov homology
(Section~\ref{sec:kh-symp}).

\emph{Acknowledgments.} We thank Mohammed Abouzaid, Dan Dugger, Tyler
Lawson, Tim Perutz, and Dev Sinha for helpful conversations. In
particular, this paper is partly inspired by Abouzaid's work
on family Floer homology~\cite{Abouzaid:Family}. We also thank the
contributors to nLab, the articles in which helped us parse the
infinity-categorical literature. Finally, we thank the referee for
a careful reading and many helpful suggestions.


\section{Background and conventions}\label{sec:background}
We collect some background about infinity categories which we
need later. Throughout this section, we provide references to
treatments with further details, but have not attempted to cite the
original sources for this material.

\begin{convention}
  Fix a commutative, unital base ring $\Ring$. Unless otherwise mentioned, all algebraic
  constructions in this paper will be over $\Ring$. Usually $\Ring$ will be
  $\ZZ$ or some field $\Field$.
\end{convention}

\subsection{Simplicial sets}
To fix notation, recall that a \emph{simplicial set} $S$ consists of
a sequence of sets $S_n$, $n=0,1,2,\dots$, and maps
\begin{align*}
  d_j&\co S_n\to S_{n-1} & 
  s_j&\co S_n\to S_{n+1}
\end{align*}
for $0\leq j\leq n$, satisfying the simplicial relations. (See,
e.g.,~\cite[Section I.1]{GoerssJardine09}.) The maps $d_j$ are the
\emph{face maps} and the maps $s_j$ are the \emph{degeneracy maps};
the elements of $S_n$ are called $n$-simplices. Equivalently, a
simplicial set is a functor from $\Delta^\op$ to $\Sets$, where
$\Delta$ is the category with one object $[n]=\{0,\dots,n\}$ for each
non-negative integer $n$ and with $\Hom([m],[n])$ the set of monotone
functions $\alpha\co [m]\to[n]$. The face map
$d_j\from S_n\to S_{n-1}$ comes from the monotone injective map
$[n-1]\to[n]$ that misses $j$, and the degeneracy map
$s_j\from S_{n}\to S_{n+1}$ comes from the monotone surjective map
$[n+1]\to[n]$ sending $j$ and $j+1$ to $j$.

For $\sigma\in S_n$ and $0\leq i_0<\dots<i_k\leq n$, let
$\sigma_{i_0,\dots,i_k}$ denote the element $\alpha^*(\sigma)\in S_k$,
where $\alpha\from[k]\to[n]$ is the map $j\mapsto i_j$.

\subsection{Cubes}\label{sec:cubes}
The \emph{standard $n$-cube} is $[0,1]^n$. The boundary facets of
$[0,1]^n$ are
\[
  \bdy_{i,\epsilon}[0,1]^n=\{(x_1,\dots,x_n)\in[0,1]^n\mid x_i=\epsilon\}
\]
where $i\in\{1,\dots,n\}$ and $\epsilon\in\{0,1\}$. Projecting off the
$i\th$ coordinate identifies $\bdy_{i,\epsilon}[0,1]^n$ with the
standard $(n-1)$-cube. This identification is orientation-preserving
if $i+\epsilon$ is even and orientation-reversing otherwise.

\subsection{Infinity categories and functors}\label{sec:infty-cat-functors}
Let $\Delta^n$ be the standard $n$-simplex. We often view $\Delta^n$ as a
simplicial set with $k$-simplices $\sigma$ corresponding to length-$(k+1)$
increasing sequences in $\{0,1,\dots,n\}$; the elements of the sequence are the
\emph{vertices} of the $k$-simplex $\sigma$. A simplex in $\Delta^n$ is
non-degenerate if the corresponding sequence is strictly increasing.

For convenience later, given $-1\leq i\leq k$, let
$\Delta^1_{k;i}$ denote the $k$-simplex in $\Delta^1$ whose vertices
are $\overbrace{0,\dots,0}^{i+1},\overbrace{1,\dots,1}^{k-i}$. (The
simplex $\Delta^1_{k;j}$ is degenerate if $k>1$.)

Let $\Lambda_i^n$ denote the $i\th$ \emph{horn} of $\Delta^n$, that is,
$\Delta^n$ with the $n$-dimensional cell and the $i\th$
$(n-1)$-dimensional cell deleted. So, $\Lambda_i^n$ is a simplicial set whose
simplices correspond to increasing sequences in $\{0,\dots,n\}$ not containing
$\{0,1,\dots,\hat{i},\dots,n\}$ or $\{0,\dots,n\}$.

Recall~\cite[Definition 1.1.2.4]{Lurie09:topos} that a \emph{weak Kan
  complex} or an \emph{infinity category} (or $(\infty,1)$-category, or quasicategory)
is a simplicial set $S$ which satisfies the \emph{weak Kan
  condition}: if $0<i<n$, and $f_0\co \Lambda_i^n\to S$ is a
map of simplicial sets then $f_0$ can be extended to a map
$f\co \Delta^n\to S$. The elements of $S_0$ are the
\emph{objects} of $S$, sometimes denoted $\ob(S)$. The
elements of $S_n$ (for $n>0$) are the \emph{$n$-morphisms} of
$S$. The infinity category is a \emph{Kan
  complex} if it satisfies the horn-filling condition for $i=0,n$ as
well.

Given any topological space $X$, the \emph{singular set} $\SingX$ of
$X$, which has $\SingX[n]$ the set of continuous maps $\Delta^n\to X$,
is a simplicial set satisfying the Kan condition (because the topological
$n$-simplex retracts to any of its horns). So, we can view $\SingX$ as
an infinity category.

A \emph{functor} of infinity categories is simply a map of simplicial
sets~\cite[Section 1.2.7]{Lurie09:topos}.

\subsection{Semi-simplicial sets}
We will occasionally use \emph{semi-simplicial sets}, which
have face maps but no degeneracy maps. Equivalently, a semi-simplicial
set is a functor from $\Delta^\op_{\inj}$ to $\Sets$, where $\Delta_{\inj}$ has
objects the non-negative integers and $\Hom([m],[n])$ the set of
monotone injections from $[m]$ to $[n]$. 

If $\Sets^{\Delta^\op}$ (respectively $\Sets^{\Delta^\op_\inj}$)
denotes the category of simplicial sets (respectively semi-simplicial
sets) then the inclusion $j\co \Delta^\op_{\inj}\to\Delta^\op$ induces
a forgetful map $j^*\co
\Sets^{\Delta^\op}\to\Sets^{\Delta^\op_\inj}$. A semi-simplicial set
$X$ has a \emph{fat realization}, $\FatRealize{{X}}$
which is defined like the geometric realization $\Realize{X}$ except
without quotienting by the relations involving the degeneracy maps of
$X$. A map of semi-simplicial sets induces a map of fat
realizations. Further, if $X$ is a simplicial set then the
quotient map $\FatRealize{{j^*X}}\to\Realize{X}$ is a
weak equivalence~\cite[Appendix A]{Segal:categories}, and is
functorial with respect to maps of simplicial sets.

Recall that a map of simplicial sets is a \emph{weak equivalence} if
the induced map of fibrant replacements induces an isomorphism on all
homotopy groups. This is equivalent to the induced map of geometric
realizations being a weak equivalence. In fact, two simplicial sets
are weakly equivalent if and only if their geometric realizations
are~\cite[Theorem I.11.4]{GoerssJardine09}.

\begin{lemma}\label{lem:semi-weak-equiv} 
  Let $S,T\in\Sets^{\Delta^\op}$ be simplicial sets. If
  $f\from j^*S\to j^*T$ is a semi-simplicial map that induces weak
  equivalence of fat realizations then $S$ and $T$ are weakly
  equivalent as simplicial sets.
\end{lemma}
\begin{proof}
  By hypothesis, $\FatRealize{f}\co
  \FatRealize{j^*S}\to\FatRealize{j^*T}$ is a weak
  equivalence of CW complexes. So, we have
  \[
    \Realize{S}\stackrel{\simeq}{\longleftarrow}\FatRealize{{j^*S}}\stackrel{\simeq}{\longrightarrow}\FatRealize{j^*T}\stackrel{\simeq}{\longrightarrow}\Realize{T}.
  \]
  It follows that $S$ and $T$ are weakly equivalent simplicial sets.
\end{proof}

The following lemma will be convenient for showing that two versions of the simplicial nerve agree:
\begin{lemma}\label{lem:fat}
  Let $S$ and $T$ be simplicial sets. If there are maps $\alpha\from
  j^*S\to j^*T$, $\beta\co j^*T\to j^*S$, $H\co j^*(\Delta^1\times
  T)\to j^*T$, and $K\from j^*(\Delta^1\times S)\to j^*S$ of
  semi-simplicial sets so that
  \begin{align*}
    H|_{\{0\}\times T}&=\Id_{T} & H|_{\{1\}\times T}&=\alpha\circ\beta\\
    K|_{\{0\}\times S}&=\Id_{S} & K|_{\{1\}\times S}&=\beta\circ\alpha
  \end{align*}
  then $S$ and $T$ are weakly equivalent as simplicial sets.
\end{lemma}
\begin{proof}
  Using Lemma~\ref{lem:semi-weak-equiv}, it suffices to show that
  $\FatRealize{\alpha}$ is a homotopy equivalence.  We have
  $\FatRealize{\alpha}\circ\FatRealize{\beta}=\FatRealize{H}\circ\FatRealize{j^*\iota_{1}}$
  and
  $\Id_{\FatRealize{j^*T}}=\FatRealize{H}\circ\FatRealize{j^*\iota_{0}}$,
  where $\iota_i\co\Delta^0\to\Delta^1$ are the endpoint inclusions.
  Hence, to show that $\FatRealize{\alpha}\circ\FatRealize{\beta}$ is
  homotopic to the identity it suffices to show that
  $\FatRealize{j^*\iota_{0}}$ and $\FatRealize{j^*\iota_{1}}$ are
  homotopic.  This follows from the diagram
  \[
    \begin{tikzpicture}[xscale=6,yscale=1.5]
      \node at (0,0) (fat) {$\FatRealize{j^*T}$};
      \node at (0,-1) (fatDelt) {$\FatRealize{j^*(\Delta^1\times T)}$};
      \node at (1,0) (thin) {$\Realize{T}$};
      \node at (1,-1) (thinDelt) {$\Realize{\Delta^1\times T}=[0,1]\times \Realize{T}$};

      \draw[->] (fat) to[out=-120,in=120] node[left]{\tiny $\FatRealize{j^*\iota_{0}}$} (fatDelt);
      \draw[->] (fat) to[out=-60,in=60] node[right]{\tiny $\FatRealize{j^*\iota_{1}}$} (fatDelt);
      \draw[->] (fat) to node[above]{\tiny $\simeq$} (thin);
      \draw[->] (thin) to[out=-120,in=120] node[left]{\tiny $\Realize{\iota_{0}}$} (thinDelt);
      \draw[->] (thin) to[out=-60,in=60] node[right]{\tiny $\Realize{\iota_{1}}$} (thinDelt);
      \draw[->] (fatDelt) to node[above]{\tiny $\simeq$} (thinDelt);
    \end{tikzpicture}
  \]
  where the rectangular face involving $\FatRealize{j^*\iota_i}$ and
  $\Realize{\iota_i}$ commutes for $i=0,1$, and the right bigon
  commutes up to homotopy, forcing the left bigon to commute up to
  homotopy as well, since the horizontal arrows are homotopy
  equivalences.
  
  A similar argument shows that
  $\FatRealize{\beta}\circ\FatRealize{\alpha}$ is homotopic to
  $\Id_{\FatRealize{j^*S}}$, proving the result.
\end{proof}

\subsection{The category of chain complexes}\label{sec:chain-complexes}
Another infinity category of importance to us is
$\ChainComplexes$, the infinity category of bounded-below chain
complexes of free $\Ring$-modules. Following
Lurie~\cite[Construction 1.3.1.6]{Lurie:HigherAlgebra}, the
$n$-simplices in $\ChainComplexes$ are pairs $(\{C_j\},\{f_J\})$ where
$\{C_j\}$ is a sequence of $n+1$ bounded-below chain complexes
$C_0,\cdots,C_n$ and $\{f_J\}$ is a collection of maps $f_{J}\co
C_{j_0}\to C_{j_k}$, one for each increasing sequence $J=(0\leq
j_0<\cdots<j_k\leq n)$ (with $k\geq 1$), of degree $k-1$, satisfying
\begin{equation}\label{eq:chain-cx-cat}
  \bdy_{C_{j_k}}\circ f_{j_0,\dots,j_k}+(-1)^{k}f_{j_0,\dots,j_k}\circ \bdy_{C_{j_0}} = 
  \sum_{0<\ell<k}(-1)^{k-\ell+1}\left(f_{J\setminus\{j_\ell\}}-f_{j_\ell,\dots,j_k}\circ f_{j_0,\dots,j_\ell}\right).
\end{equation}
The structure
maps of the simplicial set $\ChainComplexes$ are given as follows. For
each monotone function $\alpha\co [m]\to[n]$ and $n$-simplex
$(\{C_j\},\{f_J\})$ of $\ChainComplexes$ there is a corresponding
$m$-simplex $\alpha^*(\{C_j\},\{f_J\})=(\{D_\ell\},\{g_L\})$ where
$D_{\ell}=C_{\alpha(\ell)}$ and $g_L$ is:
\begin{itemize}
\item $f_{\alpha(L)}$ if $|\alpha(L)|=|L|$,
\item $\Id_{C_j}$ if $L=\{\ell_1,\ell_2\}$ has two elements and
  $\alpha(\ell_1)=\alpha(\ell_2)=j$, and
\item the zero map in all other cases.
\end{itemize}

In fact, the category of chain complexes is a \emph{stable} infinity
category~\cite[Proposition 1.3.2.10]{Lurie:HigherAlgebra},
i.e., $\ChainComplexes$ has a zero object, which is the zero complex;
every morphism admits a fiber and a cofiber; and fiber and cofiber
sequences agree. The cofiber of a morphism $f$ is simply the mapping
cone of $f$, and the (induced) shift operator is just the usual shift
operator on chain complexes.

\begin{remark}
  Our convention for indexing the morphisms $f_{j_0,\dots,j_k}$
  differs slightly from Lurie's; consequently, the sign in
  Formula~\eqref{eq:chain-cx-cat} differs from Lurie's, as well.
\end{remark}

\begin{example}\label{ref:3-cell-in-kom}
A $3$-simplex in $\ChainComplexes$ consists of four chain complexes
$C_0,C_1,C_2,C_3$ with maps
\[
\begin{tikzpicture}[xscale=2,yscale=1.5]
\tikzset{every
  node/.style={fill=white,inner sep=0pt,outer sep=1pt, fill
    opacity=0.7,text opacity=1}}

\node (a0) at (0,0) {$C_0$};
\node (a1) at (1,1) {$C_1$};
\node (a2) at (3,1) {$C_2$};
\node (a3) at (4,0) {$C_3$};

\draw[->] (a0) -- (a1) node[pos=0.5] {\tiny $f_{01}$};
\draw[->] (a1) -- (a2) node[pos=0.5] {\tiny $f_{12}$};
\draw[->] (a2) -- (a3) node[pos=0.5] {\tiny $f_{23}$};
\draw[->] (a0) -- (a2) node[pos=0.5] {\tiny $f_{02}$};
\draw[->] (a1) -- (a3) node[pos=0.5] {\tiny $f_{13}$};
\draw[->] (a0) -- (a3) node[pos=0.5] {\tiny $f_{03}$};

\draw[->,red,dashed] (a0) .. controls (0.6,0.6) and (2,1) .. (a2) node[pos=0.5] {\tiny $f_{012}$};
\draw[->,red,dashed] (a1) .. controls (2,1) and (4-0.6,0.6) .. (a3) node[pos=0.5] {\tiny $f_{123}$};
\draw[->,red,dashed] (a0) .. controls (0.9,0.5) .. (a3) node[pos=0.5] {\tiny $f_{013}$};
\draw[->,red,dashed] (a0) .. controls (4-0.9,0.5) .. (a3) node[pos=0.5] {\tiny $f_{023}$};

\draw[->,green!50!black,dotted] (a0) .. controls (2,-0.5) .. (a3) node[pos=0.5] {\tiny $f_{0123}$};
\end{tikzpicture}
\]
where the solid maps are degree-$0$ chain maps, the
{\color{red}dashed} maps are degree-$1$ chain homotopies, and the
{\color{green!50!black}dotted} map is a degree-$2$ homotopy of
homotopies filling in the following square of maps $C_0\to C_3$:
\[
\begin{tikzpicture}[cm={0,1,1,0,(0,0)},scale=1.7]
\tikzset{every
  node/.style={fill=white,inner sep=0pt,outer sep=1pt, fill
    opacity=0.7,text opacity=1}}

\draw (0,0) rectangle (2,2);

\node at (0,0) {\tiny $f_{03}$};
\node at (0,2) {\tiny $f_{13}\!\circ\! f_{01}$};
\node at (2,0) {\tiny $f_{23}\!\circ\! f_{02}$};
\node at (2,2) {\tiny $f_{23}\!\circ\! f_{12}\!\circ\! f_{01}$};

\node at (1,0) {\tiny ${\color{red}f_{023}}$};
\node at (0,1) {\tiny ${\color{red}f_{013}}$};
\node at (2,1) {\tiny $f_{23}\!\circ\! {\color{red}f_{012}}$};
\node at (1,2) {\tiny ${\color{red}f_{123}}\!\circ\! f_{01}$};

\node at (1,1) {\tiny ${\color{green!50!black}f_{0123}}$};
\end{tikzpicture}
\]
\end{example}

\begin{convention}\label{convention:highest-maps-in-complexes}
  Let $\Cat$ be an infinity category or, more generally, simplicial set and $F\co\Cat\to\ChainComplexes$
  a functor (map of simplicial sets).  For any $n$-simplex $\sigma$ in
  $\Cat$ with vertices $\sigma_0,\dots,\sigma_n$, $F(\sigma)$ is an
  $n$-simplex in $\Complexes$, consisting of chain complexes
  $F(\sigma_0),\dots,F(\sigma_n)$, and maps $F(\sigma)_J$ between
  them, one for each sequence $J=(0\leq j_0<\dots<j_k\leq n)$. We will
  let $\maxmap{F}(\sigma)$ denote the degree $(n-1)$ map
  $F(\sigma_0)\to F(\sigma_n)$ associated to the maximal sequence
  $(0<1<\dots<n)$. The functor $F$ is determined
  entirely by the chain complexes $F(x)$, $x\in\Cat_0$, and these maps
  $\maxmap{F}(\sigma)$ for non-degenerate $n$-simplices $\sigma$,
  $n\geq 1$. (For degenerate simplices except $s_0x$, $\maxmap{F}$
  must vanish.)
\end{convention}
In this notation, the fact that $F$ is a functor is equivalent to
\begin{equation}
  \label{eq:chain-cx-cat-max}
  \bdy_{C_n}\circ \maxmap{F}(\sigma)+(-1)^n\maxmap{F}(\sigma)\circ\bdy_{C_0}=
  \sum_{0<\ell<n}(-1)^{n-\ell+1}\left(\maxmap{F}(d_\ell\sigma)
    -\maxmap{F}(\sigma_{\ell,\dots,n})\circ\maxmap{F}(\sigma_{0,\dots,\ell})\right).
\end{equation}

\subsection{The homotopy coherent nerve}\label{sec:hc-nerve}
Cordier's \emph{homotopy coherent nerve} or \emph{simplicial nerve} is
a procedure for turning categories where the morphisms form (fibrant) simplicial
sets, i.e., \emph{simplicially enriched categories}, into simplicial
sets
(infinity categories). Briefly, for each non-negative integer $n$ let
$S[n]$ be the simplicially enriched category in which:
\begin{itemize}
\item The objects are $0,1,\dots,n$.
\item $\Hom_{S[n]}(i,j)$ is the order complex associated to the set of
  subsets of $\{i,i+1,\dots,j\}$ which include both $i$ and $j$,
  ordered by inclusion.
  (So,
  $\Hom_{S[n]}(i,j)$ is a simplicial set, and is
  combinatorially equivalent to
  the $(j-i-1)$-dimensional cube; the face $x_\ell=0$ (respectively
  $x_\ell=1$) of the cube corresponds to the subsets of $[i,j]$ not
  containing (respectively containing) the element $i+\ell$.)
\item Composition $\Hom_{S[n]}(j,k)\times \Hom_{S[n]}(i,j)\to
  \Hom_{S[n]}(i,k)$ is induced by the union of sets.
\end{itemize}
(We have followed the notation in \emph{nLab}; Lurie denotes
$S[n]$ by $\mathfrak{C}[\Delta^{[0,n]}]$~\cite[Definition
1.1.5.1]{Lurie09:topos}.)

Vertices of $\Hom_{S[n]}(i,j)$ correspond to subsets of
the form $V=\{i<k_1<k_2<\dots<k_\ell<j\}$; we will denote the
corresponding vertex by $A_V$. Then a non-degenerate $k$-simplex in
$\Hom_{S[n]}(i,j)$ corresponds to a sequence $V_0,\dots,V_k$ of such
subsets where ${V_{\ell}}$ is a proper subset of ${V_{\ell+1}}$,
$0\leq\ell<k$; we will represent the simplex as the sequence
$A_{V_0},\dots,A_{V_k}$ of its vertices.

There are coface and codegeneracy functors $\delta_j\co S[n]\to S[n+1]$,
$\sigma_j\co S[n]\to S[n-1]$, where $\delta_j$ is induced by the
monotone injection $\{0,\dots,n\}\to\{0,\dots,n+1\}$ which
skips $j$ and $\sigma_j$ is induced by the monotone surjection
$\{0,\dots,n\}\to \{0,\dots,n-1\}$ which sends both $j$ and $j+1$ to
$j$ (i.e., by the usual maps which dualize to $d_j$ and $s_j$); these
maps induce maps of power sets and order complexes of power sets.

Given a simplicially enriched category $\Cat$, the simplicial nerve
$\Nerve\Cat$ of $\Cat$ is the simplicial set with:
\begin{itemize}
\item $n$-simplices given by $\Fun(S[n],\Cat)$, the set of functors of
  simplicially enriched categories from $S[n]$ to $\Cat$.
\item Face and degeneracy maps induced by the $\delta_j$ and $\sigma_j$.
\end{itemize}

If each of the $\Hom$ sets in $\Cat$ is a Kan complex
then the simplicial nerve $\Nerve\Cat$ is an infinity category (weak
Kan complex)~\cite[Proposition 1.1.5.10]{Lurie09:topos}.

\begin{example}\label{eg:simp-nerve}
  A $3$-simplex of the simplicial nerve consists of four objects
$x_0,x_1,x_2,x_3$ with morphisms 
\[
\begin{tikzpicture}[xscale=2,yscale=2]
\tikzset{every
  node/.style={fill=white,inner sep=0pt,outer sep=1pt, fill
    opacity=0.7,text opacity=1}}

\node (a0) at (0,0) {$x_0$};
\node (a1) at (1,1) {$x_1$};
\node (a2) at (3,1) {$x_2$};
\node (a3) at (4,0) {$x_3$};

\draw[->] (a0) -- (a1) node[pos=0.5] {\tiny $\varphi_{01}$};
\draw[->] (a1) -- (a2) node[pos=0.5] {\tiny $\varphi_{12}$};
\draw[->] (a2) -- (a3) node[pos=0.5] {\tiny $\varphi_{23}$};
\draw[->] (a0) -- (a2) node[pos=0.5] {\tiny $\varphi_{02}$};
\draw[->] (a1) -- (a3) node[pos=0.5] {\tiny $\varphi_{13}$};
\draw[->] (a0) -- (a3) node[pos=0.5] {\tiny $\varphi_{03}$};

\draw[->,red,dashed] (a0) .. controls (0.6,0.6) and (2,1) .. (a2) node[pos=0.5] {\tiny $\varphi_{02,012}$};
\draw[->,red,dashed] (a1) .. controls (2,1) and (4-0.6,0.6) .. (a3) node[pos=0.5] {\tiny $\varphi_{13,123}$};
\draw[->,red,dashed] (a0) .. controls (0.9,0.5) .. (a3) node[pos=0.5] {\tiny $\varphi_{03,013}$};
\draw[->,red,dashed] (a0) .. controls (4-0.9,0.5) .. (a3) node[pos=0.5] {\tiny $\varphi_{03,023}$};

\draw[->,red,dashed] (a0) .. controls (2,0.2) .. (a3) node[pos=0.5] {\tiny
  $\varphi_{03,0123}$};

\draw[->,green!50!black,dotted] (a0) .. controls (3.5,-0.5) .. (a3)
node[pos=0.5] {\tiny $\varphi_{03,023,0123}$};

\draw[->,green!50!black,dotted] (a0) .. controls (0.5,-0.5) .. (a3)
node[pos=0.5] {\tiny $\varphi_{03,013,0123}$};
\end{tikzpicture}
\]
where the solid morphisms are $0$-simplices in $\Hom_{\Cat}(x_i,x_j)$, the {\color{red}dashed}
morphisms are $1$-simplices in $\Hom_{\Cat}(x_i,x_j)$ whose faces are those $0$-simplices or their
compositions, and the {\color{green!50!black}dotted} morphisms are
$2$-simplices in $\Hom_{\Cat}(x_i,x_j)$, whose faces are either the $1$-simplices or their compositions
with $0$-simplices (after treating the $0$-simplices as $1$-simplices by the
degeneracy maps). For instance, the three faces of
${\color{green!50!black} \varphi_{03,023,0123}}$ are ${\color{red}
  \varphi_{03,0123}}$, ${\color{red} \varphi_{03,023}}$, and
$s_0(\varphi_{23})\circ{\color{red}\varphi_{02,012}}$.
The morphisms
from $x_0$ to $x_3$ fit into the following triangulation of the
square:
\[
\begin{tikzpicture}[cm={0,1,1,0,(0,0)},scale=2.5]
\tikzset{every
  node/.style={fill=white,inner sep=0pt,outer sep=1pt, fill
    opacity=0.7,text opacity=1}}

\draw (0,0) rectangle (2,2);
\draw (0,0) -- (2,2);

\node at (0,0) {\tiny $\varphi_{03}$};
\node at (0,2) {\tiny $\varphi_{13}\circ\varphi_{01}$};
\node at (2,0) {\tiny $\varphi_{23}\circ\varphi_{02}$};
\node at (2,2) {\tiny $\varphi_{23}\circ\varphi_{12}\circ\varphi_{01}$};

\node at (1,0) {\tiny ${\color{red}\varphi_{03,023}}$};
\node at (0,1) {\tiny ${\color{red}\varphi_{03,013}}$};
\node at (2,1) {\tiny $s_0(\varphi_{23})\circ{\color{red}\varphi_{02,012}}$};
\node at (1,2) {\tiny ${\color{red}\varphi_{13,123}}\circ s_0(\varphi_{01})$};

\node at (1,1) {\tiny ${\color{red}\varphi_{03,0123}}$};

\node at (1.5,0.5) {\tiny ${\color{green!50!black}\varphi_{03,023,0123}}$};
\node at (0.5,1.5) {\tiny ${\color{green!50!black}\varphi_{03,013,0123}}$};
\end{tikzpicture}
\]
\end{example}

The similarity with the definition of the infinity category of chain
complexes is not a coincidence~\cite[Section
1.3.1]{Lurie:HigherAlgebra}.

\subsection{Classifying spaces and smooth nerves}\label{sec:smooth-nerves}
The degeneracy maps in the simplicial nerves are not smooth maps
between cubes. A variant of the construction, which makes sense for
topological categories, does only involve smooth maps:

\begin{definition}\label{def:smooth-CV-cat}
  The \emph{smooth Cordier-Vogt category} is the cosimplicial
  topological category $\sCVcat$ defined as follows. For $n\in\NN$,
  let $\sCVcat[n]$ be the topological category which has objects
  $\{0,\dots,n\}$, and
  \[
    \Hom_{\sCVcat[n]}(a,b)=\{1\}\times\left(\prod_{i=a+1}^{b-1}[0,1]\right)\times\{1\}=\{(1,t_{a+1},t_{a+2},\dots,t_{b-1},1)\mid t_i\in[0,1]\},
  \]
  with composition
  $\Hom_{\sCVcat[n]}(b,c)\times\Hom_{\sCVcat[n]}(a,b)\to
  \Hom_{\sCVcat[n]}(a,c)$ given by 
  \[
    \bigl((1,t_{b+1},\dots,t_{c-1},1),(1,t_{a+1},\dots,t_{b-1},1)\bigr)\mapsto (1,t_{a+1},\dots,t_{b-1},1,t_{b+1},\dots,t_{c-1},1).
  \]
  
  The co-face maps $\delta^i\co \sCVcat[n-1]\to\sCVcat[n]$,
  $i=0,\dots,n$, are given by
  \[
    \delta^i(k)=
    \begin{cases}
      k & k<i\\
      k+1 & k\geq i.
    \end{cases}
  \]
  and
  \begin{multline*}
    \delta^i|_{\Hom(a,b)}(1,t_{a+1},\dots,t_{b-1},1)\\
    =
    \begin{cases}
      (1,t_{a+1},\dots,t_{b-1},1)\in\Hom(a+1,b+1) & i\leq a\\
      (1,t_{a+1},\dots,t_{i-1},0,t_i,\dots,t_{b-1},1)\in\Hom(a,b+1) & a<i\leq b\\
      (1,t_{a+1},\dots,t_{b-1},1)\in\Hom(a,b) & i>b.
    \end{cases}
  \end{multline*}

  The co-degeneracy maps $\sigma^j\co \sCVcat[n+1]\to \sCVcat[n]$,
  $j=0,\dots,n$, are given by
  \[
    \sigma^j(k)=
    \begin{cases}
      k & k\leq j\\
      k-1 & k>j
    \end{cases}
  \]
  and
  \begin{multline*}
    \sigma^j|_{\Hom(a,b)}(1,t_{a+1},\dots,t_{b-1},1)\\
    =
    \begin{cases}
      (1,t_{a+1},\dots,t_{b-1},1)\in\Hom(a-1,b-1) & j<a\\
      (1,t_{a+1},\dots,t_{j-1},t_j+t_{j+1}-t_jt_{j+1},t_{j+2},\dots,t_{b-1},1)\in\Hom(a,b-1) & a\leq j<b\\
      (1,t_{a+1},\dots,t_{b-1},1)\in\Hom(a,b) & j\geq b.
    \end{cases}
  \end{multline*}
  In the second case, if $j=a$ then we interpret $t_a=1$ so
  $\sigma^a$ deletes $t_{a+1}$, while if $j=b-1$ we interpret $t_b=1$ so
  $\sigma^b$ deletes $t_{b-1}$. Note that $t_j+t_{j+1}-t_jt_{j+1}=1-(1-t_j)(1-t_{j+1})$.
\end{definition}

\begin{proposition}\label{prop:sCV-is-cosimp-cat}
  Definition~\ref{def:smooth-CV-cat} defines a cosimplicial
  topological category.
\end{proposition}
\begin{proof}
  We must check that the maps $\delta^i$ and $\sigma^j$ respect
  composition and the identity maps and that the $\delta^i$ and
  $\sigma^j$ satisfy the cosimplicial identities.

  The first statement, that the $\delta^i$ and $\sigma^j$ are
  functors, is clear.
  
  For the cosimplicial identities, consider
  $[0,1]^\NN=\{(t_0,t_1,t_2,\dots)\mid t_i\in [0,1]\}$. There are maps
  $\bar{\delta}^i\co [0,1]^\NN\to[0,1]^\NN$ and
  $\bar{\sigma}^j\co[0,1]^\NN\to[0,1]^\NN$, $i,j\in\NN$, defined by
  \begin{align*}
    \bar{\delta}^i(t_0,t_1,\dots)&=(t_0,t_1,\dots,t_{i-1},0,t_i,t_{i+1},\dots)\\
    \bar{\sigma}^j(t_0,t_1,\dots)&=(t_0,t_1,\dots,t_j+t_{j+1}-t_jt_{j+1},t_{j+2},\dots). 
  \end{align*}
  We can view the map $\delta^i$ (respectively $\sigma^j$) on
  $\Hom(a,b)$ as the restriction of $\bar{\delta}^i$ (respectively
  $\bar{\sigma}^j$) to
  $\{1\}^{a}\times[0,1]^{b-a-1}\times\{1\}\times\{1\}\times\cdots$. Thus,
  it suffices to check the simplicial identities for the
  $\bar{\delta}^i$ and $\bar{\sigma}^j$. This check is routine, using
  the fact that the function $f(x,y)=x+y-xy$ is associative (i.e.,
  $f(f(x,y),z)=f(x,f(y,z))$) and has unit $0$ (i.e.,
  $f(0,x)=f(x,0)=x$).
\end{proof}

\begin{definition}\label{def:smooth-cat}
  A \emph{smooth category} consists of a topological category $\Cat$
  together with a choice, for each $n$ and each pair of objects
  $x,y\in\ob(\Cat)$, of a subset of the maps
  $[0,1]^n\to \Hom_\Cat(x,y)$, which we will call the \emph{smooth
    $n$-cubes in $\Hom_\Cat(x,y)$}. These subsets are required to
  satisfy the following properties:
  \begin{itemize}
  \item 
    If $f\co [0,1]^n\to \Hom_\Cat(x,y)$ is a smooth $n$-cube and
    $\bar{\delta}^i\co [0,1]^{n-1}\to [0,1]^n$ (respectively
    $\bar{\sigma}^j\co [0,1]^{n+1}\to [0,1]^n$) is one of the
    extensions of the co-face (respectively co-degeneracy) maps from
    the proof of Proposition~\ref{prop:sCV-is-cosimp-cat} then $f\circ
    \bar{\delta}^i$ (respectively $f\circ \bar{\sigma}^j$) is a smooth
    cube.
  \item If $f\co[0,1]^n\to \Hom_\Cat(x,y)$ and
    $g\co[0,1]^n\to \Hom_\Cat(y,z)$ are smooth $n$-cubes then the
    pointwise composition $g\circ f$ is a smooth $n$-cube.
  \item The identity map of $x$, viewed as a $0$-cube in
    $\Hom_\Cat(x,x)$, is smooth.
  \end{itemize}
\end{definition}

The fundamental example is the following:
\begin{example}\label{eg:smooth-BG-cat}
  If $G$ is a Lie group there is a smooth category $BG$ with a single
  object $o$, $\Hom(o,o)=G$, and smooth $n$-cubes the set of smooth
  maps $[0,1]^n\to G$.
\end{example}

\begin{definition}
  Let $\Cat$ be a smooth category. A functor $F\co
  \sCVcat[n]\to \Cat$ is \emph{smooth} if for each $0\leq a\leq b\leq n$,
  $F|_{\Hom_{\sCVcat[n]}(a,b)}\co[0,1]^{b-a-1}\to
  \Hom_{\Cat}(F(a),F(b))$ is a smooth $(b-a-1)$-cube in $\Cat$.
\end{definition}

\begin{lemma}\label{lem:smooth-nerve-exists}
  If $F\co\sCVcat[n]\to\Cat$ is a smooth functor then for any
  $i=0,\dots,n$, $F\circ\delta^i\co \sCVcat[n-1]\to\Cat$ and
  $F\circ\sigma^i\co\sCVcat[n+1]\to\Cat$ are smooth functors.
\end{lemma}
\begin{proof}
  This is immediate from the definitions.
\end{proof}

\begin{definition}\label{def:smooth-nerve}
  Given a smooth category $\Cat$, let $\Nervesm\Cat$, the
  \emph{smooth simplicial nerve} of $\Cat$, be the simplicial set with
  $n$-simplices the set of smooth functors $\sCVcat[n]\to\Cat$ and
  face and degeneracy maps induced by the co-face and co-degeneracy
  maps of $\sCVcat$. In the same spirit as
  Convention~\ref{convention:highest-maps-in-complexes}, for any
  $n$-simplex $\sigma$ in $\Nervesm\Cat$, we will denote the
  top-dimensional smooth cube
  $\sigma|_{\Hom_{\sCVcat[n]}(0,n)}\co[0,1]^{n-1}\to
  \Hom_{\Cat}(\sigma(0),\sigma(n))$ by $\cubemap{\sigma}$.
\end{definition}
The fact that $\Nervesm\Cat$ is a well-defined simplicial set is immediate from
Lemma~\ref{lem:smooth-nerve-exists}.

\begin{remark}\label{rem:max-vs-smooth-max}
  In Cordier's simplicial nerve~\cite{Cordier82:coherent-nerve}, the $\max\{t_j,t_{j+1}\}$
  is used in place of the map $t_j+t_{j+1}-t_jt_{j+1}$, which means
  that the degeneracy maps in the simplicial nerve do not respect
  smoothness. The map $t_j+t_{j+1}-t_jt_{j+1}$ is closely related to
  Vogt's definition of homotopy coherent
  diagrams~\cite{Vogt73:hocolim} (which exchanges $0$ and $1$ and uses
  the map $t_jt_{j+1}$); the relation between the two constructions
  was further studied by Cordier~\cite{Cordier82:coherent-nerve}.
\end{remark}

\begin{lemma}\label{lem:sm-nerve-is-nerve-semi}
  Let $\Cat_\bullet$ be the simplicially enriched category obtained by
  replacing each morphism space in a smooth category $\Cat$ by its
  singular set.  Then there is an inclusion map
  $\iota\from\Nervesm\Cat\to\Nerve\Cat_\bullet$ of semi-simplicial
  sets. If for every $x$ and $y$, every continuous map
  $[0,1]^n\to\Hom_{\Cat}(x,y)$ is declared to be smooth, then $\iota$
  is an isomorphism of semi-simplicial sets.
\end{lemma}
\begin{proof}
  This is immediate from the definitions. Indeed, in the second statement,
  if $\max\{t_j,t_{j+1}\}$ were used in place of
  $t_j+t_{j+1}-t_jt_{j+1}$ in the definition of the smooth nerve,
  cf.~Remark~\ref{rem:max-vs-smooth-max}, then $\iota$ would have been
  an isomorphism of simplicial sets.
\end{proof}

Now we will apply this construction of the smooth nerve to study
classifying spaces.  Let $G$ be a Lie group, and let $\SingG$ denote
the singular set of $G$, which is a simplicial group.  The following
is a slightly unusual model for the classifying space of $G$:
\begin{lemma}\label{lem:Hinich}\cite[Proposition 2.6.2]{Hinich:nerve}
  Let $B\SingG$ denote the simplicially enriched category with a
  single object $o$ and $\Hom(o,o)=\SingG$. The composition in
  $B\SingG$ is given by $g\circ h=hg$.  Let $\Nerve B\SingG$ be the
  simplicial nerve of $B\SingG$ and $|\Nerve B\SingG|$ the geometric
  realization of $\Nerve B\SingG$. Then $|\Nerve B\SingG|$ is weakly
  homotopy equivalent to the classifying space of $G$.
\end{lemma}

If $G$ is a Lie group, we can also consider the smooth nerve
$\Nervesm BG$ of the category $BG$ from Example~\ref{eg:smooth-BG-cat}.

\begin{example}\label{exam:3-cell-in-nbg}
  Let $\sigma$ be a $3$-simplex in the smooth nerve $\Nervesm BG$,
  and let $\iota(\sigma)$ be the corresponding $3$-simplex in the
  simplicial nerve $\Nerve B\SingG$. Then the data for $\sigma$ and
  $\iota(\sigma)$ consist of
\begin{itemize}[leftmargin=*]
\item six elements of $G$, $\cubemap{\sigma}_{01}=g_{01}$,
  $\cubemap{\sigma}_{12}=g_{12}$, $\cubemap{\sigma}_{23}=g_{23}$,
  $\cubemap{\sigma}_{02}=g_{02}$, $\cubemap{\sigma}_{13}=g_{13}$,
  $\cubemap{\sigma}_{03}=g_{03}$,
\item five paths in $G$,
  $\cubemap{\sigma}_{012}={\color{red}g_{02,012}}$,
  $\cubemap{\sigma}_{123}={\color{red}g_{13,123}}$,
  $\cubemap{\sigma}_{013}={\color{red}g_{03,013}}$,
  $\cubemap{\sigma}_{023}={\color{red}g_{03,023}}$, and
  ${\color{red}g_{03,0123}}$, and
\item two singular $2$-simplices in $G$,
  ${\color{green!50!black}g_{03,023,0123}}$,
  ${\color{green!50!black}g_{03,013,0123}}$,
  fitting into the square $\cubemap{\sigma}$,
\[
\begin{tikzpicture}[cm={0,1,1,0,(0,0)},scale=2.5]
\tikzset{every
  node/.style={fill=white,inner sep=0pt,outer sep=1pt, fill
    opacity=0.7,text opacity=1}}

\begin{scope}[yshift=-1.5cm]
\draw (0,0) rectangle (2,2);

\node at (0,0) {\tiny $\cubemap{\sigma}_{03}$};
\node at (0,2) {\tiny $\cubemap{\sigma}_{01}\cubemap{\sigma}_{13}$};
\node at (2,0) {\tiny $\cubemap{\sigma}_{02}\cubemap{\sigma}_{23}$};
\node at (2,2) {\tiny $\cubemap{\sigma}_{01}\cubemap{\sigma}_{12}\cubemap{\sigma}_{23}$};

\node at (1,0) {\tiny ${\cubemap{\sigma}_{023}}$};
\node at (0,1) {\tiny ${\cubemap{\sigma}_{013}}$};
\node at (2,1) {\tiny ${\cubemap{\sigma}_{012}}\cubemap{\sigma}_{23}$};
\node at (1,2) {\tiny ${\cubemap{\sigma}_{01}\cubemap{\sigma}_{123}}$};

\node at (1,1) {\tiny $\cubemap{\sigma}$};
\end{scope}

\node at (1,1) {$=$};

\begin{scope}[yshift=1.5cm]
\draw (0,0) rectangle (2,2);
\draw (0,0) -- (2,2);

\node at (0,0) {\tiny $g_{03}$};
\node at (0,2) {\tiny $g_{01}g_{13}$};
\node at (2,0) {\tiny $g_{02}g_{23}$};
\node at (2,2) {\tiny $g_{01}g_{12}g_{23}$};

\node at (1,0) {\tiny ${\color{red}g_{03,023}}$};
\node at (0,1) {\tiny ${\color{red}g_{03,013}}$};
\node at (2,1) {\tiny ${\color{red}g_{02,012}}g_{23}$};
\node at (1,2) {\tiny ${g_{01}\color{red}g_{13,123}}$};

\node at (1,1) {\tiny ${\color{red}g_{03,0123}}$};

\node at (1.5,0.5) {\tiny ${\color{green!50!black}g_{03,023,0123}}$};
\node at (0.5,1.5) {\tiny ${\color{green!50!black}g_{03,013,0123}}$};
\end{scope}
\end{tikzpicture}
\]
\end{itemize}
While the right picture is a special case of
Example~\ref{eg:simp-nerve}, notice that the left picture is similar to
Example~\ref{ref:3-cell-in-kom}.
\end{example}

\begin{lemma}\label{lem:sm-nerve}
  The smooth nerve $\Nervesm BG$ is weakly equivalent to the ordinary
  nerve $\Nerve B\SingG$. So, $\Realize{\Nervesm BG}$ is a model for
  the classifying space of $G$.
\end{lemma} 

\begin{proof}
  Let $\iota\co \Nervesm BG\to \Nerve B\SingG$ be the inclusion map
  from Lemma~\ref{lem:sm-nerve-is-nerve-semi} (of semi-simplicial
  sets).  By Lemma~\ref{lem:fat}, it suffices to construct
  semi-simplicial maps $f\from \Nerve B\SingG\to\Nervesm BG$ and
  $H\from{\Nerve B\SingG\times\Delta^1}\to\Nerve B\SingG$ so that
  $f\circ\iota=\Id_{\Nervesm BG}$ and $H$ restricts to $\Id_{\Nerve
    B\SingG}$ and $\iota\circ f$ on $\Nerve B\SingG\times\{0\}$ and
  $\Nerve B\SingG\times\{1\}$, respectively. (To reduce notational
  clutter, we are suppressing $j^*$.)
  
  Fix, once and for all, a deformation retraction $D\co U\times [0,1]\to U$ of
  some open neighborhood $U$ of $1\in G$ to $1$. Given a topological space $X$,
  call maps $p,q\co X\to G$ \emph{close} if for each $x\in X$,
  $p(x)^{-1}q(x)\in U$. If $p$ and $q$ are close then there is a canonical
  homotopy from $q$ to $p$ given by $(x,t)\mapsto p(x)D(p(x)^{-1}q(x),t)$.

  If $\sigma$ is an $n$-simplex in $\Nerve B\SingG$, it specifies an
  $(n-1)$-dimensional singular cube in $G$, which we also denote
  $\cubemap{\sigma}$. For $\sigma$ an $n$-simplex in either
  $\Nerve B\SingG$ or $\Nervesm BG$, the restriction of
  $\cubemap{\sigma}$ to the boundary is given by
  \begin{equation}\label{eq:bdy-of-cubesym}
    \begin{split}
      \cubemap{\sigma}(t_1,\dots,t_{i-1},0,t_{i+1},\dots,t_{n-1})&=\cubemap{(d_i\sigma)}(t_1,\dots,t_{i-1},t_{i+1},\dots,t_{n-1})\\
      \cubemap{\sigma}(t_1,\dots,t_{i-1},1,t_{i+1},\dots,t_{n-1})&=\cubemap{\sigma}_{0,\dots,i}(t_1,\dots,t_{i-1})\cubemap{\sigma}_{i,\dots,n}(t_{i+1},\dots,t_{n-1}).
    \end{split}
  \end{equation}

  We construct $f$ inductively on the dimension of the simplices,
  ensuring the following:
  \begin{itemize}
  \item The smooth singular cube $\cubemap{f(\sigma)}$ is close
    to the singular cube $\cubemap{\sigma}$.
  \item If $\sigma=\iota(\sigma')$, then $f(\sigma)=\sigma'$.
  \end{itemize}
  Fix some $n$-simplex $\sigma$ in $\Nerve B\SingG$. By induction, $f$
  has already been constructed on all lower dimensional simplices, and
  hence, by Equation~\eqref{eq:bdy-of-cubesym}, the singular cube
  $\cubemap{f(\sigma)}$ has already been constructed on
  $\bdy[0,1]^{n-1}$ (and is close to
  $\cubemap{\sigma}|_{\bdy[0,1]^{n-1}}$, and equals
  $\cubemap{\sigma}|_{\bdy[0,1]^{n-1}}$ if $\sigma$ is in the image of
  $\iota$). Moreover, the restriction of $\cubemap{f(\sigma)}$ to each
  facet is a smooth singular $(n-2)$-dimensional cube in $G$, and
  hence~\cite[Lemma 18.9]{Lee13:intro},
  $\cubemap{f(\sigma)}|_{\bdy[0,1]^{n-1}}$ is a smooth map from this closed
  subset of $\RR^{n-1}$ to $G$. By definition, this gives a smooth
  extension of $\cubemap{f(\sigma)}|_{\bdy[0,1]^{n-1}}$ to a neighborhood of
  $\bdy[0,1]^{n-1}$.  If $\sigma$ is in the image of $\iota$ (and
  hence $\cubemap{\sigma}$ is smooth), choose this extension to be
  $\cubemap{\sigma}$; otherwise, fix any smooth extension.  Next,
  approximate $\cubemap{\sigma}$ on the interior of $[0,1]^{n-1}$ by a
  smooth family; once again, if $\cubemap{\sigma}$ is smooth, choose
  the approximation to be $\cubemap{\sigma}$. Patch these two
  approximations together near the boundary to define
  $\cubemap{f(\sigma)}$.
  
  We now construct $H$ inductively on the dimension of the
  simplices.  Consider any $n$-simplex $\sigma=(a,b)$, where $a$ is an
  $n$-simplex in $\Nerve B\SingG$ and $b$ is an $n$-simplex in
  $\Delta^1$. The construction of $H$ will inductively satisfy the
  condition that the singular cubes $\cubemap{H(\sigma)}$ and
  $\cubemap{a}$ are close.

  Since $H$ has already been defined on $\Nerve B\SingG\times\{0,1\}$
  (and satisfies the closeness condition), it only remains  to define
  $H$ on simplices that do not lie in $\Nerve B\SingG\times\{0,1\}$.
  Therefore, we may assume $b=\Delta^1_{n;i}$ for some $0\leq i<n$
  (using notation from Section~\ref{sec:infty-cat-functors}). For the
  induction base case, if $n=1$, map $(a,b)$ to $a=\iota\circ f(a)$.

  For the induction step, $H$ has already been defined on every face
  of $\sigma$; therefore, by Equation~\eqref{eq:bdy-of-cubesym},
  $\cubemap{H(\sigma)}$ has already been constructed on
  $\bdy[0,1]^{n-1}$. So we only need to extend $\cubemap{H(\sigma)}$ to the
  interior. Inductively, $\cubemap{H(\sigma)}|_{\bdy[0,1]^{n-1}}$ is
  close to $\cubemap{a}|_{\bdy[0,1]^{n-1}}$. So, define the
  extension by setting it to be (a scaled version of) $a$ in a
  slightly smaller sub-cube, and the canonical homotopy from
  $\cubemap{H(\sigma)}|_{\bdy[0,1]^{n-1}}$ to
  $\cubemap{a}|_{\bdy[0,1]^{n-1}}$ on the remaining annular
  neighborhood of the boundary.
\end{proof}

\begin{example}
  Consider the 3-simplex $\sigma=(a,\Delta^1_{3;0})$, where $a$ is a
  3-simplex in $\Nerve B\SingG$, given by points
  $g_{01},g_{12},g_{23},g_{02},g_{13},g_{03}$, paths
  ${\color{red}g_{02,012}}$, ${\color{red}g_{13,123}}$,
  ${\color{red}g_{03,013}}$, ${\color{red}g_{03,023}}$,
  ${\color{red}g_{03,0123}}$ and singular triangles
  ${\color{green!50!black}g_{03,013,0123}}$,
  ${\color{green!50!black}g_{03,023,0123}}$ in $G$, as in
  Example~\ref{exam:3-cell-in-nbg}. We need to define $H(\sigma)$ as a
  3-simplex in $\Nerve B\SingG$. Say the data of $H(\sigma)$ is given
  by points $h_{01},h_{12},h_{23},h_{02},h_{13},h_{03}$, paths
  ${\color{red}h_{02,012}}$, ${\color{red}h_{13,123}}$,
  ${\color{red}h_{03,013}}$, ${\color{red}h_{03,023}}$,
  ${\color{red}h_{03,0123}}$ and singular triangles
  ${\color{green!50!black}h_{03,013,0123}}$,
  ${\color{green!50!black}h_{03,023,0123}}$ in $G$. Inductively, the
  following things are defined as follows:
\begin{enumerate}
\item $h_{01}=g_{01},h_{12}=g_{12},h_{23}=g_{23},h_{02}=g_{02},h_{13}=g_{13},h_{03}=g_{03}$.
\item
  ${\color{red}h_{02,012}}=s_0(g_{02})*{\color{red}g_{02,012}}*s_0(g_{01}g_{12}),{\color{red}h_{13,123}}=\iota\circ f({\color{red}g_{13,123}}),{\color{red}h_{03,013}}=s_0(g_{03})*{\color{red}g_{03,013}}*s_0(g_{01}g_{13}),{\color{red}h_{03,023}}=s_0(g_{03})*{\color{red}g_{03,023}}*s_0(g_{02}g_{23})$. (Here
  $*$ denotes concatenation of three paths, with the middle path being
  the longer one, and two end paths being short constant ones.)
\end{enumerate}

So we only need to define ${\color{red}h_{03,0123}}$,
${\color{green!50!black}h_{03,013,0123}}$, and
${\color{green!50!black}h_{03,023,0123}}$, that is, the interior of
the square $\cubemap{H(\sigma)}$. Define it as follows.

\[
\begin{tikzpicture}[cm={0,1,1,0,(0,0)},scale=2.5]
\tikzset{every
  node/.style={fill=white,inner sep=0pt,outer sep=1pt, fill
    opacity=0.7,text opacity=1}}

\draw (0,0) rectangle (2,2);

\foreach \i/\j in {0/0,2/0,0/2,2/2,0.3/0,1.7/0,0/0.3,0/1.7,2/0.3,2/1.7}{
\node at (\i,\j) {$\bullet$};
}

\begin{scope}[xshift=1cm,yshift=1cm]
\clip (-1,-1) rectangle (1,1);
\foreach \i in {15,30,...,360}{
\draw[dotted] (0,0)--(\i:2cm);}
\end{scope}

\draw[thick,fill=white] (0.2,0.2) rectangle (1.8,1.8);
\draw[thick] (0.2,0.2) -- (1.8,1.8);

\node[anchor=north east] at (0,0) {\tiny $g_{03}$};
\node[anchor=north west] at (0,2) {\tiny $g_{01}g_{13}$};
\node[anchor=south east] at (2,0) {\tiny $g_{02}g_{23}$};
\node[anchor=south west] at (2,2) {\tiny $g_{01}g_{12}g_{23}$};

\node[rotate=90,anchor=south] at (1,0) {\tiny ${\color{red}g_{03,023}}$};
\node[anchor=north] at (0,1) {\tiny ${\color{red}g_{03,013}}$};
\node[anchor=south] at (2,1) {\tiny ${\color{red}g_{02,012}}g_{23}$};
\node[rotate=90,anchor=north] at (1,2) {\tiny ${g_{01}(\iota\circ f(\color{red}g_{13,123}}))$};

\draw[<-] (0.15,0)--++(0,-0.2) node[pos=1,rotate=90,anchor=south] {\tiny $g_{03}$};
\draw[<-] (2-0.15,0)--++(-0.2,-0.2) node[pos=1,rotate=90,anchor=south] {\tiny $g_{02}g_{23}$};
\draw[<-] (0,0.15)--++(-0.2,0) node[pos=1,anchor=north] {\tiny $g_{03}$};
\draw[<-] (0,2-0.15)--++(-0.2,0) node[pos=1,anchor=north] {\tiny $g_{01}g_{13}$};
\draw[<-] (2,0.15)--++(0.2,0) node[pos=1,anchor=south] {\tiny $g_{02}g_{23}$};
\draw[<-] (2,2-0.15)--++(0.2,0) node[pos=1,anchor=south] {\tiny $g_{01}g_{12}g_{23}$};

\node at (0.2,0.2) {\tiny $g_{03}$};
\node at (0.2,1.8) {\tiny $g_{01}g_{13}$};
\node at (1.8,0.2) {\tiny $g_{02}g_{23}$};
\node at (1.8,1.8) {\tiny $g_{01}g_{12}g_{23}$};

\node[rotate=90] at (1,0.2) {\tiny ${\color{red}g_{03,023}}$};
\node at (0.2,1) {\tiny ${\color{red}g_{03,013}}$};
\node at (1.8,1) {\tiny ${\color{red}g_{02,012}}g_{23}$};
\node[rotate=90] at (1,1.8) {\tiny $g_{01}{\color{red}g_{13,123}}$};

\node at (1,1) {\tiny ${\color{red}g_{03,0123}}$};

\node at (1.4,0.6) {\tiny ${\color{green!50!black}g_{03,023,0123}}$};
\node at (0.6,1.4) {\tiny ${\color{green!50!black}g_{03,013,0123}}$};
\end{tikzpicture}
\]
\end{example}

\begin{lemma}\label{lem:sm-nerve-inf}
  The smooth nerve $\Nervesm BG$ is an infinity category, that
  is, satisfies the weak Kan condition.
\end{lemma}

\begin{proof}
  This proof is similar to the proof of Lemma~\ref{lem:sm-nerve}. By
  Equation~\eqref{eq:bdy-of-cubesym}, any map $\Lambda^n_i\to\Nervesm BG$
  assembles to give a map
  $f\co \bdy[0,1]^{n-1}\setminus\bdy_{i,0}[0,1]^{n-1}\to G$. The map $f$
  restricts on each facet to a smooth map, and therefore forms a smooth
  $\bdy[0,1]^{n-1}\setminus\bdy_{i,0}[0,1]^{n-1}$-family~\cite[Lemma
  18.9]{Lee13:intro}. So, $f$ extends to
  some smooth family $\widetilde{g}$ on some $\epsilon$-neighborhood $U$.  To
  fill the horn, we only need to extend $\widetilde{g}$ to the entire cube $[0,1]^{n-1}$
  smoothly. This is easy: choose a smooth map $\phi\co [0,1]^{n-1}\to U$ which
  is the identity on $\bdy[0,1]^{n-1}\setminus\bdy_{i,0}[0,1]^{n-1}$ and define
  the extension of $g$ to be $\wt{g}\circ \phi$.
\end{proof}

Next, suppose we have two maps $F_0,F_1\co \Nervesm BG\to S$.  A
\emph{natural transformation} from $F_0$ to $F_1$ is a functor $H\co
\Delta^1\times \Nervesm BG\to S$ which restricts to $F_0$ on
$\{0\}\times \Nervesm BG$ and $F_1$ on $\{1\}\times \Nervesm BG$.
As we see below, this has an interpretation in terms of smooth nerves as well.

Let $\Cat$ be any smooth category (such as $BG$), and let $\ICat$ be
the discrete category associated to the ordered set $0<1$, with
objects $\{0,1\}$ and
\[
  \Hom(0,0)=\Hom(0,1)=\Hom(1,1)=\{\pt\}\qquad\qquad
  \Hom(1,0)=\emptyset,
\]
viewed as a topological category.
Treat $\ICat\times\Cat$ as a smooth category by declaring a cube
$f\from[0,1]^n\to\Hom_{\ICat\times\Cat}(x,y)$ to be smooth if and only
if $\pi_\Cat\circ
f\from[0,1]^n\to\Hom_{\Cat}(\pi_\Cat(x),\pi_\Cat(y))$ is smooth, where
$\pi_\Cat\from\ICat\times\Cat\to\Cat$ is projection to the second
factor.  It is clear that
$\Nervesm(\ICat\times\Cat)=\Delta^1\times\Nervesm\Cat$, since an
$n$-simplex of either side consists of an $n$-simplex of $\Delta^1$
(which is a length-$(n+1)$ nondecreasing sequence in $\{0,1\}$) and a
smooth functor $\sCVcat[n]\to \Cat$.
Therefore, the following is an equivalent definition of a natural transformation.

\begin{definition}\label{def:sm-htpy}
  Given functors $F_0,F_1\co \Nervesm BG\to S$, a
  \emph{(smooth) natural transformation} from $F_0$ to $F_1$ is a functor
  $H\co \Nervesm(\ICat\times BG)\to S$ whose restriction
  to $\{0\}\times \Nervesm BG$ (respectively
  $\{1\}\times \Nervesm BG$) is $F_0$ (respectively $F_1$).
\end{definition}

In light of the above lemmas, in the rest of the paper we will adopt the following convention:
\begin{convention}
  We will henceforth use $\Nerve BG$ to denote the smooth nerve
  $\Nervesm BG$ (not $\Nerve B\SingG$), and often drop the word
  ``smooth''.
\end{convention}

\subsection{Homotopy colimits}\label{sec:homotopy-colimits}
Let $I_*$ denote the simplicial chain complex of the $1$-simplex
$\Delta^1$, i.e., $I_0=\Ring\langle \{0\}, \{1\}\rangle$,
$I_1=\Ring\langle \{0,1\}\rangle$, and $\bdy\{0,1\}=\{1\}-\{0\}$. The
inclusions $\{0\}\hookrightarrow [0,1]$ and $\{1\}\hookrightarrow
[0,1]$ induce chain maps $\iota_0,\iota_1\co \Ring\hookrightarrow I_*$
and the projection $[0,1]\to \{\pt\}$ induces a chain map $\pi\co
I_*\to\Ring$.  The multiplication map $[0,1]^2\to [0,1]$,
$(x,y)\mapsto xy$, induces the ``multiplication'' $m\co
I_*\otimes I_*\to I_*$ given by
\begin{align*}
&m(\{0\}\otimes \{0\})=m(\{0\}\otimes\{1\})=m(\{1\}\mathrlap{\otimes\{0\})=\{0\}} & m(\{1\}\otimes \{1\})=\{1\}& \\
&m(\{0,1\}\otimes \{0\})=m(\{0\}\otimes \{0,1\})=0 & 
{m(\{0,1\} \otimes \{1\}) = m(\{1\}\otimes \{0,1\})=\{0,1\}}& \\
&m(\{0,1\}\otimes\{0,1\})=0.
\end{align*}

The following is an adaptation of Vogt's construction of homotopy
colimits of diagrams in topological categories~\cite{Vogt73:hocolim}.
\begin{definition} \label{def:hocolim}
 Given a map of simplicial sets $F$ from a
  simplicial set $\Cat$ to $\ChainComplexes$, the \emph{homotopy colimit} of $F$
  is the chain complex
  \begin{equation}\label{eq:chain-hocolim-def}
    \hocolim F = \bigoplus_{n\geq 0}
    \bigoplus_{\sigma\in \Cat_n}  I_*^{\otimes n}\otimes F(\sigma_0)/\sim,
  \end{equation}
  where the equivalence relation $\sim$ is generated by
  \begin{align*}
    (s_0\sigma;t_{n+1}\otimes\dots\otimes t_1\otimes
    x)&\sim(\sigma;t_{n+1}\otimes\dots\otimes t_2\otimes\pi(t_1)\otimes x)\\
    (s_i\sigma;t_{n+1}\otimes\dots\otimes t_1\otimes x)&\sim
    (\sigma;t_{n+1}\otimes\dots\otimes m(t_{i+1},t_{i})\otimes\dots
    \otimes t_1\otimes x)&\text{if }i\geq 1\\
    (\sigma;t_n\otimes\dots\otimes t_1\otimes x)&\sim(d_i\sigma;t_n\otimes\dots\otimes \widehat{t_i}\otimes\dots\otimes t_1\otimes x)&\text{if }t_i=\{1\}\\
    (\sigma;t_n\otimes\dots\otimes t_1\otimes
    x)&\sim(\sigma_{i,\dots,n};t_n\otimes\cdots\otimes t_{i+1}\otimes
    \maxmap{F}(\sigma_{0,\dots,i})(x))\\
    &&\mathllap{\text{if }t_1=\dots=t_{i-1}=\{0,1\}\text{ and }t_i=\{0\}.}
  \end{align*}
  The differential is induced by the tensor product differential in
  Formula~\eqref{eq:chain-hocolim-def}, namely:
  \begin{align*}
  \bdy(\sigma;t_n\otimes\dots\otimes t_1\otimes
  x)&=\sum_{i=1}^n(-1)^{\gr(t_{i+1})+\dots+\gr(t_n)}(\sigma;t_n\otimes\dots\otimes\bdy(t_i)\otimes\dots\otimes
  t_1\otimes x)\\
  &\quad\qquad{}+ (-1)^{\gr(t_1)+\dots+\gr(t_{n})}(\sigma;t_n\otimes\dots\otimes
  t_1\otimes \bdy(x)).
  \end{align*}
\end{definition}

The homotopy colimit has a filtration $\Filt$ where
$\Filt_m\hocolim G$ corresponds to replacing the first direct sum over
$n\geq0$ in Equation~\eqref{eq:chain-hocolim-def} with the sum
$0\leq n\leq m$.  (The special case $\Filt_0\hocolim G$ is simply the
direct sum over objects $v$ in $\Cat$ of $F(v)$.)

It is not hard to verify that the differential is consistent with the equivalence
relation. For instance, temporarily using the notation
$\messygr{a}{b}$ to denote $\sum_{j=a}^b\gr(t_j)$, we have
\[
\begin{tikzpicture}
  \node[draw] (main) at (0,0) {\tiny $(\sigma; t_{n}\otimes\dots\otimes
    t_{i+1}\otimes\{0\}\otimes\{0,1\}\otimes\dots\otimes\{0,1\}\otimes
    x)$};

  \node[draw,align=left] (second) [right=4ex of main.east,anchor=west] {\tiny $(\sigma_{i,\dots,n};t_n\otimes\cdots\otimes t_{i+1}\otimes
    \maxmap{F}(\sigma_{0,\dots,i})(x))$};

  \node[draw,align=left] (diffsecond) [below left =4ex and 6ex of
  second.south west,anchor=north west] {\tiny $\sum_{j=i+1}^n(-1)^{\messygr{j+1}{n}}(\sigma_{i,\dots,n};t_{n}\otimes\dots\otimes\bdy(t_j)\otimes\dots\otimes
  t_{i+1}\otimes \maxmap{F}(\sigma_{0,\dots,i})(x))$\\
  \tiny $+(-1)^{\messygr{i+1}{n}}(\sigma_{i,\dots,n};t_{n}\otimes\dots\otimes
  t_{i+1}\otimes\bdy\circ\maxmap{F}(\sigma_{0,\dots,i})(x))$};

\node[draw, align=left] (equalsecond) [below=4ex of
  diffsecond.south west,anchor=north west]  {\tiny $\sum_{j=i+1}^n(-1)^{\messygr{j+1}{n}}(\sigma_{i,\dots,n};t_{n}\otimes\dots\otimes\bdy(t_j)\otimes\dots\otimes
  t_{i+1}\otimes \maxmap{F}(\sigma_{0,\dots,i})(x))$\\
  \tiny $+\sum_{j=1}^{i-1}(-1)^{\messygr{i+1}{n}+i-j+1}(\sigma_{i,\dots,n};t_{n}\otimes\dots\otimes t_{i+1}\otimes \maxmap{F}(\sigma_{0,\dots,\hat{j},\dots,i})(x))$\\
  \tiny $+\sum_{j=1}^{i-1}(-1)^{\messygr{i+1}{n}+i-j}(\sigma_{i,\dots,n};t_{n}\otimes\dots\otimes
  t_{i+1}\otimes
  \maxmap{F}(\sigma_{j,\dots,i})\circ\maxmap{F}(\sigma_{0,\dots,j})(x))$\\
  \tiny $+(-1)^{\messygr{i+1}{n}+i-1}(\sigma_{i,\dots,n};t_{n}\otimes\dots\otimes
  t_{i+1}\otimes\maxmap{F}(\sigma_{0,\dots,i})\circ\bdy(x))$};

\node[draw, align=left] (diffmain) [below=2ex of main.west|-equalsecond.south,
anchor=north west] {\tiny $\sum_{j=i+1}^n(-1)^{\messygr {j+1}{n}}(\sigma;t_{n}\otimes\dots\otimes\bdy(t_j)\otimes\dots\otimes
  t_{i+1}\otimes\{0\}\otimes \{0,1\}\otimes\dots\otimes\{0,1\}\otimes x)$\\
  \tiny $+\sum_{j=1}^{i-1}(-1)^{\messygr{i+1}{n}+i-j-1}(\sigma;t_{n}\otimes\dots\otimes t_{i+1}\otimes\{0\}\otimes
  \{0,1\}\otimes\dots\otimes\{1\}\otimes\dots\otimes\{0,1\}\otimes
  x)$\\
  \tiny $+\sum_{j=1}^{i-1}(-1)^{\messygr{i+1}{n}+i-j}(\sigma;t_{n}\otimes\dots\otimes t_{i+1}\otimes\{0\}\otimes
  \{0,1\}\otimes\dots\otimes\{0\}\otimes\dots\otimes\{0,1\}\otimes x)$\\
  \tiny $+(-1)^{\messygr{i+1}{n}+i-1}(\sigma;t_{n}\otimes\dots\otimes t_{i+1}\otimes\{0\}\otimes
  \{0,1\}\otimes\dots\otimes\{0,1\}\otimes\bdy(x))$};

\draw[->] (main) -- (main.south|-diffmain.north) node[pos=0.5,fill=white] {\tiny $\bdy$};
\draw[->] (second) -- (second.south|-diffsecond.north) node[pos=0.5,fill=white] {\tiny $\bdy$};
\draw (diffsecond) -- (diffsecond.south|-equalsecond.north) node[pos=0.5,fill=white] (eq) {\tiny $=$};
\draw (main) -- (second) node[pos=0.5,fill=white] {\tiny $\sim$};
\draw (diffmain.east) to [out=0,in=270] node[midway,fill=white] {\tiny
  $\sim$} (equalsecond.south-|second.east);
\node [right=-1ex of eq] {\tiny (by Equation~\eqref{eq:chain-cx-cat})};
\end{tikzpicture}
\]

For $\sigma\in\Cat_n$ and $x\in F(\sigma_0)$, write $[\sigma]\otimes
x$ for $(\sigma;\{0,1\}\otimes\dots\otimes\{0,1\}\otimes x)$. Note that
$\hocolim F$ is generated by such $[\sigma]\otimes x$ for
non-degenerate $\sigma$ and
\begin{equation}\label{eq:simple-hocolim-diff}
\bdy([\sigma]\otimes x)=(-1)^{n}[\sigma]\otimes(\bdy
x)+\sum_{i=1}^n(-1)^{n-i}\big([\sigma_{0,\dots,\hat{i},\dots,n}]\otimes x-[\sigma_{i,\dots,n}]\otimes\maxmap{F}(\sigma_{0,\dots,i})(x)\big).
\end{equation}

Let $I_\Ring\co\Cat\to\ChainComplexes$ be the constant functor
$I_\Ring(x)=\Ring$, viewed as a chain complex concentrated in degree $0$,
for each $x\in\ob(\Cat)=\Cat_0$; $\maxmap{I}_\Ring(f)=\Id$, the identity
map, for each $1$-morphism $f\in \Cat_1$; and
$\maxmap{I}_\Ring(\sigma)=0$ for any $n$-morphism $\sigma\in \Cat_n$,
$n>1$.  Fix another functor $F\from\Cat\to\Complexes$. For any cochains
$\alpha\in C^m(\hocolim F;\Ring)$ and $\beta\in C^n(\hocolim
I_\Ring;\Ring)$, define their product $\beta\cdot\alpha\in
C^{m+n}(\hocolim F;\Ring)$ as
\begin{equation}\label{eq:hocolim-product}
(\beta\cdot\alpha)([\sigma]\otimes
x)=
\begin{cases}
  \alpha([\sigma_{0,\dots,k-n}]\otimes
  x)\beta([\sigma_{k-n,\dots,k}]\otimes 1)&\text{\parbox{2in}{if
      $\sigma$ is a $k$-simplex, $k\geq n$, and $x\in
      F(\sigma_0)$ of grading
      $(n+m-k)$,}}\\
  0&\text{otherwise.}
\end{cases}
\end{equation}
Clearly, $(\beta\cdot\alpha)([\sigma]\otimes x)=0$ if $\sigma$ is
degenerate, and for any other $\gamma\in C^p(\hocolim I_\Ring;\Ring)$,
\[
(\gamma\cdot\beta)\cdot\alpha=\gamma\cdot(\beta\cdot\alpha).
\]
We use Formula~\eqref{eq:simple-hocolim-diff} to verify that
\[
\diff(\beta\cdot\alpha)=\diff(\beta)\cdot\alpha+(-1)^n\beta\cdot\diff(\alpha).
\]
To wit, for any $k$-simplex $\sigma$ and $x\in F(\sigma_0)$ of grading $(n+m+1-k)$,
\begin{align*}
&\diff(\beta\cdot\alpha)([\sigma]\otimes x)\\
=\,&(\beta\cdot\alpha)\big((-1)^{k}[\sigma]\otimes(\bdy
x)+\sum_{i=1}^k(-1)^{k-i}\big([\sigma_{0,\dots,\hat{i},\dots,k}]\otimes x-[\sigma_{i,\dots,k}]\otimes\maxmap{F}(\sigma_{0,\dots,i})(x)\big)\big)\\
=\,&(-1)^k\alpha([\sigma_{0,\dots,k-n}]\otimes
\bdy x)\beta([\sigma_{k-n,\dots,k}]\otimes
1)\\
  &\qquad{}+\sum_{i=1}^{k-n-1}(-1)^{k-i}\alpha([\sigma_{0,\dots,\hat{i},\dots,k-n}]\otimes
x)\beta([\sigma_{k-n,\dots,k}]\otimes 1)\\
&\qquad{}+\sum_{i=k-n}^{k}(-1)^{k-i}\alpha([\sigma_{0,\dots,k-n-1}]\otimes
x)\beta([\sigma_{k-n-1,\dots,\hat{i},\dots,k}]\otimes
1)\\
&\qquad{}-\sum_{i=1}^{k-n}(-1)^{k-i}\alpha([\sigma_{i,\dots,k-n}]\otimes\maxmap{F}(\sigma_{0,\dots,i})(x))\beta([\sigma_{k-n,\dots,k}]\otimes 1)\\
\shortintertext{while}
  &(\diff(\beta)\cdot\alpha+(-1)^n\beta\cdot\diff(\alpha))([\sigma]\otimes x)\\
  =\,&\alpha([\sigma_{0,\dots,k-n-1}]\otimes
  x)\diff\beta([\sigma_{k-n-1,\dots,k}]\otimes
  1)+(-1)^n\diff\alpha([\sigma_{0,\dots,k-n}]\otimes
  x)\beta([\sigma_{k-n,\dots,k}]\otimes 1)\\
  =\,&\alpha([\sigma_{0,\dots,k-n-1}]\otimes
  x)\sum_{i=k-n-1}^k(-1)^{k-i}\beta([\sigma_{k-n-1,\dots,\hat{i},\dots,k}]\otimes
  1)\\
  &\qquad{}+(-1)^k\alpha([\sigma_{0,\dots,k-n}]\otimes \bdy
  x)\beta([\sigma_{k-n,\dots,k}]\otimes 1)\\
  &\qquad{}+\sum_{i=1}^{k-n}(-1)^{k-i}\big(\alpha([\sigma_{0,\dots,\hat{i},\dots,k-n}]\otimes
  x)-\alpha([\sigma_{i,\dots,k-n}]\otimes\maxmap{F}(\sigma_{0,\dots,i})(x))\big)\beta([\sigma_{k-n,\dots,k}]\otimes 1).
\end{align*}
Therefore, $H^*(\hocolim I_\Ring)$ becomes a graded ring and
$H^*(\hocolim F)$ becomes a graded left module over it.

\begin{lemma}\label{lem:hocolim-geom-realization}
  Consider the functor $I_\Ring\from\Cat\to\Complexes$. Then $\hocolim
  I_\Ring$ is isomorphic to the cellular chain complex of the geometric
  realization $|\Cat|$ of $\Cat$ and the induced isomorphism on cohomology
  respects the ring structures.
\end{lemma}

\begin{proof}
  This is immediate from Equation~\eqref{eq:simple-hocolim-diff}. The
  homotopy colimit is freely generated by $[\sigma]\otimes 1$ for
  non-degenerate $n$-simplices $\sigma$ of $\Cat$ and $1\in
  I_\Ring(\sigma_0)=\Ring$, and the differential of such a generator is
  \begin{align*}
  \bdy([\sigma]\otimes 1)&=\sum_{i=1}^n(-1)^{n-i}[\sigma_{0,\dots,\hat{i},\dots,n}]\otimes 1-\sum_{i=1}^n(-1)^{n-i}[\sigma_{i,\dots,n}]\otimes\maxmap{I}_\Ring(\sigma_{0,\dots,i})(1)\\
  &= \sum_{i=1}^n(-1)^{n-i}[d_i\sigma]\otimes 1+(-1)^n[\sigma_{1,\dots,n}]\otimes\maxmap{I}_\Ring(\sigma_{0,1})(1)\\
  &=(-1)^n\sum_{i=0}^n(-1)^i[d_i\sigma]\otimes 1.
  \end{align*}
  Therefore, the map $[\sigma]\otimes 1\mapsto (-1)^{\binom{n+1}{2}}|\sigma|$ is an isomorphism between $\hocolim I_\Ring$ and the
  cellular chain complex of $|\Cat|$. 

  For the ring structures, let $\alpha\in C^m(|\Cat|)$ and $\beta\in
  C^n(|\Cat|)$ be cocycles and let $\ol{\alpha}$ and $\ol{\beta}$ be
  the corresponding cocycles in $C^*(\hocolim I_\Ring)$. Then for
  any $(m+n)$-simplex $\sigma$,
  \begin{align*}
    \ol{\beta}\cdot\ol{\alpha}([\sigma]\otimes 1)&=(-1)^{\binom{m+1}{2}+\binom{n+1}{2}}\alpha(|\sigma_{0,\dots,m}|)\beta(|\sigma_{m,\dots,m+n}|)\\
    &=(-1)^{\binom{m+1}{2}+\binom{n+1}{2}}(\alpha\cup\beta)(|\sigma|)\\
    &\sim (-1)^{\binom{m+n+1}{2}}(\beta\cup\alpha)(|\sigma|)\\
    &=\ol{\beta\cup\alpha}([\sigma]\otimes 1).\qedhere
  \end{align*}
\end{proof}

\begin{definition}
  For any functor $F\from\Cat\to\Complexes$, define its
  \emph{hypercohomology} $H^*(F)$ to be the cohomology of $\hocolim
  F$. Using the product from Equation~\eqref{eq:hocolim-product} and
  the isomorphism from Lemma~\ref{lem:hocolim-geom-realization}, the
  hypercohomology $H^*(F)$ is a graded left module over $H^*(|\Cat|)$.
\end{definition}

Next, suppose we have functors $F_0,F_1\from\Cat\to\Complexes$ and a
natural transformation $G$ from $F_0$ to $F_1$, that is, $G$ is a
functor $\Delta^1\times\Cat\to\Complexes$ satisfying
$G|_{\{i\}\times\Cat}=F_i$, $i=0,1$. To this, we will associate a
chain map $f_G\from\hocolim F_0\to\hocolim F_1$.  

With notation as in 
Section~\ref{sec:infty-cat-functors}, a
non-degenerate $n$-simplex $\sigma$ in $\Cat$ produces the following
non-degenerate simplices of $\Delta^1\times\Cat$:
\begin{itemize}[label=$\circ$]
\item $\Delta^1_{n;n}\times\sigma$; we denote this $\{0\}\times\sigma$. 
\item $\Delta^1_{n;-1}\times\sigma$; we denote this $\{1\}\times\sigma$. 
\item $\Delta^1_{n;i}\times\sigma$ for $0\leq i< n$. 
\item $\Delta^1_{n+1;i}\times s_i\sigma$ for $0\leq i\leq n$.
\end{itemize}
Therefore, by Convention~\ref{convention:highest-maps-in-complexes},
$G$ consists of maps
\begin{itemize}
\item $\maxmap{G}(\Delta^1_{n;i}\times\sigma)\from
F_0(\sigma_0)\to F_1(\sigma_n)$, $0\leq i<n$, of degree $n-1$,
\item $\maxmap{G}(\Delta^1_{n+1;i}\times s_i\sigma)\from F_0(\sigma_0)\to F_1(\sigma_n)$, $0\leq i\leq n$, of degree $n$,
\end{itemize}
for each non-degenerate $n$-simplex $\sigma$ of $\Cat$.  We will use
the notation $\ol{G}(\sigma;i)$ for
$\maxmap{G}(\Delta^1_{n+1;i}\times s_i\sigma)$.  These maps are
required to satisfy certain compatibility conditions coming from
Equation~\eqref{eq:chain-cx-cat}, one of which is the following:
\begin{align*}
  &\bdy\circ \ol{G}(\sigma;i)+(-1)^{n+1}\ol{G}(\sigma;i)\circ \bdy\\
 &\qquad=\sum_{\ell=1}^{i-1}(-1)^{n-\ell}\ol{G}(\sigma_{0,\dots,\hat{\ell},\dots,n};i-1)+\sum_{\ell=i+2}^{n}(-1)^{n-\ell}\ol{G}(\sigma_{0,\dots,\widehat{\ell-1},\dots,n};i)\\
  &\qquad\qquad{}+\sum_{\ell=\max(1,i)}^{\min(i+1,n)}(-1)^{n-\ell}\maxmap{G}(\Delta^1_{n;\ell-1}\times \sigma)
    -\sum_{\ell=1}^i(-1)^{n-\ell}
  \ol{G}(\sigma_{\ell,\dots,n};i-\ell)\circ \maxmap{F}_0(\sigma_{0,\dots,\ell})\\
  &\qquad\qquad{}-\sum_{\ell=i+1}^{n}(-1)^{n-\ell}\maxmap{F}_1(\sigma_{\ell-1,\dots,n})\circ
    \ol{G}(\sigma_{0,\dots,\ell-1};i).
\end{align*}
In particular, the maps of the form
$\maxmap{G}(\Delta^1_{n;i}\times \sigma)$ are determined by the
maps of the form $\ol{G}(\sigma;i)$. Consequently, the data for
$G$ is entirely encapsulated by these maps $\ol{G}$, and one can
check that they are only required to satisfy the following equation:
\begin{equation}\label{eq:nat-transform-main-old}
\begin{split}
  &\sum_{i=0}^n(-1)^{n-i}\big(\bdy\circ \ol{G}(\sigma;i)+(-1)^{n+1}\ol{G}(\sigma;i)\circ\bdy\big)\\
&\qquad=\sum_{\substack{1\leq\ell\leq n-1\\0\leq i\leq n-1}}
(-1)^{i+\ell-1}\ol{G}(\sigma_{0,\dots,\hat{\ell},\dots,n};i)+\sum_{\substack{1\leq \ell\leq n\\0\leq i\leq n-\ell}}
(-1)^{i+1}\ol{G}(\sigma_{\ell,\dots,n};i)\circ\maxmap{F}_0(\sigma_{0,\dots,\ell})\\
&\qquad\quad{}+\sum_{\substack{0\leq\ell\leq
    n-1\\0\leq i\leq \ell}}(-1)^{i+\ell}\maxmap{F}_1(\sigma_{\ell,\dots,n})\circ \ol{G}(\sigma_{0,\dots,\ell};i).
\end{split}
\end{equation}

Now construct the map $f_G\from\hocolim F_0\to\hocolim F_1$ by defining, for non-degenerate $n$-simplices
$\sigma$ in $\Cat$ and $x\in F_0(\sigma_0)$,
\begin{equation}\label{eq:hocolim-map}
f_G([\sigma]\otimes x)=\sum_{m=0}^n\Big([\sigma_{m,\dots,n}]\otimes\big(\sum_{i=0}^{m}(-1)^{m-i}\ol{G}(\sigma_{0,\dots,m};i)(x)\big)\Big).
\end{equation}
\begin{lemma}\label{lem:hocolim-map}
  Given functors $F_0,F_1\from\Cat\to\Complexes$ and a natural
  transformation $G$ from $F_0$ to $F_1$, let
  $f_{G}\from\hocolim(F_0)\to \hocolim(F_1)$ be the map from
  Equation~\eqref{eq:hocolim-map}. Then $f_G$ is a chain map and the
  induced map $H^*(F_1)\to H^*(F_0)$ on hypercohomology respects the
  module structure over $H^*(|\Cat|)$.
\end{lemma}

\begin{proof}
  We represent a generator of the form
  $[\sigma]\otimes x$, for $x\in F_i(\sigma_0)$, as living over the
  barycenter of $\{i\}\times \sigma$. Furthermore, we represent
  the map $f_G$ by {\color{blue}solid straight} arrows, and the three terms in
  the differential from Equation~\eqref{eq:simple-hocolim-diff} by
  {\color{gray}curved}, {\color{red}dashed}, and
  {\color{green!50!black}dotted} arrows, respectively. For example, the case $n=2$ is:

  \[
  \begin{tikzpicture}[every node/.style={inner sep=0pt, outer sep=1pt}]
    \foreach \i in {0,1}
    {
      \node (A\i) at (-0.5+2*\i,-0.5+2*\i) {\tiny $\i\sigma_0$};
      \node  (B\i) at (0.3+2*\i,1.7+2*\i) {\tiny $\i\sigma_1$};
      \node (C\i) at (3+2*\i,-1+2*\i) {\tiny $\i\sigma_2$};
      
      \node (BC\i) at ($(B\i)!.5!(C\i)$) {$\cdot$};
      \node (AC\i) at ($(A\i)!.5!(C\i)$) {$\cdot$};
      \node  (AB\i) at ($(B\i)!.5!(A\i)$) {$\cdot$};
      
      \node (ABC\i) at ($(AB\i)!0.333!(C\i)$) {$\cdot$};
    }
  \foreach \i in {1}
  {
    \draw[->,thick,dashed,red] (ABC\i) -- (AC\i);
    \draw[->,thick,dashed,red] (ABC\i) -- (AB\i);
    \draw[->,thick,dashed,red] (AB\i) -- (A\i);
    \draw[->,thick,dashed,red] (AC\i) -- (A\i);
    \draw[->,thick,dashed,red] (BC\i) -- (B\i);
    
    \draw[->,thick,green!50!black,dotted] (ABC\i) -- (BC\i);
    \draw[->,thick,green!50!black,dotted] (ABC\i) -- (C\i);
    \draw[->,thick,green!50!black,dotted] (AB\i) -- (B\i);
    \draw[->,thick,green!50!black,dotted] (AC\i) -- (C\i);
    \draw[->,thick,green!50!black,dotted] (BC\i) -- (C\i);

    \draw[->,thick,gray,looseness=8] (A\i) to [out=180,in=270] (A\i);
    \draw[->,thick,gray,looseness=8] (B\i) to [out=90,in=180] (B\i);
    \draw[->,thick,gray,looseness=8] (C\i) to [out=270,in=360] (C\i);

    \draw[->,thick,gray,looseness=12] (AB\i) to [out=100,in=190] (AB\i);
    \draw[->,thick,gray,looseness=12] (BC\i) to [out=0,in=90] (BC\i);
    \draw[->,thick,gray,looseness=12] (AC\i) to [out=220,in=310] (AC\i);

    \draw[->,thick,gray,looseness=12] (ABC\i) to [out=45,in=135] (ABC\i);
  }
  \foreach \l in {A,B,C,AB,BC,AC,ABC}
  {
    \draw[->,thin,blue] (\l 0) -- (\l 1);
  }
    \draw[->,thin,blue] (AB0) -- (B1);
    \draw[->,thin,blue] (AC0) -- (C1);
    \draw[->,thin,blue] (BC0) -- (C1);

    \draw[->,thin,blue] (ABC0) -- (BC1);
    \draw[->,thin,blue] (ABC0) -- (C1);
  \foreach \i in {0}
  {
    \draw[->,thick,dashed,red] (ABC\i) -- (AC\i);
    \draw[->,thick,dashed,red] (ABC\i) -- (AB\i);
    \draw[->,thick,dashed,red] (AB\i) -- (A\i);
    \draw[->,thick,dashed,red] (AC\i) -- (A\i);
    \draw[->,thick,dashed,red] (BC\i) -- (B\i);
    
    \draw[->,thick,green!50!black,dotted] (ABC\i) -- (BC\i);
    \draw[->,thick,green!50!black,dotted] (ABC\i) -- (C\i);
    \draw[->,thick,green!50!black,dotted] (AB\i) -- (B\i);
    \draw[->,thick,green!50!black,dotted] (AC\i) -- (C\i);
    \draw[->,thick,green!50!black,dotted] (BC\i) -- (C\i);

    \draw[->,thick,gray,looseness=8] (A\i) to [out=180,in=270] (A\i);
    \draw[->,thick,gray,looseness=8] (B\i) to [out=90,in=180] (B\i);
    \draw[->,thick,gray,looseness=8] (C\i) to [out=270,in=360] (C\i);

    \draw[->,thick,gray,looseness=12] (AB\i) to [out=100,in=190] (AB\i);
    \draw[->,thick,gray,looseness=12] (BC\i) to [out=0,in=90] (BC\i);
    \draw[->,thick,gray,looseness=12] (AC\i) to [out=220,in=310] (AC\i);

    \draw[->,thick,gray,looseness=12] (ABC\i) to [out=45,in=135] (ABC\i);
  }
  \end{tikzpicture}
  \]

  To see $f_G$ is a chain map, fix a non-degenerate $n$-simplex $\sigma$
  in $\Cat$ and $x\in F_0(\sigma_0)$, and consider the terms of
  $(\bdy\circ f_G-f_G\circ\bdy)([\sigma]\otimes x)$. The terms
  are of the following two types:
  \begin{enumerate}[leftmargin=*] 
  \item $[\sigma_{m,\dots,\hat{\ell},\dots,n}]\otimes y$, $0\leq
    m<\ell\leq n$, $y\in F_1(\sigma_m)$. From both
    $\bdy(f_G([\sigma]\otimes x))$ and $f_G(\bdy([\sigma]\otimes x))$,
    we get the term
    \[
    (-1)^{n-\ell}[\sigma_{m,\dots,\hat{\ell},\dots,n}]\otimes\big(\sum_{i=0}^{m}(-1)^{m-i}\ol{G}(\sigma_{0,\dots,m};i)(x)\big)
    \]
    (as compositions of {\color{blue}solid
      straight}--{\color{red}dashed} arrows in either order) which
    therefore cancel.
  \item $[\sigma_{m,\dots,n}]\otimes y$, $0\leq m\leq n$, $y\in
    F_1(\sigma_m)$. From $\bdy(f_G([\sigma]\otimes x))$, we get the
    terms
    {\small
    \[
    (-1)^{n-m}[\sigma_{m,\dots,n}]\otimes\Big(\bdy\big(\sum_{i=0}^{m}(-1)^{m-i}\ol{G}(\sigma_{0,\dots,m};i)(x)\big)-\sum_{\ell=0}^{m-1}\maxmap{F}_1(\sigma_{\ell,\dots,m})\big(\sum_{i=0}^{\ell}(-1)^{\ell-i}\ol{G}(\sigma_{0,\dots,\ell};i)(x)\big)\Big)
    \]
  }
    (as compositions of {\color{blue}solid
      straight}--{\color{gray}curved} arrows and {\color{blue}solid
      straight}--{\color{green!50!black}dotted} arrows, repsectively)
    while from $f_G(\bdy([\sigma]\otimes x))$ we get the terms
    \begin{align*}
    [\sigma_{m,\dots,n}]\otimes\Big((-1)^n\sum_{i=0}^{m}(-1)^{m-i}&\ol{G}(\sigma_{0,\dots,m};i)(\bdy x)+\sum_{i=0}^{m-1}(-1)^{m-1-i}\sum_{\ell=1}^{m-1}(-1)^{n-\ell}\ol{G}(\sigma_{0,\dots,\hat{\ell},\dots,m};i)(x)\\
    &-\sum_{\ell=1}^m\sum_{i=0}^{m-\ell}(-1)^{m-\ell-i}(-1)^{n-\ell}\ol{G}(\sigma_{\ell,\dots,m};i)\big(\maxmap{F}_0(\sigma_{0,\dots,\ell})(x)\big)\Big)
    \end{align*}
    (as compositions of {\color{gray}curved}--{\color{blue}solid
      straight} arrows, {\color{red}dashed}--{\color{blue}solid
      straight} arrows, and {\color{green!50!black}dotted}--{\color{blue}solid
      straight} arrows, respectively).
    Therefore, we need to show:
    \begin{align*}
      &\sum_{i=0}^m(-1)^{m-i}\big(\bdy\circ \ol{G}(\sigma_{0,\dots,m};i)+(-1)^{m+1}\ol{G}(\sigma_{0,\dots,m};i)\circ\bdy\big)\\
      &\qquad=\sum_{\substack{1\leq\ell\leq m-1\\0\leq i\leq m-1}}
      (-1)^{i+\ell-1}\ol{G}(\sigma_{0,\dots,\hat{\ell},\dots,m};i)+\sum_{\substack{1\leq
          \ell\leq m\\0\leq i\leq m-\ell}}
      (-1)^{i+1}\ol{G}(\sigma_{\ell,\dots,m};i)\circ\maxmap{F}_0(\sigma_{0,\dots,\ell})\\
      &\qquad\quad{}+\sum_{\substack{0\leq\ell\leq m-1\\0\leq i\leq
          \ell}}(-1)^{i+\ell}\maxmap{F}_1(\sigma_{\ell,\dots,m})\circ
      \ol{G}(\sigma_{0,\dots,\ell};i),
    \end{align*}
    which is exactly Equation~\eqref{eq:nat-transform-main-old}.
  \end{enumerate}

Finally, it is clear from Equations~\eqref{eq:hocolim-product}
and~\eqref{eq:hocolim-map} that the dual of the map $f_G$ respects
products with cochains in $C^*(\hocolim I_\Ring)$ on the nose.
\end{proof}

Going one step further:
\begin{lemma}\label{lem:hocolim-htpy}
  Suppose we have functors $F_0,F_1,F_2\from\Cat\to\Complexes$ and
  natural transformations $G_{01}$ from $F_0$ to $F_1$, $G_{12}$ from
  $F_1$ to $F_2$, and $G_{02}$ from $F_0$ to $F_2$. Further suppose
  that there is a functor $H\from\Delta^2\times\Cat\to\Complexes$
  satisfying $H|_{\{i\}\times\Cat}=F_i$, and
  $H|_{d_i\Delta^2\times\Cat}=G_{0,\dots,\hat{i},\dots,2}$,
  $i=0,1,2$. Then there is an induced chain homotopy connecting
  $f_{G_{12}}\circ f_{G_{01}}$ and
  $f_{G_{02}}$.
\end{lemma}

We leave the proof, which is a more elaborate version of the previous
proof, to the reader.

\begin{remark}
  If we were to identify the explicit description of $\hocolim F$ with
  Joyal's more abstract definition~\cite[Definition
  4.5]{Joyal02:quasi-cat},~\cite[Definition 1.2.13.4]{Lurie09:topos}
  then Lemmas~\ref{lem:hocolim-map} and~\ref{lem:hocolim-htpy}
  would follow. 
\end{remark}

\begin{remark}
  For $\Cat$ an infinity category, functors from $\Cat$ to $\ChainComplexes$ form a stable
  infinity category~\cite[Proposition 1.1.3.1]{Lurie:HigherAlgebra}
  (see also~\cite[Proposition 1.2.7.3]{Lurie09:topos}).  In
  particular, the homotopy category $h\Fun(\Cat,\ChainComplexes)$ of
  $\Fun(\Cat,\ChainComplexes)$ is an additive (in fact, triangulated)
  category~\cite[Theorem 1.1.2.14 and Remark
  1.1.2.15]{Lurie:HigherAlgebra}. Using this, the hypercohomology of a
  functor $F$ can also be defined as
  $H^n(F)=\Hom_{h\Fun(\Cat,\Complexes)}(F,I_\Ring[n])$.
\end{remark}

\subsection{The fat homotopy colimit}\label{sec:fat-homotopy-colimits}
In this section, we define a version of the homotopy colimit that
makes sense for maps of semi-simplicial sets. We also give a modified
definition of a natural transformation. These tools are used in
Section~\ref{sec:comp-morse}.

\begin{definition} \label{def:fat-hocolim}
  Let $\Cat$ be a semi-simplicial set and
  $F\co\Cat\to \ChainComplexes$ a map of semi-simplicial sets. The
  \emph{fat homotopy colimit} of $F$ is the chain complex
  \begin{equation}\label{eq:fat-chain-hocolim-def}
    \fathocolim F = \bigoplus_{n\geq 0}
    \bigoplus_{\sigma\in \Cat_n}  I_*^{\otimes n}\otimes F(\sigma_0)/\sim,
  \end{equation}
  where the equivalence relation $\sim$ is generated by
  \begin{align*}
    (\sigma;t_n\otimes\dots\otimes t_1\otimes x)&\sim(d_i\sigma;t_n\otimes\dots\otimes \widehat{t_i}\otimes\dots\otimes t_1\otimes x)&\text{if }t_i=\{1\}\\
    (\sigma;t_n\otimes\dots\otimes t_1\otimes
    x)&\sim(\sigma_{i,\dots,n};t_n\otimes\cdots\otimes t_{i+1}\otimes
    \maxmap{F}(\sigma_{0,\dots,i})(x))\\
    &&\mathllap{\text{if }t_1=\dots=t_{i-1}=\{0,1\}\text{ and }t_i=\{0\}.}
  \end{align*}
  The differential is induced by the tensor product differential in
  Formula~\eqref{eq:fat-chain-hocolim-def}, namely:
  \begin{align*}
  \bdy(\sigma;t_n\otimes\dots\otimes t_1\otimes
  x)&=\sum_{i=1}^n(-1)^{\gr(t_{i+1})+\dots+\gr(t_n)}(\sigma;t_n\otimes\dots\otimes\bdy(t_i)\otimes\dots\otimes
  t_1\otimes x)\\
  &\quad\qquad{}+ (-1)^{\gr(t_1)+\dots+\gr(t_{n})}(\sigma;t_n\otimes\dots\otimes
  t_1\otimes \bdy(x)).
  \end{align*}
\end{definition}

\begin{lemma}
  Formula~\eqref{eq:hocolim-product} makes $H^*(\fathocolim F)$ into a module
  over $H^*(\FatRealize{\Cat})$.
\end{lemma}
\begin{proof}
  The proof is exactly the same as for the ordinary homotopy colimit (see
  Section~\ref{sec:homotopy-colimits}, following
  Equation~(\ref{eq:hocolim-product})).
\end{proof}

\begin{proposition}\label{prop:fat-hocolim-is-hocolim}
  If $\Cat$ is a simplicial set and $F\co\Cat\to\ChainComplexes$ is a
  map of simplicial sets then there is a quasi-isomorphism
  \[
    \hocolim F\simeq \fathocolim F.
  \]
  Further, the induced map on cohomology intertwines the actions of
  $H^*(\Realize{\Cat})$ and $H^*(\FatRealize{\Cat})$.
\end{proposition}
\begin{proof}
  First, note that there is an obvious quotient map
  \[
    \Pi\co \fathocolim F\to \hocolim F
  \]
  which quotients by the additional relations in
  Definition~\ref{def:hocolim} corresponding to the degeneracy
  maps.
  The map $\Pi$ clearly intertwines the actions of
  $H^*(\Realize{\Cat})$ and $H^*(\FatRealize{\Cat})$.

  Recall that there is a filtration $\Filt$ on $\hocolim F$
  where $\Filt_m\hocolim F$ comes from simplices $\sigma\in\Cat_n$ for
  $n\leq m$. There is a corresponding filtration
  $\Filt$ on $\fathocolim F$, defined in the same way. The quotient
  map $\Pi$ respects these filtrations.

  Consider the associated spectral sequences. The $E_1$-page
  for $\hocolim F$ is
  \[
    E^1_{p,q}=\bigoplus_{\sigma\in\Cat_p\text{ nondegenerate}}H_q(F(\sigma_0))
  \]
  while the $E_1$-page for $\fathocolim F$ is
  \[
    \mathbb{E}^1_{p,q}=\bigoplus_{\sigma\in\Cat_p}H_q(F(\sigma_0)).
  \]
  Generators for $\mathbb{E}^1_{p,q}$ have the form
  $(\sigma,\{0,1\}^{\otimes p}\otimes x)$ with $\sigma\in \Cat_p$ a
  $p$-simplex and $x\in H_{q}(F(\sigma_0))$; as before, we denote
  it $[\sigma]\otimes x$.

  We claim that the degenerate simplices $\sigma$ span an acyclic
  subcomplex of $\mathbb{E}^1$.  It then follows that the map
  $\Pi_*\co \mathbb{E}^2_{p,q}\to E^2_{p,q}$ is an isomorphism. Thus,
  the original map $\Pi$ is a quasi-isomorphism, as desired.

  To see that the degenerate simplices form a subcomplex, recall that
  for $\sigma\in\Cat_p$ and $x\in H_{q}(F(\sigma_0))$, the $d^1$
  differential is given by
  \begin{align*}
    d^1_{p,q}([\sigma]\otimes x)
    &=\left[\sum_{i=1}^p(-1)^{p-i}[d_i\sigma]\otimes x
    \right]
    +(-1)^p[d_0\sigma]\otimes \maxmap{F}(\sigma_{01})_*(x)\\
&=\sum_{i=0}^p (-1)^{p-i}[d_i\sigma]\otimes x_i,
  \end{align*}
  where $x_i=x$ if $i>0$ and
  $x_0=\maxmap{F}(\sigma_{01})_*(x)$. (Here,
  $\maxmap{F}(\sigma_{01})_*$ is the map on homology induced by the
  chain map $\maxmap{F}(\sigma_{01})\co 
  F(\sigma_0)\to F(\sigma_1)$.)  If $\sigma=s_j(\tau)$, $j>0$, then each
  of the simplices in this sum is degenerate except for
  $d_j\sigma=\tau$ and $d_{j+1}\sigma=\tau$; these two terms
  cancel. If $\sigma=s_0(\tau)$ then $\maxmap{F}(\sigma_{01})$ is the
  identity map so again $d_0\sigma$ and $d_1\sigma$
  cancel. Thus, the degenerate
  simplices span a subcomplex. 

  All that remains is to prove that this
  subcomplex is acyclic. Associate two integers $(J,K)$ to each 
  degenerate simplex as follows:
  \begin{align*}
    J(\sigma)&=\min\{j\mid \sigma\in\im(s_j)\}&\text{for all degenerate $\sigma$,}\\
    K(\sigma)&=\max\{k\mid \sigma\in\im(s_j^k)\}&\text{for all degenerate $\sigma$ with $J(\sigma)=j$.}
  \end{align*}
  Note that $J$ defines a filtration on the subcomplex of degenerate simplices, since for
  $\sigma=s_j\tau$, the $d^1$ differential takes the form
  \begin{align*}
    d^1_{p,q}([s_j\tau]\otimes x)
    &=\sum_{i=0}^p (-1)^{p-i}[d_is_j\tau]\otimes x_i\\
    &=\sum_{i=0}^{j-1}(-1)^{p-i}[s_{j-1}d_i\tau]\otimes x_i+\sum_{i=j+2}^{p}(-1)^{p-i}[s_{j}d_{i-1}\tau]\otimes x.
  \end{align*}
  The terms inside the first summation have smaller $J$, while the
  terms inside the second summation have the same or smaller $J$.

  Consider the associated graded complex in $J$-grading $j$. We claim
  that for all $j$, the associated graded complex itself is
  acyclic, which implies that the original complex is acyclic, as
  well. We use $K$, as follows.
  Fix $\sigma$ with $J(\sigma)=j$ and $K(\sigma)=k$, so
  $\sigma=s^k_j\tau$ for some $\tau$. The differential in the associated graded
  complex is given by
  \begin{align*}
    d^1_{p,q}([s^k_j\tau]\otimes x)
    &=\sum_{i=j+2}^{p}(-1)^{p-i}[d_is^k_j\tau]\otimes x\\
    &=\sum_{i=j+2}^{j+k}(-1)^{p-i}[s_j^{k-1}\tau]\otimes x +\sum_{i=j+k+1}^p(-1)^{p-i}[s_j^{k}d_{i-k}\tau]\otimes x
  \end{align*}
  (with the caveat that some of the terms inside the second summation
  sign might have smaller $J$-grading, and hence not actually
  appear).  For the first term, the $K$-grading drops by exactly one,
  while for the terms inside the second summation, the value of $K$
  remains same or increases. Therefore, $|\cdot|-K$ is a filtration on
  this complex, where $|\sigma|=p$ is the dimension of the simplex.
  
  Now consider the associated graded complex of this filtration. The
  differential is given by
  \begin{align*}
    d^1_{p,q}([s^k_j\tau]\otimes x)
    &=\sum_{i=j+2}^{j+k}(-1)^{p-i}[s_j^{k-1}\tau]\otimes x\\
    &=\begin{cases}
      (-1)^{p-j}[s_j^{k-1}\tau]\otimes x&\text{if $k$ is even,}\\
      0&\text{otherwise.}
      \end{cases}
  \end{align*}
  That is, the differential pairs the terms $[s^k_j\tau]\otimes x$ and
  $[s^{k-1}_j\tau]\otimes x$ ($k$ even). (Note that $s^k_j\tau=s^k_j\tau'$ if and only
  if $s^{k-1}_j\tau=s^{k-1}_j\tau'$, by applying $d_j$ or $s_j$ to
  both sides.) Therefore, the complex is acyclic.
\end{proof}

\begin{definition}\label{def:summed-nat-trans}
  Let $\Cat$ be a semi-simplicial set and let
  $F_0,F_1\co\Cat\to\ChainComplexes$ be maps of semi-simplicial
  sets. A \emph{summed natural transformation} from $F_0$ to $F_1$
  consists of a map
  \[
    G(\sigma)\co F_0(\sigma_0)\to F_1(\sigma_n)
  \]
  for each $n$-simplex $\sigma\in\Cat_n$, satisfying the structure
  equation
  \begin{equation}\label{eq:summed-hocolim-map}
    \begin{split}
      (-1)^n\bdy\circ G(\sigma)-G(\sigma)\circ\bdy
      =&\sum_{1\leq\ell\leq n-1}
      (-1)^{\ell-1}G(\sigma_{0,\dots,\hat{\ell},\dots,n})-\sum_{1\leq \ell\leq n}
      G(\sigma_{\ell,\dots,n})\circ \maxmap{F}_0(\sigma_{0,\dots,\ell})\\
      &+\sum_{0\leq\ell\leq
          n-1}(-1)^{\ell}\maxmap{F}_1(\sigma_{\ell,\dots,n})\circ G(\sigma_{0,\dots,\ell}).
    \end{split}
  \end{equation}
\end{definition}

Given a natural transformation $G$ as in
Section~\ref{sec:homotopy-colimits},
$\sum_{i=0}^n (-1)^i\ol{G}(\sigma;i)$ defines a summed natural
transformation (for non-degenerate $n$-simplices $\sigma$); compare
Formula~\eqref{eq:nat-transform-main-old}.

\begin{proposition}\label{prop:summed-nat-chain-map}
  A summed natural transformation $G$ from $F_0$ to $F_1$ induces a
  chain map
  \[
    \fathocolim G\co \fathocolim F_0\to\fathocolim F_1.
  \]
  The induced map on cohomology respects the
  $H^*(\FatRealize{\Cat})$-module structure.  Further, if for each
  $0$-simplex $\sigma\in\Cat$ the map
  $G(\sigma)\co F_0(\sigma)\to F_1(\sigma)$ is a quasi-isomorphism
  then $\fathocolim G$ is a quasi-isomorphism.
\end{proposition}
\begin{proof}
  The proof is similar to that of Lemma~\ref{lem:hocolim-map}. As in
  Section~\ref{sec:homotopy-colimits}, write
  $(\sigma;\{0,1\}\otimes\dots\otimes\{0,1\}\otimes x)$ as
  $[\sigma]\otimes x$. For $\epsilon\in\{0,1\}$, let
  \[
    \bdy_{i,\epsilon}[\sigma]=\bigl(\sigma;\overbrace{\{0,1\}\otimes\dots\otimes\{0,1\}}^{n-i}\otimes\{\epsilon\}\otimes\overbrace{\{0,1\}\otimes\dots\otimes\{0,1\}}^{i-1}\bigr)
  \]
  so 
  \[
    \bdy([\sigma]\otimes x)=(-1)^n[\sigma]\otimes(\bdy x)+\sum_{i=1}^{n}(-1)^{n-i}[(\bdy_{i,1}[\sigma])\otimes x - (\bdy_{i,0}[\sigma])\otimes x].
  \]
  
  We will write $\Gamma$ for the map $\fathocolim
  G$ we are to construct. Let
  \[
    \Gamma([\sigma]\otimes x)
    =\sum_{m=0}^n(-1)^m\Big([\sigma_{m,\dots,n}]\otimes G(\sigma_{0,\dots,m})(x)\Big).
  \]
  Then
  \begin{align*}
    &\phantom{==}\bdy_{F_1(\sigma_n)} \Gamma([\sigma]\otimes x)\\
    &=\sum_{m=0}^n(-1)^m\Big(\bigl(\bdy[\sigma_{m,\dots,n}]\bigr)\otimes G(\sigma_{0,\dots,m})(x)
      +(-1)^{n-m}[\sigma_{m,\dots,n}]\otimes\bdy\bigl(G(\sigma_{0,\dots,m})(x)\bigr)\Big)\\
    &=\sum_{m=0}^n\Big(\sum_{i=1}^{n-m}\bigl((-1)^{n-i}\bdy_{i,1}[\sigma_{m,\dots,n}]+(-1)^{n-i+1}\bdy_{i,0}[\sigma_{m,\dots,n}]\bigr)\otimes G(\sigma_{0,\dots,m})(x)\\    
      &\qquad+(-1)^{n+m}[\sigma_{m,\dots,n}]\otimes
        \big[\bigl(G(\sigma_{0,\dots,m})(\bdy x)\bigr)
        +\sum_{1\leq \ell\leq m-1}(-1)^{\ell+1}G(\sigma_{0,\dots,\hat{\ell},\dots,m})(x)\\
        &\qquad\qquad-\sum_{1\leq \ell\leq m}G(\sigma_{\ell,\dots,m})(\maxmap{F}_0(\sigma_{0,\dots,\ell})(x))
        +\sum_{0\leq \ell\leq m-1}(-1)^\ell \maxmap{F}_1(\sigma_{\ell,\dots,m})(G(\sigma_{0,\dots,\ell})(x))
          \big]\Big)\\
    &=\sum_{m=0}^n\Big(\sum_{i=1}^{n-m}\big[(-1)^{n-i}[(d_{m+i}\sigma)_{m,\dots,n-1}]\otimes
      G(\sigma_{0,\dots,m})(x)\\
    &\qquad +(-1)^{n-i+1}[\sigma_{m+i,\dots,n}]\otimes \maxmap{F}_1(\sigma_{m,\dots,m+i})(G(\sigma_{0,\dots,m})(x))\big]\\    
      &\qquad+(-1)^{n+m}[\sigma_{m,\dots,n}]\otimes
        \big[\bigl(G(\sigma_{0,\dots,m})(\bdy x)\bigr)
        +\sum_{1\leq \ell\leq m-1}(-1)^{\ell+1}G(\sigma_{0,\dots,\hat{\ell},\dots,m})(x)\\
        &\qquad\qquad-\sum_{1\leq \ell\leq m}G(\sigma_{\ell,\dots,m})(\maxmap{F}_0(\sigma_{0,\dots,\ell})(x))
        +\sum_{0\leq \ell\leq m-1}(-1)^\ell \maxmap{F}_1(\sigma_{\ell,\dots,m})(G(\sigma_{0,\dots,\ell})(x))
          \big]\Big)\\
    &=\sum_{m=0}^n\Big(\sum_{i=1}^{n-m}(-1)^{n-i}[(d_{m+i}\sigma)_{m,\dots,n-1}]\otimes
      G(\sigma_{0,\dots,m})(x)\\
      &\qquad+(-1)^{n+m}[\sigma_{m,\dots,n}]\otimes
        \big[\bigl(G(\sigma_{0,\dots,m})(\bdy x)\bigr)
        +\sum_{1\leq \ell\leq m-1}(-1)^{\ell+1}G(\sigma_{0,\dots,\hat{\ell},\dots,m})(x)\\
        &\qquad\qquad-\sum_{1\leq \ell\leq
          m}G(\sigma_{\ell,\dots,m})(\maxmap{F}_0(\sigma_{0,\dots,\ell})(x))\big]\Big)\\
    &=\Gamma(\bdy([\sigma]\otimes x)),
  \end{align*}
  as desired.

  The statement that the induced map on cohomology respects the module
  structure is left to the reader.

  Finally, the last statement follows from the spectral sequence used
  in the proof of Proposition~\ref{prop:fat-hocolim-is-hocolim} and
  spectral sequence comparison.
\end{proof}


\section{The construction}\label{sec:construction}
\subsection{Hypotheses and statement of result}\label{subsec:hypotheses}
Fix a Lie group $G$ acting on a symplectic manifold
$(M,\omega)$ on the left, preserving Lagrangians $L_0$ and $L_1$ setwise. 
\begin{hypothesis}\label{hyp:Floer-defined}
  We assume:
  \begin{enumerate}[label=(J-\arabic*)]
  \item\label{item:J-1}
    For any loop of paths
    \[
      v\co \bigl([0,1]\times S^1,\{1\}\times S^1,\{0\}\times S^1\bigr)\to (M,L_0,L_1)
    \]
    with $v|_{[0,1]\times\{1\}}$ a constant path,
    the $\omega$-area of $v$ and the Maslov index of $v$ both vanish.
    (Compare~\cite[Theorem 1.0.1]{WW12:compose-correspond}.)
  \item\label{item:J-2} The manifold $M$ is either compact or
    $I$-convex at infinity~\cite{Gromov85} with respect to an
    $\omega$-compatible almost complex structure $I$ defined outside
    some $G$-invariant compact set and such that $I$ is
    $G$-invariant. (See also~\cite{SeidelSmith10:localization}.)
        Further, in the latter case,
    the Lagrangians $L_0$ and $L_1$ are either compact or else $M$ has
    finite type and $L_0$ and $L_1$ are conical at infinity and
    disjoint outside the $G$-invariant compact set from the definition of $I$ (see, e.g.,~\cite[Section
    5]{KhS02:BraidGpAction} or~\cite[Section 2.1]{Hendricks:symplecto}
    and the references there).
  \item\label{item:J3} Every $G$-twisted loop of paths
    (Definition~\ref{def:G-twisted-loop}) has Maslov index $0$.
  \end{enumerate}
\end{hypothesis}
That is, we make enough assumptions that the Lagrangian intersection
Floer complexes of a generic perturbation of $L_0$ and $L_1$ are
defined with respect to a generic almost complex structure without
resorting to Novikov coefficients or virtual
techniques. Hypothesis~\ref{item:J3} guarantees that we have $\ZZ$
gradings on our equivariant complexes defined below. If $G$ is compact
then Hypothesis~\ref{item:J3} is automatically satisfied; see
Lemma~\ref{lem:mu-on-pi2-g}.

In what follows, either:
\begin{itemize}
\item Fix a choice of coherent $G$-orientations for the moduli spaces,
  as discussed in Section~\ref{sec:G-or}. Or
\item Take Floer complexes with coefficients in a field of characteristic $2$.
\end{itemize}

\subsection{Spaces of Hamiltonian perturbations and almost complex
  structures}\label{sec:J-space}

By a \emph{cylindrical almost complex structure} on $M$ we mean a
smooth one-parameter family of smooth almost complex structures
$J=J(t)$, $t\in[0,1]$ on $M$ each of which is compatible with
$\omega$. In the non-compact case, we require that each $J(t)$ agree
with the almost complex structure $I$ in the definition of convexity
outside the fixed $G$-invariant compact set.  Similarly, by a \emph{cylindrical
  Hamiltonian} we mean a smooth function $H\co M\to\RR$, which in the
non-compact case is supported in the $G$-invariant compact set.  Let
$\JSpace_\cyl$ denote the space of pairs $(J,H)$ of a cylindrical
almost complex structure and a cylindrical Hamiltonian (with the
product of $C^\infty$ topologies). We call points
$(J,H)\in\JSpace_{\cyl}$ \emph{cylindrical pairs}.

The left action of
$G$ on $M$ induces the following left action of $G$ on
$\JSpace_{\cyl}$: for $g\in G$ and $(J,H)\in\JSpace_{\cyl}$,
\[
\action{g}{(J,H)}\coloneqq (g_*\circ J\circ g_*^{-1},H\circ g^{-1}).
\]

Given a cylindrical Hamiltonian $H$, let $\Phi_H$ denote the time-one flow of $H$ and let
$L_0^H=\Phi_H(L_0)$ denote the image of $L_0$ under $\Phi_H$.

For $n\geq 1$, by an \emph{$n$-interval, $1$-story marked line} we mean a sequence $p_1\leq
q_1\leq p_2\leq q_2\leq\cdots\leq p_n\leq q_n$ of points in $\RR$,
modulo overall translation. Let $\MarkLine[n]$ denote
the space of $1$-story $n$-interval marked lines, topologized as
$\RR^{2n}/\RR$. 
The \emph{difference coordinates} $d_1,\dots,d_{2n-1}$
on a $1$-story marked line $(p_1,\dots,q_n)$ are $d_1=q_1-p_1$, $d_2=p_2-q_1$,
\dots, $d_{2n-1}=q_n-p_n$.  By a \emph{(multistory) $n$-interval marked
  line} $\Line$ we mean a finite sequence of $1$-story marked lines
$\Line_1,\dots,\Line_k$, where $\Line_i$ has $n_i$ intervals and
$n=n_1+\cdots+n_k$. We define difference coordinates on $\Line$, lying
in $[0,\infty]$, by setting $d_1=d_1(\Line_1)$, $d_2=d_2(\Line_1)$,
\dots, $d_{2n_1}=\infty$, $d_{2n_1+1}=d_1(\Line_2)$, \dots. Giving
$[0,\infty]$ the topology of a closed interval, these difference
coordinates induce a topology on $\oMarkLine[n]$, the space of
$n$-interval marked lines.

Define the marked lines with difference coordinates
$(d_1,\dots,d_{i-1},0,d_{i+1},\dots,d_{2n-1})$ in $\oMarkLine[n]$ and
$(d_1,\dots,d_{i-1}+d_{i+1},\dots,d_{2n-1})$ in $\oMarkLine[n-1]$ to
be equivalent, provided that at least one of $d_{i-1}$ and $d_{i+1}$
is finite. By $\oMarkLine$ we mean the set of equivalence classes of
marked lines, with the quotient topology.

We say that $\gluing(\Line)\in \MarkLine[n]$ is \emph{a gluing} of
$\Line\in\oMarkLine[n]$ if the difference coordinates for $\gluing(\Line)$ are
obtained from the difference coordinates for $\Line$ by replacing
the $\infty$s by positive real numbers. We call these new positive real
numbers \emph{gluing parameters}.

Given a $1$-story marked line $\Line=(p_1,q_1,\dots,p_n,q_n)\in\MarkLine[n]$, by a
\emph{$1$-story semicylindrical pair of shape $\Line$} we mean a
smooth, one-parameter family of cylindrical pairs $\JH\co \RR\to\JSpace_{\cyl}$ so
that $\wt{J}$ and $\wt{H}$ are constant (translation-invariant) on
each of the intervals $(-\infty,p_1]$, $[q_i,p_{i+1}]$, $i=1,\dots,n-1$,
and $[q_n,\infty)$. By a \emph{$1$-story
  semicylindrical pair} we mean a $1$-story semicylindrical pair of
some shape. Endow the set of $1$-story semicylindrical pairs with the
quotient topology of the $C^\infty$-topology on pairs $(\wt{J},\wt{H})$ of a fiberwise almost
complex structure and a smooth function on the $M$-bundle $[0,1]\times \RR\times M$, by the action of $\RR$
by translation. If $\Line$ consists of a single interval and that
interval has length $0$, then a semicylindrical pair of shape $\Line$
is a cylindrical pair.

More generally, given a marked line
$\Line=(\Line_1,\dots,\Line_k)\in\oMarkLine$, by a \emph{(multistory)
  semicylindrical pair of shape $\Line$} we mean a sequence
$(\wt{J}_1,\wt{H}_1),\dots,(\wt{J}_k,\wt{H}_k)$ of $1$-story
semicylindrical pairs of shapes $\Line_1,\dots,\Line_k$ so that the
value of $(\wt{J}_i,\wt{H}_i)$ near $+\infty$ agrees with the value of
$(\wt{J}_{i+1},\wt{H}_{i+1})$ near $-\infty$ for
$i=1,\dots,k-1$. Given a semicylindrical pair $\JH$ of shape $\Line$
and a gluing $\gluing(\Line)$ of $\Line$ there is an induced semicylindrical
pair $\gluing\JH$ of shape $\gluing(\Line)$. Define a subbasis for the topology on the space of semicylindrical pairs to consist of the sets $V_U$ where $U$
is an open set of $1$-story semicylindrical pairs and
\[
V_U = \{\JH\mid \exists \text{ gluing }\gluing\text{ such that }\gluing\JH\in U\}.
\]

Let $\JSpace$ denote the set of pairs $(\Line,\JH)$ where $\Line$
is a marked line and $\JH$ is a semicylindrical pair of
shape $\Line$. Give $\JSpace$ the product of the topology on the
space of semicylindrical pairs and the topology on $\oMarkLine$.

Given a semicylindrical pair
$\JH=\bigl((\wt{J}_1,\wt{H}_1),\dots,(\wt{J}_k,\wt{H}_k)\bigr)$ let
$\JH^{-\infty}=(J^{-\infty},H^{-\infty})$ denote the value of
$(\wt{J}_1,\wt{H}_1)$ for $t\ll0\in\RR$ and let
$\JH^{+\infty}=(J^{+\infty},H^{+\infty})$ denote the value of
$(\wt{J}_k,\wt{H}_k)$ for $t\gg0\in\RR$, so $\JH^{-\infty}$ and
$\JH^{+\infty}$ are cylindrical pairs.
There is a partially-defined composition $\circ$ on 
semicylindrical pairs: given semicylindrical pairs
$\JH=\bigl((\wt{J}_1,\wt{H}_1),\dots,(\wt{J}_k,\wt{H}_k)\bigr)$ and
$\JHp=\bigl((\wt{J}'_1,\wt{H}'_1),\dots,(\wt{J}'_{k'},\wt{H}'_{k'})\bigr)$ with
$\JH^{+\infty}=\JHp^{-\infty}$ we can form a new
semicylindrical pair $\JHp\circ \JH$ by concatenation of sequences,
\[
\JHp\circ \JH=\bigl((\wt{J}_1,\wt{H}_1),\dots,(\wt{J}_k,\wt{H}_k),(\wt{J}'_1,\wt{H}'_1),\dots,(\wt{J}'_{k'},\wt{H}'_{k'})\bigr).
\]

\begin{definition}
  We call a cylindrical pair $(J,H)$ \emph{generic} if:
  \begin{enumerate}[label=($G$-\arabic*)]
  \item $L_0^H$ is transverse to $L_1$ and
  \item for any $x,y\in L_0^H\cap L_1$ and $\phi\in\pi_2(x,y)$ a
    homotopy class of Whitney disks connecting $x$ to $y$ with Maslov
    index $\mu(\phi)\leq 1$ the moduli space $\cM(\phi;J,H)$ of
    $J$-holomorphic disks in the homotopy class $\phi$ is
    transversely cut out by the $\dbar$ equation.
  \end{enumerate}
  We call $(J,H)$ \emph{very generic} if $(J,H)$ is generic and the second
  condition also holds for homotopy classes $\phi$ with $\mu(\phi)=2$.
\end{definition}
(Below, we will show that one can always find very generic pairs. The
weaker condition of being generic is useful in some computations, such
as in~\cite[Section 5.2]{HLS:HEquivariant}.)

Given a $1$-story semicylindrical pair $\JH$ we have a family of Lagrangian submanifolds
$L_0^{H_t}$, parameterized by $t\in\RR$, agreeing with
$L_0^{H^{-\infty}}$ for $t\ll0$ and $L_0^{H^{+\infty}}$ for $t\gg
0$. Given $x\in L_0^{H^{-\infty}}\cap L_1$ and $y\in
L_0^{H^{+\infty}}\cap L_1$ we can consider the set of Whitney disks
with dynamic boundary on $L_0^{H_t}$ and $L_1$. We denote this set
$\pi_2(x,y)$. Given $\phi\in\pi_2(x,y)$ there is a moduli space
$\cM(\phi;\wt{J},\wt{H})$ of $\wt{J}$-holomorphic strips with dynamic
Lagrangian boundary conditions $L_0^{H_t}$ and $L_1$ in the homotopy
class $\phi$. 

More generally, given a semicylindrical pair
$\JH=\bigl((\wt{J}_1,\wt{H}_1),\dots,(\wt{J}_k,\wt{H}_k)\bigr)$ and points
$x_1\in L_0^{H^{-\infty}}\cap L_1$, $x_{k+1}\in L_0^{H^{+\infty}}\cap
L_1$ we can consider the set of homotopy classes of Whitney disks
$\pi_2(x_1,x_{k+1})$ with boundary specified by $\wt{H}$. The easiest
way to define this space is to choose a gluing $\gluing\JH$ of $\JH$ so that
$\gluing\JH$ is a $1$-story semicylindrical pair, and define
$\pi_2(x_1,x_{k+1})$ to be the set of homotopy classes of Whitney
disks in this gluing; there are canonical bijections between these sets for different choices of gluings. Given a sequence of points $x_i\in
L_0^{H_i^{-\infty}}\cap L_1$, $i=1,\dots,k$, $x_{k+1}\in L_0^{H_k^{+\infty}}\cap L_1$, and homotopy classes
$\phi_i\in\pi_2(x_i,x_{i+1})$, $i=1,\dots,k$, (pre)gluing the homotopy
classes $\phi_i$ gives a homotopy class
$\phi=\phi_1*\cdots*\phi_k\in\pi_2(x_1,x_{k+1})$.

Given $\phi\in\pi_2(x_1,x_{k+1})$ there is a corresponding moduli
space 
\[
\cM(\phi;\wt{J},\wt{H})=\coprod_{\phi_1*\cdots*\phi_k=\phi}\cM(\phi_1;\wt{J}_1,\wt{H}_1)\times\cdots\times\cM(\phi_k;\wt{J}_k,\wt{H}_k).
\]

The moduli spaces $\cM(\phi;\wt{J},\wt{H})$ have compactifications by
allowing disks at intermediate, cylindrical levels between
$(\wt{J}_i,\wt{H}_i)$ and $(\wt{J}_{i+1},\wt{H}_{i+1})$, as well as
below $(\wt{J}_1,\wt{H}_1)$ and above $(\wt{J}_k,\wt{H}_k)$. To make
this precise, given a semicylindrical pair
$\JH=((\wt{J}_1,\wt{H}_1),\dots,(\wt{J}_k,\wt{H}_k))$ we say that the
semicylindrical pair
\[
\JHp=((\wt{J}_1,\wt{H}_1),\dots,(\wt{J}_i,\wt{H}_i),(J,H),(\wt{J}_{i+1},\wt{H}_{i+1}),\dots,(\wt{J}_k,\wt{H}_k))
\]
is an \emph{elementary expansion} of $\JH$ if $(J,H)$ is a cylindrical
pair (a pair with shape a $1$-story marked line with a single interval with length
$0$). We call $\JHp$ an
\emph{expansion} of $\JH$ if $\JHp$ is obtained from $\JH$ by some
sequence of elementary
expansions.
(In other words, $\JHp$ is an expansion of $\JH$ if $\JH$ is obtained
from $\JHp$ by forgetting some $\RR$-invariant levels in $\JHp$.)
Note that if $\JHp$ is an expansion of $\JH$ and $x\in
L_0^{H^{-\infty}}\cap L_1$, $y\in L_0^{H^{+\infty}}\cap L_1$ then
there is an identification between the homotopy classes of Whitney
disks $\pi_2(x,y)$ connecting $x$ and $y$ with respect to $\JH$ and
$\JHp$. Define
\[
  \ocM(\phi;\wt{J},\wt{H})=\coprod_{\JHp\text{ an expansion of }\JH}\cM(\phi;\wt{J}',\wt{H}').
\]
Here, we do not allow $\RR$-invariant (trivial) disks in the moduli
spaces on the right, or equivalently we consider only stable curves.
There is a natural topology on $\ocM(\phi;\wt{J},\wt{H})$ which we leave to
the reader to spell out. With this topology, the spaces $\ocM(\phi;\wt{J},\wt{H})$ are, in fact, compact; see the references following Theorem~\ref{thm:compactness}.  In particular, Hypothesis~\ref{hyp:Floer-defined} implies that we do not need to consider sphere or disk bubbles in our moduli spaces.

\subsection{Twisted homotopy classes}\label{sec:twisted-pi2}
Fix, for the rest of this section, a generic cylindrical pair $(J,H)$.

Fix $g\in G$ and $x,y\in L_0^H\cap L_1$. We define a set
$\pi_2^g(x,y)$ of \emph{$g$-twisted homotopy classes of Whitney disks}
as follows. Choose a semi-cylindrical Hamiltonian $\wt{H}$ so that
$\wt{H}^{-\infty}=H$ and $\wt{H}^{+\infty}=gH$. There is a
corresponding $1$-parameter family $L_0^{\wt{H}}\subset \RR\times M$,
  where $L_0^{\wt{H}}\cap (\{t\}\times M)$ is $L_0^{\wt{H}_t}$. Let
  $\pi_2^g(x,y)$ be the set of homotopy classes of sections
\begin{equation}\label{eq:pi-2-g}
  ([0,1]\times\RR,\{0\}\times\RR,\{1\}\times\RR)\to \bigl([0,1]\times\RR\times M, \{0\}\times\RR\times L_1, \{1\}\times L_0^{\wt{H}}\bigr)
\end{equation}
which are asymptotic to $x$ (or rather, $[0,1]\times x$) at $-\infty$
and $gy$ at $+\infty$. (Note that $y\in L_0^H$ implies $gy\in L_0^{gH}$ since $gH=H\circ g^{-1}$.)

We verify that the dependence of $\pi_2^g(x,y)$ on the choice of
$\wt{H}$ is superficial. Provisionally, denote $\pi_2^g(x,y)$ defined
using $\wt{H}$ by $\pi_2^g(x,y;\wt{H})$.  Given another
semi-cylindrical $\wt{K}$ with $\wt{K}^{-\infty}=H$ and
$\wt{K}^{+\infty}=gH$ consider the $1$-parameter family of dynamic
Hamiltonians $(1-s)\wt{H}+s\wt{K}$. Given $\phi\in\pi_2^g(x,y;\wt{H})$ as in
Formula~\eqref{eq:pi-2-g}, applying the
Hamiltonian flow of $(1-s)\wt{H}+s\wt{K}$ to $\phi(1,t)$ gives a
section
\[
  ([0,1]\times\RR,\{0\}\times\RR,\{1\}\times\RR)\to \bigl([0,1]\times\RR\times M, \{0\}\times L_0^{\wt{H}}, \{1\}\times L_0^{\wt{K}}\bigr).
\]
Concatenating $\phi$ with this new section gives a section
\[
  ([0,2]\times\RR,\{0\}\times\RR,\{2\}\times\RR)\to \bigl([0,2]\times\RR\times M, \{0\}\times\RR\times L_1, \{2\}\times L_0^{\wt{K}}\bigr),
\]
and rescaling $[0,2]$ to $[0,1]$ gives a section as in
Formula~\eqref{eq:pi-2-g} except with respect to $\wt{K}$ instead of
$\wt{H}$. This gives us a map
\[
  \psi_{\wt{H},\wt{K}}\co \pi_2^g(x,y;\wt{H})\to\pi_2^g(x,y;\wt{K}).
\]

\begin{lemma}
  The maps $\psi_{\wt{H},\wt{K}}$ satisfy $\psi_{\wt{H},\wt{H}}=\Id$
  and
  $\psi_{\wt{H}_2,\wt{H}_3}\circ
  \psi_{\wt{H}_1,\wt{H}_2}=\psi_{\wt{H}_1,\wt{H}_3}$.
  So, the maps $\psi_{\wt{H},\wt{K}}$ induce canonical bijections
  $\pi_2^g(x,y;\wt{H})\to\pi_2^g(x,y;\wt{K})$.
\end{lemma}
\begin{proof}
  This is left as an exercise to the reader.
\end{proof}

Thus, we are justified in dropping $\wt{H}$ from the notation, and
simply writing $\pi_2^g(x,y)$.

Concatenation (in the $\RR$-direction) gives an operation
\[
  *\from\pi_2^g(x,y)\times\pi_2^h(y,z)\to \pi_2^{gh}(x,z).
\]

Next, consider two points $g_0,g_1\in G$ and a path $g_t$ in
$G$ from $g_0$ to $g_1$.  For $y\in L_0^H\cap L_1$ there is a
corresponding element $\phi_y\in\pi_2^{g_0^{-1}g_1}(y,y)$ defined as
follows. View $g_t$ as a map $\RR\to G$ with $g_t=g_0$ for $t<0$ and
$g_t=g_1$ for $t>1$. Then there is a corresponding semi-cylindrical
Hamiltonian $\wt{H}=g_0^{-1}g_tH$ and we define $\phi_y(s,t)=g_0^{-1}g_ty\in\pi_2^{g_0^{-1}g_1}(y,y)$.

Splicing with the $\phi_y$ defines a bijection
\begin{align*}
  \xi_{g_t}&\co \pi_2^{g_0}(x,y)\to \pi_2^{g_1}(x,y)\\
  \xi_{g_t}(\phi)&=\phi*\phi_y.
\end{align*}
\begin{lemma}
  If $g_t$ and $g'_t$ are homotopic rel endpoints then
  $\xi_{g_t}=\xi_{g'_t}$.
\end{lemma}
\begin{proof}
  Again, this is straightforward from the definitions, and is left as
  an exercise.
\end{proof}

Given a twisted homotopy class $\phi\in\pi_2^g(x,y)$ there is a
corresponding Maslov index $\mu(\phi)$. Further, $\mu$ is additive in
the sense that $\mu(\phi_1*\phi_2)=\mu(\phi_1)+\mu(\phi_2)$.

\begin{definition}\label{def:G-twisted-loop}
  A \emph{$G$-twisted loop} is a homotopy class
  $\phi\in\pi_2^g(x,x)$ for some $x\in L_0^H\cap L_1$ and $g\in G$.
\end{definition}

\begin{lemma}\label{lem:mu-on-pi2-g}
  If $G$ is compact and
  Hypothesis~\ref{hyp:Floer-defined}~\ref{item:J-1} is satisfied then
  any $G$-twisted loop $\phi\in\pi_2^g(x,x)$ has Maslov index $0$.
\end{lemma}
\begin{proof}
  For any neighborhood $U\ni 1$ in $G$, there is some $n$ so that
  $g^{-n}\in U$. For $U$ a small enough neighborhood, there is a (nearly
  constant) Maslov index $0$ disk $\phi_0$ in $\pi_2^{g^{-n}}(x,x)$. Now, given
  $\phi\in\pi_2^g(x,x)$ we have
  \[
    \mu(\overbrace{\phi*\phi*\cdots*\phi}^n*\phi_0)=n\mu(\phi),
  \]
  but the left side is an element of $\pi_2^{1}(x,x)$ and so, by
  Hypothesis~\ref{hyp:Floer-defined}~\ref{item:J-1}, has Maslov
  index $0$.
\end{proof}

\subsection{Equivariant coherent orientations}\label{sec:G-or}
We continue to fix a generic cylindrical pair $(J,H)$.
In this section, we adapt the notions of coherent
orientations~\cite{FloerHofer93:orientations,EES05:orientations,FOOO1,SeidelBook} to the
$G$-equivariant setting. (We explain what happens when one varies $(J,H)$ or considers a semi-cylindrical pair at the end of this section.) We start by reviewing the
non-equivariant situation.

Fix a generic, cylindrical pair $(J,H)$. Given points
$x,y\in L_0^H\cap L_1$, let $\mathcal{B}_{(J,H)}(x,y)$ denote an
appropriate Sobolev space of maps as in
Formula~\eqref{eq:pi-2-g}. (See also the proof of
Theorem~\ref{thm:transversality}, below.) The $\dbar$-operator is a
section of a vector bundle $\mathcal{E}_{(J,H)}$ over
$\mathcal{B}_{(J,H)}$.  There is a virtual vector bundle
$\ind(D\dbar)$ over $\mathcal{B}_{(J,H)}(x,y)$, the \emph{index
  bundle} of the family of Fredholm maps $D\dbar$. If $\dbar$ is
transverse to the $0$-section then the restriction of $\ind(D\dbar)$
to the subspace $\cM_{(J,H)}(x,y)$ of holomorphic disks is identified
with the tangent space $T\cM_{(J,H)}(x,y)$. So, orienting the
determinant line bundle $\det(D\dbar)$ (which is an honest line
bundle) orients the moduli spaces $\cM_{(J,H)}(x,y)$, when they are
transversely cut out.

Given $u\in\mathcal{B}_{(J,H)}(x,y)$ and
$v\in\mathcal{B}_{(J,H)}(y,z)$ and a parameter $R\gg0$ large enough
that $u([0,1]\times[R,\infty))$ and $v([0,1]\times(-\infty,-R])$ are
contained in a ball neighborhood of $y$ we can use bump functions to
\emph{pre-glue} $u$ and $v$ to get a curve
$(u\natural_R v)\in\mathcal{B}_{(J,H)}(x,z)$.  We view pre-gluing as a
map
\[
  \gamma\co \mathcal{B}_{(J,H)}(y,z)\times \mathcal{B}_{(J,H)}(x,y)\to
  \mathcal{B}_{(J,H)}(x,z);
\]
note the ordering of the factors, which we have chosen to be
consistent with composition in a category, and not with the
concatenation operator $*$.  Pre-gluing is covered by bundle maps
\begin{align*}
  T_v\mathcal{B}_{(J,H)}(y,z)\oplus T_u\mathcal{B}_{(J,H)}(x,y) &\to
                                                                T_{u\natural_R v}\mathcal{B}_{(J,H)}(x,z)\\
  \mathcal{E}_{(J,H),v}\oplus \mathcal{E}_{(J,H),u}&\to\mathcal{E}_{(J,H),u\natural_Rv}.
\end{align*}
The two paths around the following diagram agree modulo terms of
lower order:
\[
  \xymatrix{
    T_v\mathcal{B}_{(J,H)}(y,z)\oplus T_u\mathcal{B}_{(J,H)}(x,y)
    \ar[rr]^-{\text{pre-glue}}
    \ar[d]^{D\dbar} && T_{u\natural_R
      v}\mathcal{B}_{(J,H)}(x,z)\ar[d]^{D\dbar}\\
    \mathcal{E}_{(J,H),v}\oplus\mathcal{E}_{(J,H),u}\ar[rr]^-{\text{pre-glue}}
    & & \mathcal{E}_{(J,H),u\natural_Rv}.
  }
\]
This gives an identification of the determinant lines of the
Fredholm maps $D_{u\natural_R v}\dbar|_{\mathcal{B}_{(J,H)}(x,z)}$ and
$(D_v\dbar|_{\mathcal{B}_{(J,H)}(y,z)})\oplus (D_u\dbar|_{\mathcal{B}_{(J,H)}(x,y)})$.

By \emph{coherent orientations} of the moduli spaces we mean
trivializations of the determinant line bundles
\[
\phi_{x,y}\co \det(D\dbar)|_{\mathcal{B}_{(J,H)}(x,y)}\stackrel{\cong}{\longrightarrow}\RR\times
\mathcal{B}_{(J,H)}(x,y)=\underline{\RR}
\]
so that the diagram
\[
  \xymatrix{
    \det(D\dbar)|_{\mathcal{B}(y,z)}\otimes \det(D\dbar)|_{\mathcal{B}(x,y)}\ar[rr]^-{\mathrm{pre-glue}} \ar[d]_{\phi_{y,z}\otimes\phi_{x,y}} & & \gamma^*\det(D\dbar)|_{\mathcal{B}(x,z)}\ar[d]^{\phi_{x,z}}\\
    \underline{\RR}\otimes\underline{\RR}\ar[rr]^-{\cong} & & \underline{\RR}
}
\]
commutes up to positive scaling. (Here, the bottom horizontal arrow is the standard
isomorphism $1\otimes 1\mapsto 1$.) A coherent
orientation induces orientations of the moduli spaces of holomorphic
disks, and these orientations are coherent with respect to
gluing. (See also the proof of Theorem~\ref{thm:gluing}.)

The $G$-equivariant analogue of coherent orientations is as follows. 

Given $x,y\in L_0^H\cap L_1$ and $g\in G$ we can consider the space
$\mathcal{B}_g(x,y)$ of smooth Whitney disks connecting $x$ and $g y$, with
respect to the Hamiltonian isotopy from $L_0^H$ to
$L_0^{gH}$ coming from the inverse of the Hamiltonian isotopy
induced by $H$ followed by the Hamiltonian isotopy corresponding
to $gH$, $L_0^{H}\sim L_0\sim L_0^{gH}$ (cf.\ Section~\ref{sec:twisted-pi2}). The spaces
$\mathcal{B}_g(x,y)$ fit together into a fiber bundle
\[
  \xymatrix{
    \mathcal{B}_g(x,y)\ar[r] & \mathcal{B}_G(x,y)\ar[d]\\
    & G.
  }
\]
At each $g\in G$ and each point $u\in \mathcal{B}_g(x,y)$ we have a linearized
$\dbar$ operator $D_u\dbar$. The index bundles of these operators
fit together to give a virtual index bundle
\[
  \xymatrix{
    \ind(D_u\dbar)=[\Ker(D_u\dbar)-\Coker(D_u\dbar)]\ar[r] & \ind(D\dbar)\ar[d]\\
    & \mathcal{B}_G(x,y).
  }
\]
The determinant line bundle of this virtual index bundle is
$\det(D\dbar)\coloneqq \det(\ind(D\dbar))$.

As in the non-equivariant case, pre-gluing $\natural$ gives a map
$\gamma\co \mathcal{B}_h(y,z)\times \mathcal{B}_g(x,y)\to \mathcal{B}_{gh}(x,z)$,
$(v,u)\mapsto u\natural (g v)$, which is covered by an
identification of bundles
$\psi_{x,y,z}\co \ind(D\dbar)|_{\mathcal{B}_h(y,z)}\otimes
\ind(D\dbar)|_{\mathcal{B}_g(x,y)}\to \gamma^*\ind(D\dbar)|_{\mathcal{B}_{gh}(x,z)}$.

\begin{definition}\label{def:G-or-system}
  By a \emph{$G$-orientation system} we mean a trivialization
  $\phi_{x,y}\co \det(D\dbar)\cong\underline{\RR}$ over each $\mathcal{B}_G(x,y)$ so that for
  each $[g],[h]\in\pi_0(G)$ and $x,y,z\in L_0^H\cap L_1$, the trivializations 
  \[
    \det(D\dbar)|_{\mathcal{B}_h(y,z)}\otimes\det(D\dbar)|_{\mathcal{B}_g(x,y)}\stackrel{\phi_{y,z}\otimes\phi_{x,y}}{\relbar\joinrel\relbar\joinrel\relbar\joinrel\relbar\joinrel\longrightarrow} \underline{\RR}\otimes\underline{\RR}\cong \underline{\RR}
  \]
  and
  \[
    \det(D\dbar)|_{\mathcal{B}_h(y,z)}\otimes\det(D\dbar)|_{\mathcal{B}_g(x,y)}\stackrel{\psi_{x,y,z}}{\relbar\joinrel\longrightarrow}\gamma^*\det(D\dbar)|_{\mathcal{B}_{gh}(x,z)}\stackrel{\phi_{x,z}}{\relbar\joinrel\longrightarrow}\underline{\RR}
  \]
  agree up to positive scaling.
\end{definition}

\begin{figure}
  \centering
  \includegraphics{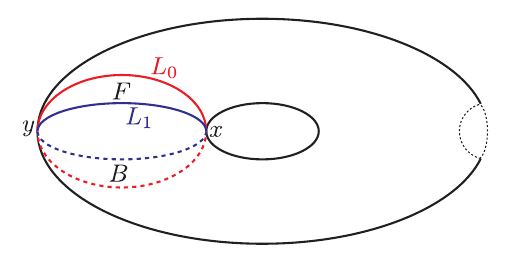}
  \caption{\textbf{Lagrangians in a punctured torus.} The two regions
    supporting holomorphic bigons are labeled $F$ and $B$, for front
    and back.}
  \label{fig:Lag-in-torus}
\end{figure}

\begin{example}\label{eg:coh-or}
  (Compare~\cite[Example 13.5]{SeidelBook}.)
  Consider the Lagrangians shown in
  Figure~\ref{fig:Lag-in-torus}. Here, the symplectic manifold
  $(M,\omega)$ is a punctured torus, with an action of an involution
  $\tau$ induced by the elliptic involution of $T^2$, with the
  puncture as one of the fixed points. The Lagrangians $L_0$, $L_1$
  are homologically nontrivial circles which pass through two of the
  fixed points, and so that $L_0\pitchfork L_1$ at exactly these fixed
  points. Call the points in $L_0\pitchfork L_1$ $x$ and $y$, labeled
  so that the moduli space of holomorphic disks from $x$ to $y$ has
  exactly two elements. Call the two regions supporting these bigons
  $F$ and $B$, for front and back.

  Each of the spaces $\mathcal{B}(x,x)$, $\mathcal{B}(x,y)$,
  $\mathcal{B}(y,x)$, and $\mathcal{B}(y,y)$ is homotopy equivalent to
  $\ZZ$; the homology classes in $H_2(M,L_0\cup L_1)$ corresponding
  to these elements are:
  \begin{align*}
    \mathcal{B}(x,x)&:\ \{nF-nB\mid n\in\ZZ\} 
    & \mathcal{B}(x,y)&:\ \{(n+1)F-nB\mid n\in\ZZ\}\\
    \mathcal{B}(y,x)&:\ \{(n-1)F-nB\mid n\in\ZZ\} 
    & \mathcal{B}(y,y)&:\ \{nF-nB\mid n\in\ZZ\}.
  \end{align*}
  A choice of coherent orientation is uniquely determined by the
  trivialization over the component of
  $\ind(D\dbar)|_{\mathcal{B}(x,x)}$ corresponding to $F-B$, say, and
  a trivialization of $\ind(D\dbar)$ over the component of
  $\mathcal{B}(x,y)$ corresponding to $F$. As is well known, there are
  two different Floer homologies that can result, coming from whether
  the holomorphic disks from $x$ to $y$ are given the same or opposite
  signs.
  
  For $G=\ZZ/2=\{e,\tau\}$, since the Lagrangians already
  intersect transversely we may choose the Hamiltonian $H=0$. 
  Choosing a $G$-orientation involves
  choosing compatible orientations over $\mathcal{B}_e(x,x)$,
  $\mathcal{B}_\tau(x,x)$, $\mathcal{B}_e(x,y)$,
  $\mathcal{B}_\tau(x,y)$, and so on. Writing $B_g(x,y,R)$
  for the component of the configuration space $\mathcal{B}_g(x,y)$
  with homology class $R$, pre-gluing gives maps
  \begin{align*}
    B_e(x,x,mF-mB)\times B_e(x,x,nF-nB)&\to B_e(x,x,(m+n)F-(m+n)B)\\
    B_\tau(x,x,mF-mB)\times B_e(x,x,nF-nB)&\to
    B_\tau(x,x,(m+n)F-(m+n)B)\\
    B_e(x,x,mF-mB)\times B_\tau(x,x,nF-nB)&\to
    B_\tau(x,x,(n-m)F-(n-m)B)\\
    B_\tau(x,x,mF-mB)\times B_\tau(x,x,nF-nB)&\to
    B_e(x,x,(n-m)F-(n-m)B)
  \end{align*}
  and so on. So, to choose a $G$-orientation, it suffices to choose an
  orientation over $B_e(x,x,F-B)$, over $B_\tau(x,x,0)$, and over
  $B_e(x,y,F)$.

  There are conceivable obstructions to the existence of a
  $G$-orientation. For instance, the concatenation map
  \[
    B_e(x,x,F-B)\times B_\tau(x,x,0)\times B_e(x,x,F-B)\times B_\tau(x,x,0)\to B_e(x,x,0)
  \]
  induces an orientation over $B_e(x,x,0)$, which is required to agree
  with the canonical orientation. Since each of $B_e(x,x,F-B)$ and
  $B_\tau(x,x,0)$ appears twice in each formula, either these coherence
  conditions are satisfied for every choice of orientation over
  $B_e(x,x,F-B)$ and $B_\tau(x,x,0)$, or they are satisfied for no
  choice.  In fact, it will follow from Example~\ref{eg:or-2} that the
  compatibility holds for any choice.
\end{example}

Next we recall that brane data for the Lagrangians induce coherent
orientations, and then give a $G$-equivariant analogue.
Following Seidel~\cite{SeidelBook},
assume that twice the first Chern class of $(M,\omega)$
vanishes, and fix a trivialization $\eta^2$ of $(\Lambda^n_\CC TM)^{\otimes 2}$, 
the square of the top
exterior power of $TM$ viewed as a complex bundle with respect to any
almost complex structure compatible with $\omega$. There is an induced
\emph{squared phase map} $\alpha\co \Gr_{\Lag}(TM)\to S^1$ on the space of Lagrangian
subspaces of $TM$, representing the
Maslov class. For each point $p\in L_i$ we can evaluate
$\alpha$ on $T_pL_i$; we abuse notation and denote the result
$\alpha(p)$.  \emph{Brane data} for $L_i$ consists of a
$\pin$-structure $P_i$ on $TL_i$, i.e., a principal $\Pin(n)$ bundle $P_i$ and bundle map $P_i\to \Fr(TL_i)$ so that the diagram
\[
  \xymatrix{
    P_i\times \Pin(n)\ar[r]\ar[d] & P_i\ar[d]\\
    \Fr(TL_i)\times O(n)\ar[r] & \Fr(TL_i)
  }
\]
commutes;
and a \emph{grading} of $L_i$, i.e., a lift
$\alpha_{L_i}^\#\co L_i\to\RR$ of the squared phase map
$\alpha|_{L_i}\co L_i\to S^1$~\cite[Section~(11j)]{SeidelBook}. (Here,
we will always use $\Pin(n)$ for $\Pin^+(n)$, as we have no use for
$\Pin^-(n)$.) We
will denote brane data $(P_i,\alpha_{L_i}^\#)$ on $L_i$ by
$\brane_i$.
 
For the moment, suppose $L_0\pitchfork L_1$. Fix brane data
$\brane_i$ on each $L_i$. For $x\in L_0\cap L_1$,
Seidel defines a $1$-dimensional vector space $o(x)$ with the property that 
\begin{equation}\label{eq:o-spaces}
  \det(D\dbar|_{\mathcal{B}(x,y)})\cong (o(y)\otimes o(x)^*)\times\mathcal{B}(x,y),
\end{equation}
as vector bundles~\cite[Equation (12.2)]{SeidelBook}.  

The space $o(x)$ is defined as follows. By a \emph{phased symplectic
  vector space} we mean a symplectic vector space $(V,\omega)$
together with an identification
$\eta^2\co (\Lambda^n_\CC
V)^{\otimes 2}\stackrel{\cong}{\longrightarrow}\CC$. Given a Lagrangian subspace
$L$ of a phased symplectic vector space $(V,\omega,\eta^2)$ there is a
well defined element 
\begin{equation}\label{eq:phase}
  \alpha(L)=\frac{\eta^2\bigl((v_1\wedge\cdots\wedge
    v_n)^{\otimes 2}\bigr)}{|\eta^2\bigl((v_1\wedge\cdots\wedge
    v_n)^{\otimes 2}\bigr)|}\in S^1,
\end{equation}
where $v_1,\dots,v_n$ is a basis for $L$.  By a \emph{graded
  Lagrangian subspace} of a phased symplectic vector space
$(V,\omega,\eta^2)$ we mean a pair $(W,\alpha^\#)$ where $W$ is a
Lagrangian subspace of $V$ and $\alpha^\#\in\RR$ satisfies
$e^{2\pi i\alpha^\#}=\alpha(W)$ (i.e., is a preimage of $\alpha(W)$). Note
that $(T_xL_i,\alpha_{L_i}^\#(T_xL_i))$ is a graded Lagrangian subspace of
$(T_xM,\omega,\eta^2)$. Choose a one-parameter family
$(W_t,\alpha^\#_t)$, $t\in[0,1]$, of graded Lagrangian subspaces of
$(T_xM,\omega,\eta^2)$ so that:
\begin{itemize}
\item $(W_i,\alpha^\#_i)=(T_xL_i,\alpha_{L_i}^\#(T_xL_i))$ for $i=0,1$,
\item viewed as a path in the Lagrangian Grassmannian, $W_t$ is
  transverse to the Maslov cycle
  $\mathrm{Mas}=\{V\in\Gr_{\Lag}(T_xM)\mid V\cap T_xL_1\neq\{0\}\}$, and
\item at $t=1$ the path $W_i$ approaches $T_xL_1$ in the positive
  direction in the Lagrangian Grassmanian.
\end{itemize}
Choose also a $\pin$-structure $P$ on the bundle $\cup_t W_t$ over
$[0,1]$ and isomorphisms $P|_{W_i}\cong P_i$ (where $P_i$ is the
$\pin$-structure on $T_xL_i$). Then
\[
  o(x)=\bigotimes_{\substack{t\in[0,1)\\W_t\cap T_xL_1\neq\{0\}}}(W_t\cap
  T_xL_1)^{\sign(W_t\cdot\mathrm{Mas})},
\]
where $\sign(W_t\cdot\mathrm{Mas})$ denotes the sign of the
intersection of the path $W_t$ and the Maslov cycle at $t$. Given
a different choice of path $W'_t$ and $\pin$-structure $P'$, there is
a
canonical isomorphism between the vector spaces, $o_{W_t,P}(x)\cong
o_{W'_t,P'}(x)$, forming a transitive system. 

\begin{remark}
  While the $\pin$-structure is not used in the definition of $o(x)$,
  it is involved in determining these canonical isomorphisms. Seidel
  notes that, given $W_t$, there are essentially two choices $P$, $P'$ of
  $\pin$-structure, and that the map
  \[
    \bigotimes_{\substack{t\in[0,1)\\W_t\cap T_xL_1\neq\{0\}}}(W_t\cap
    T_xL_1)^{\sign(W_t\cdot\mathrm{Mas})}=
    o_{W_t,P}(x)\to o_{W_t,P'}(x)=\bigotimes_{\substack{t\in[0,1)\\W_t\cap T_xL_1\neq\{0\}}}(W_t\cap
    T_xL_1)^{\sign(W_t\cdot\mathrm{Mas})}
  \]
  is an orientation-reversing map from a vector space to
  itself~\cite[Remark 11.19]{SeidelBook}. Nonetheless, the colimit
  \[
    o(x)=\varinjlim_{W_t,P}o_{W_t,P}(x)
  \]
  is a single, well-defined, one-dimensional vector space.
\end{remark}

It is immediate from their definitions that gluing respects the identifications~\eqref{eq:o-spaces}, in the sense that the
diagram
\begin{equation}
  \label{eq:seidel-glue-coh}
  \begin{gathered}
    \xymatrix{
      \det(D\dbar|_{\mathcal{B}(y,z)})\otimes\det(D\dbar|_{\mathcal{B}(x,y)})\ar[r]^-{\psi_{x,y,z}}\ar[d] & 
      \gamma^*\det(D\dbar|_{\mathcal{B}(x,z)})\ar[d]\\
      o(z)\otimes o(y)^*\otimes o(y)\otimes o(x)^* \ar[r] & o(z)\otimes o(x)^*
    }
  \end{gathered}
\end{equation}
commutes. The bottom horizontal map is induced by the canonical
identification of vector spaces
$o(y)^*\otimes o(y)\cong \RR$~\cite[Formula (12.4)]{SeidelBook}.  Thus, if we choose a
trivialization $\phi_x\co o(x)\stackrel{\cong}{\longrightarrow}\RR$
for each $x\in L_0\cap L_1$ there is an induced trivialization
\[
  \phi_{x,y}\co \det(D\dbar|_{\mathcal{B}(x,y)})\stackrel{\cong}{\longrightarrow}\RR\times\mathcal{B}(x,y).
\]
It is immediate from Formula~\eqref{eq:seidel-glue-coh} that these identifications
$\phi_{x,y}$ give coherent orientations.

Sometimes it will be convenient to write $o_{\brane_0,\brane_1}(x)$
instead of $o(x)$, to indicate the brane data used to define the space
$o(x)$.

If $L_0^H$ is Hamiltonian isotopic to $L_0$ then brane data $\brane_0$
on $L_0$ together with the map $L_0\to L_0^H$ induced by
the Hamiltonian isotopy induces brane data on $L_0^H$. In
particular, choosing brane data on the $L_i$ induces
brane data on any pair of Lagrangians $L'_i$ which are related to
the $L_i$ by (chosen) Hamiltonian isotopies, and hence orientations of
the moduli spaces of holomorphic disks between such $L'_i$.

To adapt brane data to the equivariant case, we will use an extension
of Cho-Hong's notion of \emph{$\spin$ profiles} for finite group
actions~\cite[Section 5.1]{CH:group-actions} to the case of compact Lie
groups.

Assume that the Lagrangians $L_i$ are connected.  Assume also that we
are given $\pin$-structures $P_i$ on $L_i$ which are weakly
$G$-equivariant, in the sense that for any $g\in G$, the induced
$\pin$-structure $g\cdot P_i\coloneqq (g^{-1})^*P_i$
is isomorphic to $P_i$. In this case, for any
$g\in G$ the set of isomorphisms $\Iso(g\cdot P_i,P_i)$ of $\pin$-structures
from $P_i$ to $g^*P_i$ has two elements.

Recall that isomorphism classes of central extensions of $G$ by
$\ZZ/2$ correspond to elements of $H^2(BG;\ZZ/2)$ (see,
e.g.,~\cite{Segal70:group-coho}).

\begin{definition}\label{def:pin-profile}
  The \emph{$\pin$ profile} associated to $(L_i,G,P_i)$ is the class in
  $H^2(BG;\ZZ/2)$ associated to the central extension 
  \[
    \ZZ/2\to \wt{G}_i\to G
  \]
  defined as follows. The fiber of $\wt{G}_i$ over $g\in G$ is
  $\Iso(g\cdot P_i,P_i)$. Given $\alpha\in\Iso(g\cdot P_i,P_i)$ and
  $\beta\in \Iso(h\cdot P_i,P_i)$ the product $\alpha\beta$ is the induced
  isomorphism
  \[
    (gh)\cdot P_i=(h^{-1}g^{-1})^*P_i=(g^{-1})^*(h^{-1})^*P_i\stackrel{(g^{-1})^*\beta}{\longrightarrow}(g^{-1})^*P_i
    \stackrel{\alpha}{\longrightarrow}P_i
  \]
  or, more succinctly, $\alpha\circ ((g^{-1})^*\beta)$.

  If $L_i$ is $\spin$ and the $\spin$-structure $Q_i$ satisfies
  $g^*Q_i\cong Q_i$ for all $g\in G$ then the \emph{$\spin$ profile}
  associated to $(L_i,G,Q_i)$ is defined in the same way as the $\pin$ profile.
\end{definition}

The $\pin$ profile associated to $(L_i,G,P_i)$ is local to
$L_i$, in the sense that it depends only on the action of $G$ on $L_i$
and the $\pin$-structure $P_i$ on $TL_i$.

\begin{lemma}
  If $G$ is discrete then the $\spin$ profile of $(L_i,G,Q_i)$ as in
  Definition~\ref{def:pin-profile} agrees with the $\spin$ profile as
  defined by Cho-Hong~\cite[Proposition 5.6]{CH:group-actions}.
\end{lemma}
\begin{proof}
  This is immediate from Cho-Hong's reformulation of their
  definition~\cite[Proof of Proposition 5.6]{CH:group-actions}.
\end{proof}

Turning to the gradings, we will need the following assumptions:
\begin{enumerate}[label=(Gr-\arabic*)]
\item \label{hyp:phase-invt}The squared phase map $\eta^2$ on $M$ is $G$-invariant.
\item \label{hyp:gr-invt}The grading $\alpha_{L_i}^\#$ on $L_i$ is
  $G$-invariant, in the sense that
  $\alpha_{L_i}^\#(gx)=\alpha^\#_{L_i}(x)$ for all $x\in L_i$ and
  $g\in G$.
\end{enumerate}
Note that if the action of $G$ on $L_i$ has at least one fixed point
on each component of $L_i$
then Hypothesis~\ref{hyp:gr-invt} is automatically satisfied.

Next, we relate brane data to $G$-orientations.  Under
Hypothesis~\ref{hyp:phase-invt}, given $g\in G$ and brane data
$\brane_i$ on $L_i$ there are induced brane data $g\cdot\brane_i$ on
the $L_i$ given by the pin structure $g\cdot P_i=(g^{-1})^*P_i$ and
grading $(g\cdot \alpha_{L_i}^\#)(x)=\alpha_{L_i}^\#(g^{-1}x)$.
The brane data $\brane_0$ also induces brane data $\brane_0$ on
$L_0^H$ and on $L_0^{gH}$ for any $g\in G$. These brane data induced by $\brane_0$ on
$L_0^{gH}$ are not typically the same as $g$ of the brane data induced
by $\brane_0$ on $L_0^H$.

\begin{proposition}\label{prop:spin-or}
  Fix a squared phase map $\eta^2$ on $(M,\omega)$ satisfying
  Hypothesis~\ref{hyp:phase-invt}.  Let $\brane_i$ be brane data on
  $L_i$, $i=0,1$, so that:
  \begin{enumerate}
  \item The gradings satisfy Hypothesis~\ref{hyp:gr-invt}.
  \item The $\pin$ profiles associated to the $\pin$-structures $P_i$
    on $L_i$ are the same, i.e., there is an isomorphism $\iota$ of central
    extensions
    \begin{equation}\label{eq:iso-extens}
      \xymatrix{
        \ZZ/2\ar[r]\ar[d]_= & \wt{G}_{L_0}\ar[r]\ar[d]^\iota_\cong & G\ar[d]^=\\        
        \ZZ/2\ar[r] & \wt{G}_{L_1}\ar[r] & G.
        }
    \end{equation}
  \end{enumerate}
  Then $(\eta^2,\brane_0,\brane_1,\iota)$ induce a $G$-orientation (in the sense of
  Definition~\ref{def:G-or-system}).
\end{proposition}
\begin{proof}
  We claim that the isomorphism~\eqref{eq:iso-extens} induces an isomorphism
  \begin{equation}\label{eq:o-cong}
    o_{g\brane_0,g\brane_1}(x)\cong o_{\brane_0,\brane_1}(x),
  \end{equation}
  well-defined up to positive scaling. Let $\wt{g}_0,\wt{g}'_0$ be the
  preimages of $g$ in $\wt{G}_{L_0}$. The element $\wt{g}_0$ is an
  identification $\brane_0\cong g\brane_0$, and hence gives an
  identification $\brane_0|_{T_xL_0}\cong g\brane_0|_{T_{x}L_0}$.
  (This uses Hypothesis~\ref{hyp:gr-invt}, which says that the
  gradings are already identified.)  Similarly, $\iota(\wt{g}_0)$
  gives an identification
  $\brane_1|_{T_xL_1}\cong g\brane_1|_{T_{x}L_1}$. These two
  identifications then give the identification~\eqref{eq:o-cong}, as
  the orientation spaces are defined purely in terms of the brane
  data.

  The key point is that if we instead used $\wt{g}'_0$ and
  $\iota(\wt{g}'_0)$ we would get the same identification of
  orientation spaces: replacing $\wt{g}_0$ by $\wt{g}'_0$ and
  replacing $\iota(\wt{g}_0)$ by $\iota(\wt{g}'_0)$ each changes the
  sign of the identification~\eqref{eq:o-cong}. Further, continuity of
  $\iota$ implies that these identifications are continuous in $g$.
  
  Given group elements $g,h\in G$, Formula~\eqref{eq:o-cong} gives a commutative diagram
  \begin{equation}
    \label{eq:o-cong-1-comm}
    \begin{gathered}
      \xymatrix{
        o_{\brane_0,\brane_1}(x)\ar[r]^\cong_h \ar[dr]_\cong^-{gh} &o_{h\brane_0,h\brane_1}(x)\ar[d]^\cong_{g}\\
        & o_{gh\brane_0,gh\brane_1}(x).  }
    \end{gathered}
  \end{equation}

  There is also an identification
  \begin{equation}
    \label{eq:o-cong-2}
    o_{g\brane_0,g\brane_1}(gx)\cong o_{\brane_0,\brane_1}(x),    
  \end{equation}
  coming from naturality of the orientation spaces. The obvious
  analogue of Diagram~\eqref{eq:o-cong-1-comm} commutes, of course, as
  does the diagram
  \begin{equation}
    \label{eq:o-cong-2-comm}
    \begin{gathered}
      \xymatrix{
        o_{\brane_0,\brane_1}(x)\ar[r]^-h\ar[d]_{g} & o_{h\brane_0,h\brane_1}(x)\ar[d]_g\\
        o_{g\brane_0,g\brane_1}(gx)\ar[r]_-{ghg^{-1}} &
        o_{gh\brane_0,gh\brane_1}(gx), }
    \end{gathered}
  \end{equation}
  where the vertical arrows use the isomorphism~(\ref{eq:o-cong-2})
  and the horizontal arrows use the isomorphism~(\ref{eq:o-cong}).

  Fix an orientation
  $\phi_x\co
  o_{\brane_0,\brane_1}(x)\stackrel{\cong}{\longrightarrow}\RR$
  for each $x\in L_0^H\cap L_1$. These choices give an orientation of
  $\mathcal{B}_g(x,y)$ via the sequence of identifications
  \begin{align*}
    \det(D\dbar|_{\mathcal{B}_g(x,y)})
    &\cong (o_{\brane_0,\brane_1}(gy)\otimes o_{\brane_0,\brane_1}(x)^*)\times\mathcal{B}_g(x,y)\\
    &\cong (o_{g\brane_0,g\brane_1}(gy)\otimes o_{\brane_0,\brane_1}(x)^*)\times\mathcal{B}_g(x,y)\\
    &\cong (o_{\brane_0,\brane_1}(y)\otimes o_{\brane_0,\brane_1}(x)^*)\times\mathcal{B}_g(x,y)\\
    &\cong \underline{\RR}\otimes\underline{\RR}\cong\underline{\RR},
  \end{align*}
  where the isomorphisms use Seidel's result~\cite[Equation
  (12.2)]{SeidelBook} (with respect to a one-parameter family of
  Hamiltonians connecting $H$ and $gH$), Formula~(\ref{eq:o-cong}),
  Formula~(\ref{eq:o-cong-2}), and the orientations $\phi_x$ and
  $\phi_y$, respectively. These identifications are continuous in
  $g$, so globalize to give an orientation of $\det(D\dbar|_{B_G(x,y)})$.

  Before verifying
  coherence of these orientations (Definition~\ref{def:G-or-system}), note that given $g,h\in G$ and
  $y,z\in L_0^H\cap L_1$, instead of using the orientation defined
  above on $\det(D\dbar|_{\mathcal{B}_h(y,z)})$ we could introduce an
  action of $g$, via
  \begin{equation}\label{eq:no-twist-or}
    o_{\brane_0,\brane_1}(hz)\otimes o_{\brane_0,\brane_1}(y)^*\cong o_{g\brane_0,g\brane_1}(ghz)\otimes o_{g\brane_0,g\brane_1}(gy)^*\cong o_{\brane_0,\brane_1}(ghz)\otimes o_{\brane_0,\brane_1}(gy)^*,
  \end{equation}
  where the first isomorphism comes from the action of $g$ on $(M,L_0,L_1)$ as in Formula~(\ref{eq:o-cong-2}) and the second uses Formula~(\ref{eq:o-cong}).
  However, commutativity of Diagrams~(\ref{eq:o-cong-1-comm})
  and~(\ref{eq:o-cong-2-comm}) implies that this gives the same
  orientation of $\det(D\dbar|_{\mathcal{B}_h(y,z)})$.

  To verify coherence, consider the diagram
{\small
  \begin{equation}
    \label{eq:seidel-glue-coh-G}
    \begin{gathered}
      \xymatrix{
        \det(D\dbar|_{\mathcal{B}_h(y,z)})\otimes\det(D\dbar|_{\mathcal{B}_g(x,y)})\ar[r]^-{\psi_{x,y,z}}\ar[d] & 
        \gamma^*\det(D\dbar|_{\mathcal{B}_{gh}(x,z)})\ar[d]\\
        {\begin{aligned}o_{\brane_0,\brane_1}(ghz)&\otimes o_{\brane_0,\brane_1}(gy)^*\\\otimes o_{\brane_0,\brane_1}(gy)&\otimes o_{\brane_0,\brane_1}(x)^* \end{aligned}}\ar[r]\ar[d] & o_{\brane_0,\brane_1}(ghz)\otimes o_{\brane_0,\brane_1}(x)^*\ar[d]\\
        {\begin{aligned}o_{gh\brane_0,gh\brane_1}(ghz)&\otimes o_{g\brane_0,g\brane_1}(gy)^*\\\otimes o_{g\brane_0,g\brane_1}(gy)&\otimes o_{\brane_0,\brane_1}(x)^*\end{aligned}} \ar[r]\ar[d]& o_{gh\brane_0,gh\brane_1}(ghz)\otimes o_{\brane_0,\brane_1}(x)^*\ar[d]\\
        o_{\brane_0,\brane_1}(z)\otimes o_{\brane_0,\brane_1}(y)^*\otimes o_{\brane_0,\brane_1}(y)\otimes o_{\brane_0,\brane_1}(x)^* \ar[r]\ar[d]^{\phi_z\otimes\phi_y\otimes\phi_y\otimes\phi_x}& o_{\brane_0,\brane_1}(z)\otimes o_{\brane_0,\brane_1}(x)^*\ar[d]^{\phi_z\otimes\phi_x}\\
        {\RR}\otimes{\RR}\otimes{\RR}\otimes{\RR}\ar[r] & {\RR}\otimes{\RR}.
      }
    \end{gathered}
  \end{equation}
}
  It follows from the same argument as in the non-invariant
  case~\cite[Formula (12.4)]{SeidelBook} that the top square commutes
  up to positive scaling: this is simply gluing moduli spaces defined
  with respect to a variety of Hamiltonian perturbations of $L_0$. The
  vertical arrows in the second square come from the
  identification~\eqref{eq:o-cong}. The vertical arrows in the third
  square come from the identification~\eqref{eq:o-cong-2}, and the
  fourth square uses our chosen orientations $\phi_x$.  The second,
  third, and fourth squares all obviously commute, as we are using the
  same identifications on the two sides. Note that the left hand
  column is not the definition of the orientation on
  $\det(D\dbar|_{\mathcal{B}_h(y,z)})$, but rather is that orientation
  twisted by $g$ in the sense of Formula~(\ref{eq:no-twist-or}); but
  as discussed there this in fact gives the same orientation.

  Thus, the orientations we have chosen give a $G$-orientation, as
  desired.
\end{proof}

\begin{example}\label{eg:or-2}
  Continuing with Example~\ref{eg:coh-or}, recall that
  $\Pin(1)=O(1)\times O(1)$ (which is the same as the Klein four
  group), and the map to $O(1)$ is projection to the first factor,
  say. Thus, specifying a $\Pin(1)$-structure on a line bundle $E$ is
  the same as specifying a second line bundle $E'$; thinking of this
  as lifting the map $X\to BO(1)$ to a map $X\to BO(1)\times BO(1)$,
  this new line bundle is the projection to the second factor (while
  the line bundle $E$ is projection to the first
  factor). In particular, we can choose equivariant $\Pin$-structures
  on the Lagrangians $L_i$.

  To specify equivariant gradings, view $T^2$ as a quotient of the
  unit square in $\CC$. The tangent bundle inherits a complex
  trivialization from $\CC$. Then for a curve $\gamma$ in $T^2$ we can
  view $\gamma'(t)$ as lying in $\CC$. The square phase map $\alpha$
  is $\alpha(\gamma(t))=\gamma'(t)^2/\|\gamma'(t)\|^2\in S^1$. If we
  view the Lagrangians $L_0$ as $\{1/2+it\mid t\in[0,1]\}/\sim$ and
  $L_1=\{1/2+\sin(2\pi t)/100+it\mid t\in[0,1]\}/\sim$ then
  $\alpha_{L_0}(t)$ is the constant function $-1$, which has an
  obvious lift to $\RR$, while $\alpha_{L_1}(t)$ is non-constant but
  clearly null-homotopic, and hence lifts to $\RR$. Such lifts give
  gradings on the $L_i$. The action of $\ZZ/2$ on $T^2$ is rotation
  around $1/2+i/2$, and the gradings are clearly equivariant.

  Thus, the Lagrangians $L_0$ and $L_1$ admit equivariant
  $\pin$-structures and gradings, and in fact any choice of
  $\pin$-structures can be made equivariant.
\end{example}

Finally, we note that a coherent orientation with respect to $(J,H)$
induces a coherent orientation with respect to certain
semi-cylindrical pairs. First, fix $g$ and consider the space
$\GJHspace^g$ of semi-cylindrical pairs $(\wt{J},\wt{H})$ connecting
$(J,H)$ and $g(J,H)$. This space is contractible, since both the space
of almost complex structures compatible with $\omega$ and the space of
Hamiltonian functions are contractible. The determinant line bundle
$\det(D\dbar)$ extends to a bundle over
$\mathcal{B}_g(x,y)\times \GJHspace^g$, and contractibility of
$\GJHspace^g$ implies that a trivialization of $\det(D\dbar)$ over
$\mathcal{B}_g(x,y)$ induces a trivialization over
$\mathcal{B}_g(x,y)\times \GJHspace^g$. More generally, let
$\GJHspace=\bigcup_{g\in G}\GJHspace^g$ (with topology inherited from
the product topology on $G\times \JSpace$); see also
Equation~\ref{eq:GJHspace-def}. Then there is a fiber bundle
$\GJHspace\times_G \mathcal{B}_G(x,y)\to \mathcal{B}_G(x,y)$ with
contractible fiber. So, a trivialization of $\det(D\dbar)$ over
$\mathcal{B}_G(x,y)$ induces a trivialization over
$\GJHspace\times_G \mathcal{B}_G(x,y)$. Given a $G$-orientation
system, these trivializations clearly satisfy the condition of
Definition~\ref{def:G-or-system} with $\mathcal{B}_g(x,y)$
(respectively $\mathcal{B}_h(x,y)$) replaced by
$\mathcal{B}_g(x,y)\times \GJHspace^g$ (respectively
$\mathcal{B}_h(x,y)\times\GJHspace^h$).

\subsection{Polytopes and moduli spaces}\label{subsec:moduli}
We continue to fix a generic cylindrical pair $(J,H)$.
We will be interested in semi-cylindrical pairs whose asymptotics are related to $(J,H)$ by elements of $G$. Precisely, let
\begin{equation}\label{eq:GJHspace-def}
  \GJHspace=\{(g,\JH)\in G\times \JSpace\mid \JH^{-\infty}=(J,H),\ \JH^{+\infty}=g(J,H)\},
\end{equation}
which we endow with the subspace topology. There are projections
$\pi_G\co\GJHspace\to G$ and $\pi_{\JSpace}\co \GJHspace\to\JSpace$.

Let $\Polyh^n$ be an $n$-dimensional (compact) polytope. The main case of
interest will be
$\Polyh^n=[0,1]^n$.  We will be interested in families
$\sigma=(g_\sigma,\JH_\sigma)\co \Polyh\to \GJHspace$ which are smooth in the
following sense:
\begin{definition}\label{def:smooth-family}
  Call a continuous map
  $\sigma=(g_\sigma,\JH_\sigma)\co \Polyh\to \GJHspace$ \emph{smooth}
  if it satisfies the following conditions:
  \begin{enumerate}
  \item The map $g_\sigma\co \Polyh\to G$ is smooth (i.e., extends to
    a smooth map on some open neighborhood of $\Polyh$ in $\RR^N$).
  \item For each open face $F$ of $\Polyh$ (of any dimension), the
    number of non-cylindrical stories of $\JH_\sigma|_F$ is constant.
  \item For each open face $F$ of $\Polyh$ (of any dimension) and each
    story $(\wt{J}_i,\wt{H}_i)$ of $\JH_\sigma|_{F}$, $\wt{J}_i$ is
    smooth when viewed as a fiberwise almost complex structure on the
    $M$-bundle 
  \[
    M\times [0,1]\times\RR\times F\to [0,1]\times\RR\times F
  \]
  and $H_i$ is a smooth function $M\times\RR\times F\to\RR$.
  \end{enumerate}
\end{definition}

Given such a family $\sigma\co \Polyh\to \GJHspace$, choose a point $v_0\in \Polyh$ and let
$g(\sigma)=g_\sigma(v_0)\in G$. Given $x,y\in L_0^H\cap L_1$, by
the \emph{$\sigma$-twisted homotopy classes from $x$ to $y$} we mean
$\pi_2^{g(\sigma)}(x,y)$. Since $\Polyh$ is simply connected, the
map $\xi_{g_\sigma\circ \gamma_t}$ associated to any path
$\gamma_t$ in $\Polyh$ from $v_0$ to $v\in\Polyh$ gives a
canonical bijection
$\pi_2^{g_\sigma(v)}(x,y)\cong \pi_2^{g(\sigma)}(x,y)$.

Given a family $\sigma=(g_\sigma,\JH_\sigma)$ and a $\sigma$-twisted homotopy class
$\phi$, we can define
\begin{align*}
  \cM(\phi;\sigma)&=\bigcup_{\quad\mathclap{v\in\interior(\Polyh)}\quad}\cM(\phi;\wt{J}(v),\wt{H}(v))\\
  \ocM(\phi;\sigma)&=\bigcup_{\quad\mathclap{v\in\Polyh}\quad}\ocM(\phi;\wt{J}(v),\wt{H}(v)).
\end{align*}
The spaces $\cM(\phi;\sigma)$ and $\ocM(\phi;\sigma)$ have natural topologies which we
again leave to the reader to spell out.

\begin{definition}
  Given a smooth family $\sigma\co \Polyh\to\GJHspace$ (in the sense
  of Definition~\ref{def:smooth-family}), we say that
  $\cM(\phi;\sigma)$ is \emph{transversely cut out} by the
  $\dbar$-operator if the following condition is satisfied for each
  open face $F$ of $\Polyh$ (of any dimension). Let
  $\JH|_F=(\wt{J_k},\wt{H}_k)\circ\cdots\circ(\wt{J}_1,\wt{H}_1)$ be
  the decomposition of $\JH|_F$ into stories. Then at each
  $u=(u_1,\dots,u_k)\in\cM(\phi;\sigma|_F(v))$ (where $v\in F$) we
  require that for each $i$, the linearized $\dbar$-operator with
  respect to $(\wt{J}_i,\wt{H}_i)$ at $u_i$ is transverse to the
  $0$-section and the map
  \[
    \cM([u_1];\wt{J}_1,\wt{H}_1)\times\cdots\times\cM([u_k];\wt{J}_k,\wt{H}_k)\to
    F\times\cdots\times F
  \]
  is transverse to the (thin) diagonal
  $\Delta=\{(v,v,\dots,v)\mid v\in F\}$ at $v$, where $[u_i]$ is the
  homotopy class of $u_i$.
\end{definition}
Note that this transversality condition does not impose a restriction
in the normal direction to the facets.

Given a submanifold $Y^n\subset X^m$ and a map $\pi\co Z\to X$ from
another manifold $Z^p$ with $p\leq m-n$, we say that $\pi$ is
\emph{topologically transverse} to $Y$ if the following condition is
satisfied. If $p<m-n$, then we require $\pi(Z)\cap Y=\emptyset$. If
$p=m-n$, we require for each point $z\in \pi^{-1}(Y)$ that there is an open
neighborhood $U\ni z$ in $Z$, an open neighborhood $V\ni
\pi(z)$ and a homeomorphism $\psi\co V\to \RR^m$ so that $\psi(Y\cap
V)=\RR^n\times\{0\}$ and
$\psi(\pi(U))=\{0\}\times\RR^{m-n}$.
This extends in an obvious way to the case of manifolds with
boundary. In particular, if $Z$ is one-dimensional and $Y=\bdy X$ then
$\pi$ is topologically transverse to $Y$ if there is a topological
collar neighborhood $Y\times[0,\epsilon)$ of $Y$ so that
$\pi(U)=\{\pi(z)\}\times[0,\epsilon)$.

\begin{definition}\label{def:generic-pair}
  We say that a smooth map $\sigma\co \Polyh^n\to \GJHspace$ is
  \emph{generic} if it
  satisfies the following conditions:
  \begin{enumerate}[resume, label=($\wt{G}$-\arabic*)]
  \item\label{item:family-transv} For any $\sigma$-twisted homotopy class $\phi$ with
    $\mu(\phi)\leq -n$, the moduli space $\cM(\phi;\sigma)$ is
    transversely cut out by the $\dbar$ equation.
  \item\label{item:family-transv-transv} For any $\sigma$-twisted
    homotopy class $\phi$ with $\mu(\phi)\leq -n$ and any face $F$ of
    $\Polyh^n$ (of any dimension), the projection $\ocM(\phi;\sigma)\to \Polyh^n$ is
    topologically transverse to $F$.
  \item\label{item:family-transv-faces} The restriction of $\sigma$ to any face of $\Polyh^n$ (of any dimension) satisfies Conditions~\ref{item:family-transv} and~\ref{item:family-transv-transv}, with $n$ replaced by the dimension of the face (i.e., is itself generic).
  \end{enumerate}
  We say that $\sigma$ is \emph{very generic} if the above hold with
  $\mu(\phi)\leq -n$ replaced by $\mu(\phi)\leq -n+1$.
\end{definition}
 
The property of being (very) generic is preserved by a kind of composition:
\begin{lemma}\label{lem:exterior-prod-generic}
  Let $P_1$ and $P_2$ be polyhedra and $\sigma_i\co P_i\to \GJHspace$ a generic
  (respectively very generic) map for $i=1,2$. Define a map $\tau\co P_1\times P_2\to\GJHspace$ by
  \begin{equation}\label{eq:exterior-prod-polyh}
    \tau(p_1,p_2)=\bigl(g_{\sigma_1}(p_1)g_{\sigma_2}(p_2), (g_{\sigma_1}(p_1)\cdot\JH_{\sigma_2}(p_2))\circ\JH_{\sigma_1}(p_1)\bigr).
  \end{equation}
  Then $\tau$ is generic (respectively very generic).
\end{lemma}
\begin{proof}
  Since the restriction of $\tau$ to each face of $P_1\times P_2$ is defined by
  Formula~\eqref{eq:exterior-prod-polyh}, it suffices to verify
  Conditions~\ref{item:family-transv} and~\ref{item:family-transv-transv}. We
  will focus on the (slightly harder) very generic case.
  Let $n_i=\dim(P_i)$.
  Given a homotopy class $\phi$, the corresponding moduli space is 
  \[
    \bigcup_{\psi_1*\psi_2=\phi}\cM(\psi_1;\sigma_1)\times \cM(\psi_2;\sigma_2).
  \]
  Since the $\sigma_i$ are generic, if $\mu(\phi)\leq -n_1-n_2$ then this moduli
  space is empty unless $\mu(\psi_1)=-n_1$ and $\mu(\psi_2)=-n_2$, in which case
  the moduli space is transversally cut out. Similarly,
  if $\mu(\phi)=-n_1-n_2+1$ then this moduli space is empty unless either
  \begin{itemize}
  \item $\mu(\psi_1)=-n_1$ and $\mu(\psi_2)=-n_2+1$ or
  \item $\mu(\psi_1)=-n_1+1$ and $\mu(\psi_2)=-n_2$,
  \end{itemize}
  and in both of those cases the moduli spaces are transversally cut
  out. Further, in each of these cases the intersection of the moduli space with
  each face of the boundary is clearly transverse.
\end{proof}

We recall some fundamental results about these moduli spaces of
holomorphic curves:
\begin{theorem}\label{thm:compactness} (Compactness)
  Given a convergent sequence $(\wt{J}_i,\wt{H}_i)\in\JSpace$ and
  a sequence $\{u_i\}$ of $(\wt{J}_i,\wt{H}_i)$-holomorphic disks with
  bounded area there is a subsequence of $\{u_i\}$ which converges
  to a broken holomorphic disk. In particular, the moduli spaces
  $\ocM(\phi;\sigma)$ are compact.
\end{theorem}
This is a modest extension of Floer's compactness
theorem~\cite[Proposition 2.2]{Floer88:LagrangianHF}, and a special
case of, e.g., Sullivan's extension of Floer's result~\cite[Theorem
2.7]{Sullivan02:K}. Note in particular that
Hypothesis~\ref{hyp:Floer-defined} rules out bubbling disks and
spheres (by Item~\ref{item:J-1}) and curves escaping to infinity
(Item~\ref{item:J-2}), and gives a uniform bound on the energy of
holomorphic disks in $\ocM(\phi;\sigma)$ (again by Item~\ref{item:J-1}).

\begin{theorem}\label{thm:transversality} (Transversality)
  If
  \[
  g_\sigma\co \Polyh^n\to G\qquad\text{and}\qquad
  \sigma_\bdy\co\bdy \Polyh\to\GJHspace
\]
  are such that $\pi_G\circ\sigma_\bdy=g_\sigma|_{\bdy \Polyh}$ and
  the restriction of $\sigma_\bdy$ to each boundary facet is generic
  (respectively very generic) then there is a generic (respectively
  very generic) extension $\sigma\co \Polyh\to\GJHspace$ of
  $\sigma_\bdy$ so that $g_\sigma=\pi_G\circ \sigma$.
\end{theorem}

That is, we can find a dashed arrow in the following diagram (very)
generically, assuming the top arrow is (very) generic:
\[
  \xymatrix{
    \partial \Polyh\ar@{^{(}->}[d] \ar[r]^{\sigma_\bdy}& \GJHspace\ar[d]^-{\pi_G}\\
    \Polyh\ar[r]_{g_\sigma} \ar@{-->}[ur]^{\sigma}& G.
  }
\]

\begin{proof}[Outline of proof]
  This is a simple adaptation of standard arguments
  (e.g.,~\cite{FHS95:transversality} or~\cite[Chapter
  3]{MS04:HolomorphicCurvesSymplecticTopology}). For definiteness,
  we will focus on the (slightly harder) very generic case.

  We are already given $\pi_G\circ \sigma$, so the goal is to extend
  $\pi_{\JSpace}\circ \sigma_\bdy$ from $\bdy \Polyh$ to $\Polyh$
  compatibly with $\pi_G\circ\sigma$. In what follows, we will denote
  the extension of $\pi_{\JSpace}\circ \sigma_\bdy$ we are
  constructing by $\JH_\sigma$.

  The assumption that $\sigma_\bdy$ is very generic on the boundary
  implies that for each $\phi$ with $\mu(\phi)\leq -n$ the space
  $\bigcup_{p\in\bdy \Polyh}\cM(\phi;\wt{J}(p),\wt{H}(p))$ is empty, while for each
  $\phi$ with $\mu(\phi)=-n+1$ the space
  $\bigcup_{p\in\bdy \Polyh}\cM(\phi;\wt{J}(p),\wt{H}(p))$ consists of finitely many
  points, none of which lie over the codimension $>1$ faces of $\Polyh$ (the \emph{corners} of $\Polyh$). Further,
  Hypothesis~\ref{hyp:Floer-defined} ensures that there are only
  finitely many homotopy classes $\phi$ with $\mu(\phi)\leq -n+1$ for
  which these moduli spaces are non-empty.
 
  Given an open subset $U\subset \Polyh$ and a map
  $\sigma_{J}\co \Polyh\to\JSpace$ it makes sense to say that $\JH_\sigma$ is
  very generic on $U$. Since the moduli spaces lying over the corners
  of $\Polyh$ are empty, by compactness given any extension of $\pi_{\JSpace}\circ \sigma_\bdy$
  to a neighborhood of the corners of $\Polyh$ there is a smaller
  neighborhood on which this extension is very generic (because the
  corresponding moduli spaces are empty). So, we may assume $\JH_\sigma$ is
  given and very generic in a neighborhood of the corners. Thus,
  instead of choosing a very generic extension of $\sigma_\bdy$ from $\bdy \Polyh$
  to $\Polyh$ it suffices to find a very generic extension from
  $\bdy\bD^n$ to $\bD^n$. (Upshot: no more corners.)

  Our next reduction is to the case that the image of
  $\pi_{\JSpace}\circ \sigma_\bdy$ is contained in the subspace of $1$-story
  semicylindrical pairs. Identify a closed neighborhood of $\bdy\bD^n$ with
  $\bdy\bD^n\times[0,\epsilon]$.  It follows from standard gluing arguments
  (compare Theorem~\ref{thm:gluing}) that given a very generic semicylindrical pair with
  length $>1$ for any gluing $\gluing(\pi_{\JSpace}\circ\sigma_\bdy)$ of
  $\pi_{\JSpace}\circ \sigma_\bdy$ with sufficiently large gluing parameters,
  the moduli spaces of $\gluing(\pi_{\JSpace}\circ\sigma_\bdy)$-holomorphic disks in the homotopy
  classes $\phi$ with $\mu(\phi)\leq -n+1$ are also transversely cut out. In
  particular, we can extend $\pi_{\JSpace}\circ \sigma_\bdy$ to
  $\bdy \bD^n\times[0,\epsilon]$ as follows. 
  If the semi-cylindrical pairs in $\pi_{\JSpace}\circ \sigma_\bdy$ have one
  story, define
  $\JH_{\sigma,\nbd(\bdy)}'\co \bdy \bD^n\times[0,\epsilon]\to\JSpace$ to be
  $\JH_{\sigma,\nbd(\bdy)}'(p,t)=\pi_{\JSpace}\bigl(\sigma_\bdy(p)\bigr)$. Otherwise,
  choose a homeomorphism $\chi\co (0,\epsilon]\to [R,\infty)$ for some $R\gg0$
  and let $\JH_{\sigma,\nbd(\bdy)}'(p,t)$ be the result of gluing
  $\pi_{\JSpace}\bigl(\sigma_\bdy(p)\bigr)$ with gluing parameter $\chi(t)$. For
  $R$ sufficiently large, for each fixed $t$ (i.e., on each spherical shell) the moduli space of holomorphic
  disks with respect to $\bigcup_{p\in \bdy\bD^n}\JH_{\sigma,\nbd(\bdy)}'(p,t)$
  is transversely cut out, and consists of a single point near each point in the
  moduli space with respect to $\pi_{\JSpace}\circ\sigma_\bdy$. In particular,
  these moduli spaces are topologically transverse to the boundary.
  
  The asymptotics of $\JH_{\sigma,\nbd(\bdy)}'$ at $+\infty$ are the same as for
  $\sigma_\bdy$ (i.e., are constant in the $[0,\epsilon]$-direction), but
  $g_\sigma\cdot(J,H)$ is not necessarily constant in the
  $[0,\epsilon]$-direction. Shrinking $\epsilon$, however, we can find another
  extension $\JH_{\sigma,\nbd(\bdy)}$ arbitrarily close to $\JH_{\sigma,\nbd(\bdy)}'$ and asymptotic to 
  $\bigl(g_\sigma\cdot (J,H)\bigr)|_{\bdy \bD^n\times [0,\epsilon]}$ at $+\infty$ (since $g_\sigma$ is smooth).
  For $\epsilon$ small enough, the moduli space with respect to $\JH_{\sigma,\nbd(\bdy)}|_{\bdy\bD^n\times\{t\}}$ is still transversely cut out, because transversality is an open condition, and consists of a single point near each point of the $\pi_{\JSpace}\circ\sigma_\bdy$ moduli space. In particular, these moduli spaces are still topologically transverse to the boundary. So, this extension $\JH_{\sigma,\nbd(\bdy)}$ is very generic.
  
  Let $U=\bdy\bD^n\times[0,\epsilon)$ with such a small $\epsilon$ and, by abuse
  of notation, write $\sigma_\bdy$ for the extension of $\sigma_\bdy$ to $U$.
  Note that the image of $\pi_{\JSpace}\circ \sigma_\bdy$ is contained in the
  space of $1$-story semicylindrical pairs away from $\bdy \bD^n$.

  We next adapt standard transversality arguments to show that there
  is a generic extension of $\sigma_\bdy$ to the rest of the disk.  
  Choose a smaller neighborhood $V$ of $\bdy\bD^n$ with
  $\overline{V}$ contained in the original neighborhood. Given
  $\ell>2$ consider the following spaces:
  \begin{itemize}
  \item The space of maps $\bD^n\to \JSpace$ whose restriction to $\ol{V}$
    agrees with $\pi_{\JSpace}\circ \sigma_\bdy$ and which agree with $(J,H)$ and
    $g_\sigma(J,H)$ at $-\infty$ and $+\infty$, respectively,
    so that the induced fiberwise almost complex
    structure on $\bD^n\times [0,1]\times\RR\times M$, and the
    function $\bD^n\times \RR\times M$ induced by the Hamiltonian, are
    $C^\ell$. (So, by
    $\JSpace$ we really mean the $C^\ell$ semi-cylindrical pairs,
    whereas usually we mean smooth semi-cylindrical pairs.) We will denote this space
    $\Map_{C^\ell}((\bD^n,V),(\JSpace,\sigma_\bdy))$. Note that
    $\Map_{C^\ell}((\bD^n,V),(\JSpace,\sigma_\bdy))$ is nonempty and
    contractible, since the space of $C^\ell$ almost complex
    structures on $M$ is contractible.
  \item The bundle $\mathcal{B}(x,y)$ over
    $\Map_{C^\ell}((\bD^n,V),(\JSpace,\sigma_\bdy))\times \bD^n$
    whose fiber over $(\JH_\sigma,v)$ is the space of maps
    \[
      ([0,1]\times\RR,\{0\}\times\RR,\{1\}\times\RR)\to(M,L_1,L_0^{\wt{H}(v)})
    \]
    connecting the generators $x$ and $g_\sigma(v)\cdot y$, in the weighted Sobolev space
    $W^{1,p}_\delta$, for an appropriate weight $\delta\approx 0$; see,
    e.g.,~\cite[Section (8h)]{SeidelBook}.
  \item The bundle $\mathcal{E}$ over $\mathcal{B}(x,y)$ whose fiber
    over $u$ is the space of
    $W^{0,p}_\delta$ sections of $\Lambda^{0,1}u^*TM$ (i.e.,
    $(0,1)$-forms on $[0,1]\times\RR$ valued in $u^*TM$ with
    exponential decay).
  \end{itemize}
  The $\dbar$ operator defines a section
  \[
    \dbar\co \mathcal{B}(x,y)\to \mathcal{E},
  \]
  by $\dbar(u,\JH,v)=\dbar_{\JH(v)}(u)$.
  A choice of Riemannian metric on $M$ induces a linearization
  \[
    D\dbar\co T\mathcal{B}(x,y)\to \mathcal{E}
  \]
  of this section $\dbar$.  Further, $\dbar$ is a $C^{\ell-1}$ map
  (cf.~\cite[Section
  3.2]{MS04:HolomorphicCurvesSymplecticTopology}). A standard argument
  (again, cf.~\cite[Section
  3.2]{MS04:HolomorphicCurvesSymplecticTopology}) shows that the
  linearization $D\dbar$ is surjective, so $\dbar^{-1}(0)$ is a
  $C^{\ell-1}$ Banach manifold. The fact that maps $\JH_\sigma$ as desired
  are co-meager in $\Map_{C^\ell}((\bD^n,V),(\JSpace,\sigma_\bdy))$
  then follows from Smale's infinite-dimensional version of Sard's
  theorem~\cite[Theorem 1.3]{Smale65:Sard}, applied to the projection
  \[
   \dbar^{-1}(0) \to \Map_{C^\ell}((\bD^n,V),(\JSpace,\sigma_\bdy));
  \]
  here we need that $\ell>2$ to have enough regularity to apply the
  Sard-Smale theorem. Finally, to bootstrap to a result about the
  space of smooth (rather than $C^\ell$) almost complex structures
  requires a small further argument (again, cf.~\cite[Section
  3.2]{MS04:HolomorphicCurvesSymplecticTopology}).
\end{proof}

\begin{theorem}\label{thm:gluing} (Gluing)
  Given a very generic family $\sigma\co\Polyh^{n-1}\to\GJHspace$ and
  a homotopy class $\phi\in\pi_2^{g(\sigma)}(x,y)$ with
  $\mu(\phi)=2-n$, the compactified moduli space $\ocM(\phi;\sigma)$
  is a topological $1$-manifold with boundary, the boundary of which is
  \begin{equation}\label{eq:gluing}
  \begin{split}
    (-1)^n\cM(\phi;\sigma|_{\bdy \Polyh})
    &\amalg
    \bigcup_{z\in L_0^H\cap L_1}
    \bigcup_{\substack{\psi\in\pi_2^{1}(x,z),\ \psi'\in\pi_2^{g(\sigma)}(z,y)\\\psi*\psi'=\phi\\
    \mu(\psi)=1,\ \mu(\psi')=-n+1}}\hspace{-2em}
    (-1)^n\cM(\psi';\sigma)\times\cM(\psi;J,H)\\
    &\amalg
    \bigcup_{z\in L_0^H\cap L_1}
    \bigcup_{\substack{\psi\in\pi_2^{g(\sigma)}(x,z),\ \psi'\in\pi_2^{1}(z,y)\\\psi*\psi'=\phi\\
    \mu(\psi)=-n+1,\ \mu(\psi')=1}}\hspace{-2em}
    \cM(\psi';J,H)\times \cM(\psi;\sigma).
  \end{split}
  \end{equation}
  (That is, the boundary consists of holomorphic curves over
  $\bdy\Polyh^{n-1}$ and breakings of curves at either $-\infty$ or
  $+\infty$.) Further, 
  if $\Polyh=[0,1]^{n-1}$ then the first term in Formula~\eqref{eq:gluing} is given by
  \begin{equation}
    \label{eq:gluing-cube}
    (-1)^n\cM(\phi;\sigma|_{\bdy [0,1]^{n-1}})
        =\bigcup_{i=1}^{n-1}\bigcup_{\epsilon\in\{0,1\}}(-1)^{n+i+\epsilon}\cM(\phi;\sigma|_{\bdy_{i,\epsilon}[0,1]^{n-1}}).
  \end{equation}
  In Formulas~(\ref{eq:gluing}) 
  and~(\ref{eq:gluing-cube}), the signs denote the orientations of the
  moduli spaces, when we are working in the oriented setting.
\end{theorem}
Note that the ordering of the factors in the second and third terms
agrees with composition order in a category (with objects intersection
points and morphisms holomorphic disks), and not with concatenation
order (the operation $*$ above).
\begin{proof}[Aspects of proof]
  Again, this is a simple adaptation of standard arguments. 
  Note that
  Condition~\ref{item:family-transv-transv}
  easily implies the result (modulo signs)
  for curves lying over $\bdy\Polyh^{n-1}$, so the interesting case is
  gluing a broken disk lying over an interior point of $\Polyh^{n-1}$. The
  argument in this case goes back to
  Floer~\cite[Proposition 4.1]{Floer88:LagrangianHF} (see
  also~\cite[Theorem 4.5]{Sullivan02:K}), and a nice explanation in a
  related context can be found in McDuff-Salamon~\cite[Chapter
  10]{MS04:HolomorphicCurvesSymplecticTopology}. 
  Note that in the last term we have used the action of $G$ to
  identify the moduli spaces of disks from $g_\sigma(v)\cdot z$ to
  $g_\sigma(v)\cdot y$ with respect to $g_\sigma(v)\cdot (J,H)$ 
  with the space of disks from $z$ to $y$ with
  respect to $(J,H)$, i.e., $\cM(\psi';J,H)$.
  
  We will not say more about these gluing arguments, but will discuss
  the orientations. 
  Fix a $G$-orientation as in Section~\ref{sec:G-or}, for some cylindrical pair $(J,H)$.
  Given any other semicylindrical pair $\JH$ agreeing with $(J,H)$ at $-\infty$, as discussed at the end of Section~\ref{sec:G-or} there is an induced
  trivialization of $\det(D\dbar|_{\JH})$.
  There is also an
  identification
  \begin{equation}\label{eq:det-dbar-decomp}
    \det(D\dbar)=\det(D\dbar|_{\JH})\otimes\det(T\Polyh^{n-1}).
  \end{equation}
  So, via the orientation of the polyhedron $\Polyh^{n-1}$ we obtain an
  orientation of $\det(D\dbar)$.

  By construction, if $\sigma$ is a family of semi-cylindrical pairs
  such that $\cM(\phi;\sigma)$ is transversely cut out then the
  tangent space $T\cM(\phi;\sigma)$ is identified with the index
  bundle of $D\dbar$, so the coherent orientation induces an
  orientation of $\cM(\phi;\sigma)$. In the case that
  $\cM(\phi;\sigma)$ is $0$-dimensional, this means that
  $\cM(\phi;\sigma)$ is a collection of signed points.

  In the cylindrical case, there is an additional complication. We
  quotient by the action of $\RR$ to get the unparameterized moduli
  space, so the orientation of the unparameterized moduli space
  depends on both the coherent orientation system and the orientation
  of $\RR$. If we let $\widetilde{\cM}$ denote the parameterized
  moduli space then we define the orientation of the unparameterized
  moduli space by declaring that
  \begin{equation}\label{eq:R-side}
    \widetilde{\cM}=\RR\times\cM
  \end{equation}
  as oriented manifolds (where $\RR$ has the standard orientation and
  the standard action by $\RR$).

  Our goal is to compare the induced orientation of the boundary of
  $\ocM(\phi;\sigma)$ with the orientation of the moduli spaces
  appearing in Formula~\eqref{eq:gluing} coming from the coherent
  orientations, and deduce the signs given in
  Formulas~\eqref{eq:gluing} 
  and~(\ref{eq:gluing-cube}). To this end, note that the sign in
  Formula~(\ref{eq:gluing-cube}) comes from the fact that the
  identification of $\bdy_{i,\epsilon}[0,1]^{n-1}$ with $[0,1]^{n-2}$
  is orientation-preserving or orientation-reversing depending on the
  parity of $i+\epsilon$. Note also that, before taking determinants,
  there is a formally $(2-n)$-dimensional bundle to the left of
  $T[0,1]^{n-2}$ in Formula~\eqref{eq:det-dbar-decomp}. Thus, the
  ``outward normal first'' convention gives the sign of the first
  term. This also explains the sign in the first term of
  Formula~(\ref{eq:gluing}), from commuting the outward normal vector
  past the $(2-n)$-dimensional bundle.

  For the second term in Formula~(\ref{eq:gluing}), consider a curve 
  \[
    (u_1,u_2)\in\cM(\psi';\sigma)\times \wt{\cM}(\psi;J,H).
  \]
  Gluing using a fixed, large gluing parameter $R$ identifies a
  neighborhood of $(u_1,u_2)$ in $\cM(\psi*\psi';\sigma)$ with
  $[0,1)$. Since the $\RR$-action on $u_2$ translates upwards, it
  agrees with the inward-pointing normal vector to the boundary of
  $[0,1)$, so switching to the usual, outward-normal boundary
  orientation contributes a minus sign. Further, the isomorphism
  \[
    \det(D\dbar|_{\JH})\otimes\det(T\Polyh^{n-1})\otimes\RR\otimes
    \det(D\dbar|_{(J,H)})\to 
    \RR\otimes\det(D\dbar|_{\JH})\otimes
    \det(D\dbar|_{(J,H)})\otimes\det(T\Polyh^{n-1})
  \]
  given by permuting the factors
  picks up a sign of $(-1)^{3(n-1)}=(-1)^{n-1}$.
  Thus, using the standard, outward-normal first convention for
  orienting boundaries, the orientation of $(u_1,u_2)$ in
  $\cM(\psi';\sigma)\times\cM(\psi;J,H)$ and in
  $\cM(\psi*\psi';\sigma)$ differ by $(-1)^{n-1+1}=(-1)^n$.

  The argument for the third term is similar, except that now:
  \begin{itemize}
  \item The action of $\RR$ is on the top story, so points outwards
    with respect to gluing.
  \item The tensor product of index bundles is already
    \[
      \RR\otimes\det(D\dbar|_{(J,H)})\otimes\det(D\dbar|_{\JH})\otimes \det(T\Polyh^{n-1}),
    \]
    so does not contribute an additional sign.
  \item The fact that we have a $G$-orientation, rather than just an
    orientation, is important, as we are pre-gluing
    $g_\sigma\cdot (u_1)$ and $u_2$.
  \end{itemize}

  Finally, when one of the semi-cylindrical pairs $\JH_\sigma(v)$ has more
  than $1$ story, it is possible to degenerate a cylindrical story in
  the middle, i.e., a curve $(u_1,u,u_2)$ where $u$ is
  cylindrical. These degenerations occur as the boundary of two moduli
  spaces, corresponding to gluing $u_1$ and $u$ and to gluing $u$ and
  $u_2$, respectively. In the first case the $\RR$-action is
  outward-pointing, while in the second case the $\RR$-action is
  inward-pointing. The rest of the orientation discussion is the same
  for the two cases, so the two cases contribute with opposite signs.
\end{proof}

\subsection{Constructing the diagram of complexes}\label{sec:build-diag}

By Sections~\ref{sec:smooth-nerves} and~\ref{sec:homotopy-colimits},
in order to construct equivariant Floer homology, it suffices to
construct an appropriate functor $F\co \Nerve BG\to\Complexes$ and
then take the hypercohomology. The functor will factor through another
category, $\Nerve\GJHcat$.

We continue to fix a generic cylindrical pair $(J,H)$.
Let $\GJHcat$ be the topological category with a single object $(J,H)$ and
$\Hom((J,H),(J,H))=\GJHspace/\sim$ where $\sim$ identifies semi-cylindrical
pairs which differ by expansions (i.e., allows forgetting cylindrical stories).
Make $\GJHcat$ into a smooth category using the definition of smoothness from
Definition~\ref{def:smooth-family}; it is clear that this definition
satisfies the conditions required for a smooth category (Definition~\ref{def:smooth-cat}). Let
$\Nerve\GJHcat$ be the smooth nerve of $\GJHcat$. 

We will actually be interested in a particular subcategory of $\Nerve\GJHcat$.
\begin{definition}\label{def:acme}
  Call $\sigma\in\Nerve_n\GJHcat$ \emph{acme} (respectively \emph{very acme}) if
  for each $0\leq i<j\leq n$, the map
  \[
    \sigma|_{\Hom(i,j)}\co[0,1]^{j-i-1}\to \GJHspace
  \]
  is generic (respectively very generic).
  
  Let $\NerveWAc_0\GJHcat=\NerveAc_0\GJHcat=\Nerve_0\GJHcat=\{(J,H)\}$
  and for $n>0$ let $\NerveWAc_n\GJHcat\subset \Nerve_n\GJHcat$
  (respectively $\NerveAc_n\GJHcat$) be the set of acme (respectively
  very acme) $n$-simplices in $\Nerve\GJHcat$. We call
  $\NerveWAc\GJHcat$ the \emph{acme nerve} and $\NerveAc\GJHcat$ the
  \emph{very acme nerve} of $\GJHcat$.
\end{definition}

\begin{lemma}
  The acme (respectively very acme) nerve is a subcategory of
  $\Nerve\GJHcat$, i.e., is a sub-simplicial set that satisfies the
  weak Kan condition.
\end{lemma}
\begin{proof}
  The proofs of the two statements are essentially the same; we focus
  on the very acme case.  We must check that $\NerveAc\GJHcat$ is
  closed under face and degeneracy maps and under horn filling. The
  fact that $\NerveAc\GJHcat$ is closed under face maps is immediate
  from the definition. For
  degeneracies, for each $i<j$, $(s_k\sigma)|_{\Hom(i,j)}$
  is either given by the restriction of $\sigma$ to one of the
  $\Hom$-spaces (if $k\not\in[i,j-1]$), or is given by projecting
  $[0,1]^{j-i-1}\to[0,1]^{j-i-2}$ by a smooth map $\delta$ and
  applying the restriction of $\sigma$ to one of its $\Hom$-spaces. In
  the former case, the map $(s_k\sigma)|_{\Hom(i,j)}$ is very generic
  by assumption. In the latter case, note that every point in the
  interior of $[0,1]^{j-i-2}$ is a regular value of $\delta$.  The
  index $i-j+1$ moduli spaces with respect to
  $(s_k\sigma)|_{\Hom(i,j)}$ are empty by dimension counting. There
  are finitely-many points $q_1,\dots,q_\ell$ in the interior of
  $[0,1]^{j-i-2}$ for which the index $i-j+2$ moduli space with
  respect to $\sigma|_{\Hom(i,j-1)}(q_m)$ is non-empty. The index
  $(i-j+2)$-dimensional moduli spaces with respect to
  $(s_k\sigma)|_{\Hom(i,j)}$ are
  $\cM(\sigma|_{\Hom(i,j-1)}(q_i))\times \delta^{-1}(q_i)$. In
  particular, it follows from the fact that $q_i$ is a regular value
  of $\delta$ that these moduli spaces are transversally cut-out and
  it is clear from the definition of $\delta$ that these moduli spaces
  are transverse to the boundary strata in $[0,1]^{j-i-1}$.

  For horn filling, suppose we are given a map
  $f\co \Lambda^n_i\to \NerveAc\GJHcat$ for some $0<i<n$, which we want to
  extend to an $n$-simplex $\sigma\in \NerveAc_n\GJHcat$.  Given
  $0\leq a<b\leq n$, if $i\not\in[a,b]$ then the map $f$ specifies
  $\sigma|_{\Hom(a,b)}$. If $a\leq i\leq b$ then $f$ specifies
  $\sigma|_{\Hom(a,b)}\co[0,1]^{b-a-1}\to\GJHspace$ on the boundary facets
  $t_j=0$ for $j\neq i$. Further, if we already know $\sigma|_{\Hom(c,d)}$ for
  $d-c<b-a$ then the fact that $\sigma$ must respect composition specifies
  $\sigma|_{\Hom(a,b)}$ on the facets $t_j=1$. Further, if these
  $\sigma|_{\Hom(c,d)}$ are very generic then, by
  Lemma~\ref{lem:exterior-prod-generic}, so is the value of $\sigma$ on the
  facets $t_j=1$.

  So, define the extension $\sigma|_{\Hom(a,b)}$ inductively on
  $b-a$. At each stage, $\sigma|_{\Hom(a,b)}$ is defined except on the
  facet $t_i=0$ and the top-dimensional cell, and is very generic
  where defined. Applying Theorem~\ref{thm:transversality} first to
  $\bdy_{i,0}[0,1]^{b-a-1}$ (i.e., the face with $t_i=0$) and then to
  $[0,1]^{b-a-1}$ itself gives the desired very acme extension of $f$.
\end{proof}

\begin{example}\label{eg:ngj-3cell}
A $3$-cell $\sigma$ in the smooth nerve $\Nerve \GJHcat$ consists of:
\def\yscale{0.35}
\def\xscale{0.08}
\begin{enumerate}[leftmargin=*,label=(N-\arabic*),ref=(N-\arabic*)]
\item a $3$-cell in $\Nerve B\SingG$, with data as in Example~\ref{exam:3-cell-in-nbg},
\item\label{item:nerve-0} six elements of $\JSpace$, 
 $\vcenter{\hbox{\begin{tikzpicture}[xscale=\xscale,yscale=\yscale]
    \fill[blue!50!white] (0,0) rectangle (2,2);
    \node[intext] at (1,0) {\tiny $(J,H)$};
    \node[intext] at (1,2) {\tiny $\action{g_{01}}{(J,H)}$};
  \end{tikzpicture}}}$, $\vcenter{\hbox{\begin{tikzpicture}[xscale=\xscale,yscale=\yscale]
    \fill[blue!50!white] (0,0) rectangle (2,2);
    \node[intext] at (1,0) {\tiny $(J,H)$};
    \node[intext] at (1,2) {\tiny $\action{g_{12}}{(J,H)}$};
  \end{tikzpicture}}}$, $\vcenter{\hbox{\begin{tikzpicture}[xscale=\xscale,yscale=\yscale]
    \fill[blue!50!white] (0,0) rectangle (2,2);
    \node[intext] at (1,0) {\tiny $(J,H)$};
    \node[intext] at (1,2) {\tiny $\action{g_{23}}{(J,H)}$};
  \end{tikzpicture}}}$, $\vcenter{\hbox{\begin{tikzpicture}[xscale=\xscale,yscale=\yscale]
    \fill[blue!50!white] (0,0) rectangle (2,2);
    \node[intext] at (1,0) {\tiny $(J,H)$};
    \node[intext] at (1,2) {\tiny $\action{g_{02}}{(J,H)}$};
  \end{tikzpicture}}}$, $\vcenter{\hbox{\begin{tikzpicture}[xscale=\xscale,yscale=\yscale]
    \fill[blue!50!white] (0,0) rectangle (2,2);
    \node[intext] at (1,0) {\tiny $(J,H)$};
    \node[intext] at (1,2) {\tiny $\action{g_{13}}{(J,H)}$};
  \end{tikzpicture}}}$, $\vcenter{\hbox{\begin{tikzpicture}[xscale=\xscale,yscale=\yscale]
    \fill[blue!50!white] (0,0) rectangle (2,2);
    \node[intext] at (1,0) {\tiny $(J,H)$};
    \node[intext] at (1,2) {\tiny $\action{g_{03}}{(J,H)}$};
  \end{tikzpicture}}}$,
\item\label{item:nerve-1} four smooth paths in $\JSpace$, 
  $\vcenter{\hbox{\begin{tikzpicture}[xscale=\xscale,yscale=\yscale]
    \fill[blue!50!white] (0,0) rectangle (2,2);
    \node[intext] at (1,0) {\tiny $(J,H)$};
    \node[intext] at (1,2) {\tiny $\action{{\color{red}g_{02,012}}}{(J,H)}$};
  \end{tikzpicture}}}$, $\vcenter{\hbox{\begin{tikzpicture}[xscale=\xscale,yscale=\yscale]
    \fill[blue!50!white] (0,0) rectangle (2,2);
    \node[intext] at (1,0) {\tiny $(J,H)$};
    \node[intext] at (1,2) {\tiny $\action{{\color{red}g_{13,123}}}{(J,H)}$};
  \end{tikzpicture}}}$, $\vcenter{\hbox{\begin{tikzpicture}[xscale=\xscale,yscale=\yscale]
    \fill[blue!50!white] (0,0) rectangle (2,2);
    \node[intext] at (1,0) {\tiny $(J,H)$};
    \node[intext] at (1,2) {\tiny $\action{{\color{red}g_{03,013}}}{(J,H)}$};
  \end{tikzpicture}}}$, $\vcenter{\hbox{\begin{tikzpicture}[xscale=\xscale,yscale=\yscale]
    \fill[blue!50!white] (0,0) rectangle (2,2);
    \node[intext] at (1,0) {\tiny $(J,H)$};
    \node[intext] at (1,2) {\tiny $\action{{\color{red}g_{03,023}}}{(J,H)}$};
  \end{tikzpicture}}}$,
\item\label{item:nerve-2} a square in $\JSpace$ 
 $\vcenter{\hbox{\begin{tikzpicture}[xscale=\xscale,yscale=\yscale]
    \fill[blue!50!white] (0,0) rectangle (2,2);
    \node[intext] at (1,0) {\tiny $(J,H)$};
    \node[intext] at (1,2) {\tiny $\action{{\color{green!50!black}g_{\maxmap{\sigma}}}}{(J,H)}$};
  \end{tikzpicture}}}$, 
\end{enumerate}
fitting into the following square:
\[
\begin{tikzpicture}[cm={0,1,1,0,(0,0)},scale=4]

\tikzstyle{pair}=[rectangle]

\draw[thick] (0,0) rectangle (2,2);

\node[pair] at (0,0) {\begin{tikzpicture}[xscale=\xscale,yscale=\yscale]
    \fill[blue!50!white] (0,0) rectangle (2,6);
    \node[intext] at (1,0) {\tiny $(J,H)$};
    \node[intext] at (1,6) {\tiny $\action{g_{03}}{(J,H)}$};
  \end{tikzpicture}};
\node[pair] at (0,2) {\begin{tikzpicture}[xscale=\xscale,yscale=\yscale]
    \fill[blue!50!white] (0,0) rectangle (2,2);
    \fill[blue!50!white] (0,3) rectangle (2,7);
    \node[intext] at (1,0) {\tiny $(J,H)$};
    \node[intext] at (1,2) {\tiny $\action{g_{01}}{(J,H)}$};
    \node[intext] at (1,3) {\tiny $(J,H)$};
    \node[intext] at (1,7) {\tiny $\action{g_{13}}{(J,H)}$};
    \node[left] at (0,5) {\tiny$\mathllap{\action{g_{01}}{}}$};
  \end{tikzpicture}};
\node[pair] at (2,0) {\begin{tikzpicture}[xscale=\xscale,yscale=\yscale]
    \fill[blue!50!white] (0,0) rectangle (2,4);
    \fill[blue!50!white] (0,5) rectangle (2,7);
    \node[intext] at (1,0) {\tiny $(J,H)$};
    \node[intext] at (1,4) {\tiny $\action{g_{02}}{(J,H)}$};
    \node[intext] at (1,5) {\tiny $(J,H)$};
    \node[intext] at (1,7) {\tiny $\action{g_{23}}{(J,H)}$};
    \node[left] at (0,6) {\tiny$\mathllap{\action{g_{02}}{}}$};
  \end{tikzpicture}};
\node[pair] at (2,2) {\begin{tikzpicture}[xscale=\xscale,yscale=\yscale]
    \fill[blue!50!white] (0,0) rectangle (2,2);
    \fill[blue!50!white] (0,3) rectangle (2,5);
    \fill[blue!50!white] (0,6) rectangle (2,8);
    \node[intext] at (1,0) {\tiny $(J,H)$};
    \node[intext] at (1,2) {\tiny $\action{g_{01}}{(J,H)}$};
    \node[intext] at (1,3) {\tiny $(J,H)$};
    \node[intext] at (1,5) {\tiny $\action{g_{12}}{(J,H)}$};
    \node[intext] at (1,6) {\tiny $(J,H)$};
    \node[intext] at (1,8) {\tiny $\action{g_{23}}{(J,H)}$};
    \node[left] at (0,4) {\tiny$\mathllap{\action{g_{01}}{}}$};
    \node[left] at (0,7) {\tiny$\mathllap{\action{g_{01}g_{12}}{}}$};
  \end{tikzpicture}};

\node[pair] at (1,2) {\begin{tikzpicture}[xscale=\xscale,yscale=\yscale]
    \fill[blue!50!white] (0,0) rectangle (2,2);
    \fill[blue!50!white] (0,3) rectangle (2,7);
    \node[intext] at (1,0) {\tiny $(J,H)$};
    \node[intext] at (1,2) {\tiny $\action{g_{01}}{(J,H)}$};
    \node[intext] at (1,3) {\tiny $(J,H)$};
    \node[intext] at (1,7) {\tiny $\action{{\color{red}g_{13,123}}}{(J,H)}$};
    \node[left] at (0,5) {\tiny$\mathllap{\action{g_{01}}{}}$};
  \end{tikzpicture}};
\node[pair] at (2,1) {\begin{tikzpicture}[xscale=\xscale,yscale=\yscale]
    \fill[blue!50!white] (0,0) rectangle (2,4);
    \fill[blue!50!white] (0,5) rectangle (2,7);
    \node[intext] at (1,0) {\tiny $(J,H)$};
    \node[intext] at (1,4) {\tiny $\action{{\color{red}g_{02,012}}}{(J,H)}$};
    \node[intext] at (1,5) {\tiny $(J,H)$};
    \node[intext] at (1,7) {\tiny $\action{g_{23}}{(J,H)}$};
    \node[left] at (0,6) {\tiny$\mathllap{\action{{\color{red}g_{02,012}}}{}}$};
  \end{tikzpicture}};

\node[pair] at (1,1) {\begin{tikzpicture}[xscale=\xscale,yscale=\yscale]
    \fill[blue!50!white] (0,0) rectangle (2,6);
    \node[intext] at (1,0) {\tiny $(J,H)$};
    \node[intext] at (1,6) {\tiny $\action{{\color{green!50!black}g_{\maxmap{\sigma}}}}{(J,H)}$};
  \end{tikzpicture}};
\node[pair] at (1,0) {\begin{tikzpicture}[xscale=\xscale,yscale=\yscale]
    \fill[blue!50!white] (0,0) rectangle (2,6);
    \node[intext] at (1,0) {\tiny $(J,H)$};
    \node[intext] at (1,6) {\tiny $\action{{\color{red}g_{03,023}}}{(J,H)}$};
  \end{tikzpicture}};
\node[pair] at (0,1) {\begin{tikzpicture}[xscale=\xscale,yscale=\yscale]
    \fill[blue!50!white] (0,0) rectangle (2,6);
    \node[intext] at (1,0) {\tiny $(J,H)$};
    \node[intext] at (1,6) {\tiny $\action{{\color{red}g_{03,013}}}{(J,H)}$};
  \end{tikzpicture}};

\end{tikzpicture}
\]
This 3-cell is (very) acme if all of the data~\ref{item:nerve-0}--\ref{item:nerve-2} is (very) generic.
\end{example}

Projection $(g_\sigma,\JH_\sigma)\mapsto g_\sigma$ defines a functor
$\pi\co \GJHcat\to BG$. Taking smooth nerves, we get an
infinity functor (simplicial set map)
$\Nerve \pi\co \Nerve \GJHcat\to \Nerve BG$. This restricts to simplicial set maps $\Nerve \pi\co \NerveWAc \GJHcat\to \Nerve BG$ and $\Nerve \pi\co \NerveAc \GJHcat\to \Nerve BG$.
\begin{lemma}\label{lem:sec-exists}
  There is a section $\BGSection\co \Nerve BG\to\NerveAc \GJHcat$
  of the projection $\Nerve \pi$, i.e., a simplicial set map $\BGSection\co
  \Nerve BG\to\NerveAc\GJHcat$ so that $(\Nerve \pi)\circ
  \BGSection=\Id_{\Nerve BG}$.
\end{lemma}
\begin{proof}
  The proof is by induction on the dimension of a simplex. There is a unique
  $0$-simplex of $\Nerve BG$ and a unique $0$-simplex $(J,H)$ of
  $\Nerve\GJHcat$, so a unique way to define $\BGSection$ on the
  $0$-simplices. Now, assume $\BGSection$ has been constructed for all simplices
  of dimension $\leq n-1$, respecting the simplicial identities among simplices
  of those dimensions. Consider a non-degenerate $n$-simplex $\sigma$ of
  $\Nerve BG$, i.e., a functor $\sigma\co \sCVcat[n]\to BG$ of topological
  categories. By induction, since the functor must respect the face
  maps $\partial_0$ and $\partial_n$, we have already
  defined $\BGSection(\sigma)$ on $\Hom_{\sCVcat[n]}(i,j)$ except for $i=0$ and
  $j=n$.  
  Further, the fact that $\BGSection(\sigma)$ must respect composition in
  $\sCVcat[n]$ means that
  $\BGSection(\sigma)|_{\Hom(0,n)}=\cubemap{\BGSection(\sigma)}\co[0,1]^{n-1}\to \GJHspace$ is already
  defined on $\bdy_{i,1}[0,1]^{n-1}$ for each $i$, and the fact that
  $\BGSection$ must respect face maps means that
  $\cubemap{\BGSection(\sigma)}$ is already defined on
  $\bdy_{i,0}[0,1]^{n-1}$ for each $i$.

  We claim that the restriction of $\cubemap{\BGSection(\sigma)}$ to each of these
  facets is very generic. The face $\bdy_{j,0}\cubemap{\BGSection(\sigma)}$ is
  $\cubemap{\BGSection((d_j\sigma))}$, which is very generic by induction. The face
  $\bdy_{j,1}\cubemap{\BGSection(\sigma)}$ is slightly more complicated.
  Let
  $\sigma_{[0,j]}\co \sCVcat[j]\to BG$ be the restriction
  of $\sigma$ to the full subcategory of $\sCVcat[n]$ spanned by
  $\{0,\dots,j\}$. Let $\sigma_{[j,n]}\co \sCVcat[n-j]\to BG$ be the composition
  \[
    \sCVcat[n-j]\to \sCVcat[n]\stackrel{\sigma}{\longrightarrow}BG
  \]
  where the first map is induced by $\{0,\dots,n-j\}\to \{j,\dots,n\}$,
  $i\mapsto i+j$. We have already defined maps
  \begin{align*}
    \cubemap{\BGSection(\sigma_{[0,j]})}&\co [0,1]^{j-1}\to \GJHspace \qquad\qquad\text{and}\\
    \cubemap{\BGSection(\sigma_{[j,n]})}&\co [0,1]^{n-j-1}\to \GJHspace.
  \end{align*}
  The face $\bdy_{j,1}\cubemap{\BGSection(\sigma)}$ is given by the map $[0,1]^{n-2}=[0,1]^{j-1}\times [0,1]^{n-j-1}\to\GJHspace$, 
  \[
    (p,q)\mapsto\bigl(g_{\sigma_{[0,j]}}(p)g_{\sigma_{[j,n]}}(q),(g_{\sigma_{[0,j]}(p)}\cdot\JH_{\sigma_{[j,n]}}(q))\circ \JH_{\sigma_{[0,j]}}(p)\bigr).
  \]
  This map is very generic by Lemma~\ref{lem:exterior-prod-generic}.

  Thus, we have checked that $\cubemap{\BGSection(\sigma)}$ is very generic on the entire boundary of $[0,1]^{n-1}$. It follows from Theorem~\ref{thm:transversality} that $\cubemap{\BGSection(\sigma)}$ admits a very generic extension to the cube $[0,1]^{n-1}$. Choose such an extension and proceed with the induction.
  
  It is immediate from the construction that the resulting map $\BGSection$ is
  a map of simplicial sets.
\end{proof}

Given an $n$-simplex $\sigma\in \NerveWAc\GJHcat$, points
$x,y\in L_0^H\cap L_1$, and a homotopy class
$\phi\in\pi_2^{g(\sigma)}(x,y)$ there is a moduli space of holomorphic
disks
\[
\cM(\phi;\sigma)=\bigcup_{v\in\interior([0,1]^{n-1})}\cM\bigl(\phi;\wt{J}_{\vec{\sigma}}(v),\wt{H}_{\vec{\sigma}}(v)\bigr)
\]
defined in Section~\ref{subsec:moduli}. Given $\sigma\in
\NerveWAc_n\GJHcat$ and $x\in L_0^H\cap L_1$ define
\begin{equation}\label{eq:FloerFuncDef}
  \maxFloerFunc(\sigma)(x)=\sum_{y\in L_0^H\cap L_1}\sum_{\substack{\phi\in\pi_2^{g(\sigma)}(x,y)\\\mu(\phi)=-n+1}}
  \#\cM(\phi;\sigma)y.
\end{equation}
Define $\FloerFunc(J,H)=\CF(L_0^H,L_1;J)$. If $(J,H)$ is very generic
then the usual proof of $\bdy^2=0$ implies that $\FloerFunc(J,H)$ is a
chain complex. If $(J,H)$ is merely generic then a short additional
argument still implies $\bdy^2=0$. Specifically, there is a very generic
cylindrical pair $(J',H')$ arbitrarily close to $(J,H)$. Since the
$0$-dimensional moduli spaces with respect to $(J,H)$ are
transversally cut out, for $(J',H')$ close enough to $(J,H)$, the
$0$-dimensional moduli spaces with respect to $(J',H')$ are identified
with those for $(J,H)$. Thus, the differential with respect to $(J,H)$
and $(J',H')$ are the same, and since the differential with respect to
$(J',H')$ satisfies $\bdy^2=0$, so does the differential with respect
to $(J,H)$. (This argument recurs, in a parametric version, in the
proof of Lemma~\ref{lem:is-sset-map}.)

We digress briefly to discuss gradings. In order to apply the
homotopical algebra from Section~\ref{sec:background} 
we need $\ZZ$-gradings on our
chain complexes. To this end, note that the Floer complex
$\CF(L_0^H,L_1)$ decomposes as a direct sum
\[
  \bigoplus_{\spinc\in\pi_0(P(L_0,L_1))}\CF(L_0^H,L_1;\spinc)
\]
over homotopy classes $\spinc$ of paths from $L_0$ to $L_1$. The group
$G$ acts on $P(L_0,L_1)$ and hence on $\pi_0(P(L_0,L_1))$.  For each
$G$-orbit $[\spinc]\in \pi_0(P(L_0,L_1))/G$ for which there is an
$x\in L_0^H\cap L_1$ so that the constant path $x$ represents
$[\spinc]$, choose a representative $\spinc\in P(L_0,L_1)$ and a base
generator $x_\spinc\in L_0^H\cap L_1$ representing $\spinc$. We
declare that each of these base generators has grading $0$,
$\gr(x_\spinc)=0$. Next, given any other generator
$y\in L_0^H\cap L_1$, choose a $g\in G$ and homotopy class
$\xi_y\in\pi_2^g(x_\spinc,y)$ and define $\gr(y)=-\mu(\xi_y)$. By
Hypothesis~\ref{item:J3}, $\gr(y)$ is independent of the choices of
$g$ and $\xi_y$. In the case the Lagrangians $L_0$ and
$L_1$ are endowed with $G$-equivariant gradings in the sense of
Section~\ref{sec:G-or}, up to an overall shift the induced grading of the Floer complexes is
a case of the grading defined in this paragraph.

\begin{lemma}
  Given an $n$-simplex $\sigma\in\NerveWAc_n\GJHcat$, the map
  $\maxFloerFunc(\sigma)$  is homogeneous of degree $n-1$.
\end{lemma}
\begin{proof}
  Suppose $y$ occurs in $\maxFloerFunc(\sigma)(x)$. Then there
  is a homotopy class $\phi\in\pi_2^{g(\sigma)}(x,y)$ with
  $\mu(\phi)=-n+1$. We have 
  \[
    \xi_x*\phi*\xi_y^{-1}\in\pi_2^{g_xg(\sigma)g_y^{-1}}(x_\spinc,x_\spinc),
  \]
  so by Hypothesis~\ref{item:J3}
  \[
    -\gr(x)-n+1+\gr(y)=\mu(\xi_x)+\mu(\phi)-\mu(\xi_y)=\mu(\xi_x*\phi*\xi_y^{-1})=0,
  \]
  so $\gr(y)=\gr(x)+n-1$, as desired.
\end{proof}

\begin{lemma}\label{lem:is-sset-map}
  The map $\FloerFunc\co \NerveWAc \GJHcat\to\ChainComplexes$ is a map
  of simplicial sets.
\end{lemma}
\begin{proof}
  We need to check that the maps $\maxFloerFunc(\sigma)$
  satisfy Formula~(\ref{eq:chain-cx-cat-max}).  First, fix a very acme
  simplex $\sigma\in\NerveAc_n\GJHcat$. Let $x,y\in L_0^H\cap L_1$ and
  $\phi\in\pi_2^{g(\sigma)}(x,y)$ with $\mu(\phi)=2-n$. By
  Formulas~(\ref{eq:gluing}) and~(\ref{eq:gluing-cube}),
  \begin{equation}\label{eq:FloerFuncPf}\begin{split}
    0&=\sum_{i=1}^{n-1}
    (-1)^{n+i}\#\cM(\phi;\sigma|_{\bdy_{i,0}[0,1]^{n-1}})
    +\sum_{i=1}^{n-1}
    (-1)^{n+i+1}\#\cM(\phi;\sigma|_{\bdy_{i,1}[0,1]^{n-1}})\\
    &\qquad\qquad+\sum_{z\in L_0^H\cap L_1}
    \sum_{\substack{\psi\in\pi_2^{1}(x,z),\ \psi'\in\pi_2^{g(\sigma)}(z,y)\\\psi*\psi'=\phi\\
        \mu(\psi)=1,\ \mu(\psi')=-n+1}}\hspace{-2em}
    (-1)^n\bigl(\#\cM(\psi';\sigma)\bigr)\bigl(\#\cM(\psi;J,H)\bigr)\\
    &\qquad\qquad+
    \sum_{z\in L_0^H\cap L_1}
    \sum_{\substack{\psi\in\pi_2^{g(\sigma)}(x,z),\ \psi'\in\pi_2^{1}(z,y)\\\psi*\psi'=\phi\\
    \mu(\psi)=-n+1,\ \mu(\psi')=1}}\hspace{-2em}
    \bigl(\#\cM(\psi';J,H)\bigr)\bigl(\#\cM(\psi;\sigma)\bigr).
  \end{split}\end{equation}
  Summing over $\phi$, this says
  \begin{align*}
    0&=\sum_{i=1}^{n-1}\Bigl[(-1)^{n+i}\maxFloerFunc(d_i\sigma)(x)+(-1)^{n+i+1}\maxFloerFunc(\sigma_{i,\dots,n})\bigl(\maxFloerFunc(\sigma_{0,\dots,i})(x)\bigr)\Bigr]\\
    &\qquad\qquad+(-1)^n\maxFloerFunc(\sigma)(\bdy x)+\bdy\bigl(\maxFloerFunc(\sigma)(x)\bigr)
  \end{align*}
  which is exactly Formula~(\ref{eq:chain-cx-cat-max}).

  Finally, given an acme simplex $\sigma$, we can perturb $\sigma$
  slightly to be very acme. (This perturbation is done inductively,
  over the faces of $\sigma$.) Since the $0$-dimensional moduli spaces
  with respect to $\cubemap{\sigma}$ and its faces are
  transversally cut out, for a small enough perturbation this does not
  change the count of $\cM(\psi;\sigma)$ for any $\psi$ with
  $\mu(\psi)=-n+1$ or, indeed, any of the counts in
  Formula~\eqref{eq:FloerFuncPf}. Thus, the argument above, applied to
  the perturbed simplex, shows that
  Formula~(\ref{eq:chain-cx-cat-max}) still holds.
\end{proof}

Let
\[
  \chainFb\co \Nerve BG\to\ChainComplexes
\]
be the composition of $\FloerFunc$ and the section $\BGSection$ from
Lemma~\ref{lem:sec-exists},
\[
  \chainFb=\FloerFunc\circ \BGSection.
\]

\subsection{Equivariant Floer cohomology}\label{seq:Floer-coho}
\begin{definition} \label{def:equivariant-cohomology}
  The \emph{equivariant Floer cohomology} of $(M,L_0,L_1)$, denoted
  $\eHF[G](L_0,L_1)$, is the hypercohomology of the diagram $\chainFb\co\Nerve BG\to\ChainComplexes$. 
\end{definition}

As discussed in Section~\ref{sec:homotopy-colimits},
$\eHF[G](L_0,L_1)$ is a module over $H^*(|\Nerve BG|)=H^*(BG)$.

\begin{theorem}\label{thm:eHF-inv}
  Up to isomorphism over $H^*(BG)$, $\eHF[G](L_0,L_1)$ is independent
  of the choices in its construction. Further, if $L'_0$ and $L'_1$ are isotopic to $L_0$ and
  $L_1$ via compactly supported, $G$-equivariant Hamiltonian isotopies
  then $\eHF[G](L_0,L_1)\cong\eHF[G](L'_0,L'_1)$ as modules over
  $H^*(BG)$.
\end{theorem}
\begin{proof}
  For the first statement, suppose that $(J_0,H_0)$ and $(J_1,H_1)$
  are generic cylindrical pairs. Let $\GJHcat_i$ be the category
  $\GJHcat$ defined using $(J_i,H_i)$ as the base cylindrical pair,
  and suppose that $\BGSection_i\co \Nerve BG\to\NerveWAc\GJHcat_i$ is a
  section of $\Nerve\pi$ for $i=0,1$. Consider a larger smooth category
  $\GJHcat'$ with:
  \begin{itemize}
  \item Objects $\{(J_0,H_0),(J_1,H_1)\}$.
  \item For $i=0,1$,
    $\Hom_{\GJHcat'}((J_i,H_i),(J_i,H_i))=\Hom_{\GJHcat_i}((J_i,H_i),(J_i,H_i))$
    (so $\GJHcat_i$ is a full subcategory of $\GJHcat'$).
  \item The space $\Hom_{\GJHcat'}((J_0,H_0),(J_1,H_1))$ consists of pairs
    $(g,\JH)$ where $g\in G$ and $\JH$ is a semi-cylindrical pair with
    $\JH^{-\infty}=(J_0,H_0)$ and $\JH^{+\infty}=g\cdot (J_1,H_1)$.
  \item $\Hom_{\GJHcat'}((J_1,H_1),(J_0,H_0))$ is empty.
  \item Composition is defined by multiplying maps to $G$ and
    composing (stacking) semi-cylindrical pairs.
  \item Smoothness is as in Definition~\ref{def:smooth-family}. 
  \end{itemize}
  Define the acme nerve $\NerveWAc\GJHcat'$ of $\GJHcat'$ similarly to
  Definition~\ref{def:acme}.
  
  Recall from Section~\ref{sec:smooth-nerves} that to define a natural
  transformation from $\BGSection_0$ to $\BGSection_1$ it suffices to define a
  map $\BGSection_{01}\co \Nerve(\ICat\times BG)\to \NerveWAc\GJHcat'$ extending
  $\BGSection_0$ and $\BGSection_1$. An inductive argument analogous to the
  proof of Lemma~\ref{lem:sec-exists} shows that we can find such a natural
  transformation $\BGSection_{01}$. Applying Floer theory then gives a functor
  \[
    \FloerFunc\circ \BGSection_{01}\co \Nerve (\ICat\times BG)=\Delta^1\times\Nerve(BG)\to\ChainComplexes
  \]
  extending the $(\FloerFunc)_i$.
  By Lemma~\ref{lem:hocolim-map}, there is an induced map of modules 
  \[
    (\FloerFunc\circ \BGSection_{01})^*\co H^*(\BGSection_1)\to 
    H^*(\BGSection_0).
  \]
  Exchanging the roles of $\BGSection_0$ and $\BGSection_1$ in the
  construction above gives a natural transformation $\FloerFunc\circ
  \BGSection_{10}$ from $\FloerFunc\circ \BGSection_1$ to
  $\FloerFunc\circ \BGSection_0$. Construct two more
  sections, $\BGSection_{010}$ and $\BGSection_{101}$, over $\Nerve[(0\to 1\to 2)\times BG]$, so that
  \begin{align*}
    \BGSection_{010}|_{\Nerve [(0\to 1)\times BG]}&=\BGSection_{01}\\ 
    \BGSection_{010}|_{\Nerve [(1\to 2)\times BG]}&=\BGSection_{10} \\
    \BGSection_{010}|_{\Nerve [(0\to 2)\times BG]}&=(\pi_{\Nerve(0\to2)},\BGSection_{0}\circ \pi_{\Nerve BG}) 
  \end{align*}
  where $\pi_{\Nerve(0\to 2)}\co \Nerve [(0\to 2)\times BG]\to \Nerve(0\to 2)$
  and $\pi_{\Nerve BG}\co \Nerve [(0\to 2)\times BG]\to \Nerve BG$ are the
  projections; and similarly for $\BGSection_{101}$ (with $1$ and $0$ reversed).
  We get functors
  $\Nerve((0\to 1\to 2)\times BG)=\Delta^2\times\Nerve BG\to\Complexes$. By
  dimension counting,
  $\FloerFunc\circ\BGSection_{010}|_{\Nerve [(0\to 2)\times BG]}$ and
  $\FloerFunc\circ\BGSection_{101}|_{\Nerve [(0\to 2)\times BG]}$ induce the
  identity maps on the homotopy colimits.
  Therefore, we conclude from Lemma~\ref{lem:hocolim-htpy} that the maps of homotopy colimits induced by
  $\FloerFunc\circ \BGSection_{01}$ and $\FloerFunc\circ\BGSection_{10}$
  are homotopy inverses to each other. It follows from
  Lemmas~\ref{lem:hocolim-map} and~\ref{lem:hocolim-htpy} that
  $H^*(\BGSection_0)\cong H^*(\BGSection_1)$ as modules over
  $H^*(BG)$, as desired.
  
  The proof of invariance under $G$-equivariant Hamiltonian isotopies
  is similar, but with $\GJHcat'$ replaced by a slightly different
  category. Write $L_i^t$ for the $G$-equivariant
  Hamiltonian isotopy from $L_i$ to $L'_i$, where $t\in\RR$, and
  $L_i^t=L_i$ for $t\ll 0$ while $L_i^t=L'_i$ for $t\gg 0$. Fix 
  generic cylindrical pairs $(J_0,H_0)$ for $(L_0,L_1)$ and
  $(J_1,H_1)$ for $(L'_0,L'_1)$.  Given $g\in G$, by an
  \emph{interpolating Lagrangian isotopy} from $L_0^{H_0}$ to
  $g((L_0')^{H_1})=(L'_0)^{gH_1}$ we mean a family of compactly
  supported Hamiltonians $H_t$, $t\in\RR$, so that $H_t=H_0$ for $t\ll
  0$ and $H_t=gH_1$ for $t\gg 0$, and a family of Lagrangians
  $L_0^{t,H_t}$, $t\in\RR$, which are time-one flows of $L_0^t$ under
  $H_t$, $L_0^{t,H_t}=(L_0^t)^{H_t}$.

  Now, let $\GJHcat'$ be the smooth category with:
  \begin{itemize}
  \item Two objects, $(J_0,H_0)$ and $(J_1,H_1)$ where $(J_0,H_0)$ is
    a generic cylindrical pair for $(L_0,L_1)$ and $(J_1,H_1)$ is a
    generic cylindrical pair for $(L'_0,L'_1)$.
  \item For $i=0,1$,
    $\Hom_{\GJHcat'}((J_i,H_i),(J_i,H_i))=\Hom_{\GJHcat_i}((J_i,H_i),(J_i,H_i))$. (Here,
    $\GJHcat_0$ is the category $\GJHcat$ as defined in
    Section~\ref{sec:build-diag} for $(L_0,L_1,J_0,H_0)$, while
    $\GJHcat_1$ is the category $\GJHcat$ as defined in
    Section~\ref{sec:build-diag} for $(L'_0,L'_1,J_1,H_1)$).
  \item $\Hom_{\GJHcat'}((J_1,H_1),(J_0,H_0))=\emptyset$.
  \item $\Hom_{\GJHcat'}((J_0,H_0),(J_1,H_1))$ is the space of pairs
    $(g,(\wt{J},\wt{L}))$ where $g\in G$, $\wt{J}$ is a (possibly multi-story)
    cylindrical-at-infinity almost complex structure and $\wt{L}=\{L_0^{t,H_t}\}$ is an
    interpolating Lagrangian isotopy (for $g$), again perhaps with several
    stories. We require that $\wt{J}$ agrees at $-\infty$ (of the bottom story)
    with $J_0$ and at $+\infty$ (of the top story) with $g\cdot J_1$.
  \item Smoothness for cubes is as in Definition~\ref{def:smooth-family}.
  \end{itemize}
  Again, the category $\GJHcat'$ has an acme nerve $\NerveWAc\GJHcat'$.
  Similar arguments to the proof of Lemma~\ref{lem:sec-exists} produce
  a section $\Nerve (\ICat\times BG)\to \NerveWAc \GJHcat'$.
  Applying Floer theory with the dynamic boundary conditions
  $L_0^{t,H_t}$ and $L_1^t$ gives a functor
  $\FloerFunc\co \NerveWAc \GJHcat'\to\ChainComplexes$. The composition
  \[
    \Delta^1\times (\Nerve BG)=\Nerve(\ICat\times BG)\to \NerveWAc\GJHcat'\to \ChainComplexes
  \]
  is a natural transformation between the equivariant diagrams before
  and after the Hamiltonian isotopy. The same argument as in the proof
  of independence of the perturbation data then implies that the
  equivariant cohomologies before and after the Hamiltonian isotopy
  agree.
\end{proof}

\subsection{An analogous construction in Morse theory}\label{sec:Morse-const}
In this section, we explain how to adapt the construction of $\eHF[G]$
to the setting of Morse homology. In Section~\ref{sec:comp-morse} we
show that the resulting equivariant Morse cohomology agrees with
the usual Borel equivariant cohomology.

Fix a compact, connected, smooth manifold $M$ and a smooth action of a
Lie group $G$ on $M$. To reduce confusion, either assume that $M$ is orientable and $G$
acts by orientation-preserving diffeomorphisms $M\to M$ or else take
coefficients in $\Field_2$ below.

By a \emph{cylindrical Morse pair} we mean a
pair $(\metric,f)$ where $\metric$ is a Riemannian metric on $M$ and
$f$ is a Morse function on $M$. A cylindrical Morse pair is
\emph{generic} if $(\metric,f)$ is Morse-Smale. A
\emph{semi-cylindrical Morse pair} is a triple
$(\Line,\tmetric,\wt{f})$ where $\Line$ is a marked line, $\tmetric$
is a family of Riemannian metrics parameterized by $\Line$, and
$\wt{f}$ is a smooth function $\Line\times M\to\RR$ so that $\tmetric$
and $\wt{f}$ are constant outside the intervals $[p_i,q_i]$ in
$\Line$. The space of semi-cylindrical Morse pairs has a topology
similar to the topology on $\JSpace$, which we leave to the reader to
spell out.

There is a left action of $G$ on the space of cylindrical Morse
pairs, by $(g\cdot f)(x)=f(g^{-1}x)$ and
$(g\metric) (v,w) =\langle g^{-1}_*v,g^{-1}_*w\rangle$.

Fix a cylindrical Morse pair $(\metric,f)$. Consider the space
$\MPspace$ of triples $(g,\tmetric,\wt{f})$, where $g\in G$
and $(\tmetric,\wt{f})$ is a semi-cylindrical Morse pair connecting
$(\metric,f)$ at $-\infty$ and $g\cdot(\metric,f)$ at $+\infty$. There is a
projection $\MPspace\to G$, and a composition map
$\circ\co \MPspace\times\MPspace\to\MPspace$,
$(g',\tmetric',\wt{f}')\circ(g,\tmetric,\wt{f})=
(gg',(g\cdot\tmetric')*(\tmetric),(g\cdot \wt{f}')*\wt{f})$
(where $*$ denotes stacking).  Given $x,y\in\Crit(f)$, and
$(g,\tmetric,\wt{f})\in\MPspace$, where $\tmetric$ and $\wt{f}$ are
parameterized by some marked line $\Line$, a \emph{flow line from $x$
  to $y$ with respect to $(g,\tmetric,\wt{f})$} is a smooth map
$\gamma=\gamma(t)\co \Line\to M$ so that
\begin{equation}\label{eq:grad-flow-line}
\gamma'(t)=-\nabla_{\gamma(t)}^{\tmetric(t)}\wt{f}(t)
\end{equation}
and so that $\lim_{t\to-\infty}\gamma(t)=x$ and
$\lim_{t\to+\infty}\gamma(t)=gy$. Let $\cM(x,y;g,\tmetric,\wt{f})$
denote the space of flow lines from $x$ to $y$ with respect to
$(g,\tmetric,\wt{f})$; we leave topologizing the space
$\cM(x,y;g,\tmetric,\wt{f})$ to the reader. More generally, given a
polyhedron $P^n$ and
map $\sigma=(g_{\sigma},\tmetric_{\sigma},\wt{f}_{\sigma})\co P^n\to \MPspace$ let
\[
  \cM(x,y;\sigma)=\bigcup_{v\in P^n}\cM(x,y;g_{\sigma}(v),\tmetric_{\sigma}(v),\wt{f}_{\sigma}(v)).
\]
The expected dimension of $\cM(x,y;\sigma)$ is $\ind(x)-\ind(y)+n$.
We call $\sigma$ \emph{generic} if the spaces $\cM(x,y;\sigma)$ are
transversely cut out whenever $\ind(x)-\ind(y)\leq -n+1$, and the
restriction of $\sigma$ to every boundary face is also generic. (This
is the analogue of ``very generic'' above, but we will have no use for
the analogue of the weaker ``generic'' notion.)

To orient the moduli spaces $\cM(x,y;\sigma)$ choose an orientation of
$M$ and an orientation of the descending disk $D_d(x)$ for the flow
of $-\nabla^{\metric}f$ from each point $x\in\Crit(f)$. (This is equivalent
to choosing an orientation for the negative eigenspace of the Hessian
$\Hess_x(f)$.) There is an induced orientation of the ascending disk
$D_a(x)$ for each $x\in\Crit(f)$ by requiring that
$D_a(x)\cap D_d(x)$ is positive. Via the action of $g$, there
is then an induced orientation of the ascending disk $D_a(gx)$ of the
point $gx\in\Crit(g\cdot f)$. For the moduli spaces of flow lines with
respect to the cylindrical pair $(\metric,f)$, we also have an identification
\[
  \RR\times \cM(x,y;\metric,f)\cong D_a(y)\cap D_d(x)
\]
for $x\neq y$ (cf.\ Equation~\eqref{eq:R-side})
so the orientations of $D_d(x)$, $D_a(y)$, $M$, and $\RR$ induce an
orientation of $\cM(x,y;\metric,f)$.

Now, suppose that $(\tmetric,\wt{f})$ is a $1$-story semi-cylindrical
Morse pair. If we let
$D_d(x;\tmetric,\wt{f})$ denote the space of solutions of
Equation~\eqref{eq:grad-flow-line} asymptotic to $x$ at $-\infty$ then
we can identify $D_d(x;\tmetric,\wt{f})$ with $D_d(x)$ as
follows. Choose $T\ll0$ small enough that $\wt{f}(t)=f$ for all
$t\leq T$. Then $\gamma\in D_d(x;\tmetric,\wt{f})$ if and only if
$\gamma$ is a solution to Equation~\eqref{eq:grad-flow-line} and
$\gamma(T)\in D_d(x)$, so the map $\gamma\mapsto \gamma(T)$ identifies
$D_d(x;\tmetric,\wt{f})$ and $D_d(x)$. In particular, $D_d(x;\tmetric,\wt{f})$
inherits an orientation from the orientation of $D_d(x)$. Similarly,
the space $D_a(gx;\tmetric,\wt{f})$ of solutions of
Equation~\eqref{eq:grad-flow-line} asymptotic to $gx$ at $+\infty$
inherits an orientation from $D_a(gx)$. The map $\gamma\mapsto
\gamma(0)$ maps $\cM(x,y;g,\tmetric,\wt{f})$, $D_d(x;\tmetric,\wt{f})$, and
$D_a(gy;\tmetric,\wt{f})$ into $M$, and
\[
  \cM(x,y;g,\tmetric,\wt{f})=D_a(gy;\tmetric,\wt{f})\cap D_d(x;\tmetric,\wt{f}).
\]
Thus, the space $\cM(x,y;g,\tmetric,\wt{f})$ inherits an
orientation. If $(\tmetric,\wt{f})$ is a generic multi-story
semi-cylindrical Morse pair and $\gluing(\tmetric,\wt{f})$ is a gluing
of $(\tmetric,\wt{f})$ with large gluing parameters then there is an
identification 
\[
\cM(x,y;g,\tmetric,\wt{f})\cong \cM(x,y;g,\gluing(\tmetric,\wt{f})),
\]
so the orientation of $\cM(x,y;g,\gluing(\tmetric,\wt{f}))$ induces an
orientation of $\cM(x,y;g,\tmetric,\wt{f})$.  Combining these
orientations with the orientation of the polyhedron $P^n$ gives an orientation
of $\cM(x,y;\sigma)$ (cf.\ Equation~\eqref{eq:det-dbar-decomp}).

We claim that these orientations are coherent in the sense that
Equation~\eqref{eq:gluing} and, in the case 
$P=[0,1]^{n-1}$, Equation~(\ref{eq:gluing-cube}), hold with signs; we leave the proof to the
reader.

There is a smooth category $\MPcat$ with a single object $(\metric,f)$,
$
  \Hom((\metric,f),(\metric,f))=\MPspace,
$
and composition given by
concatenation. The definition of smoothness is analogous to
Definition~\ref{def:smooth-family} and is left to the reader. Projection induces
  a functor $\pi\co \MPcat\to BG$.

Taking smooth nerves, we have a projection
$\pi\co \Nerve \MPcat\to \Nerve BG$. There is a subcategory of
$\Nerve\MPcat$ of acme simplices, defined just as in the Floer case (Definition~\ref{def:acme}). Like
Lemma~\ref{lem:sec-exists}, there is a section
$\BGSection\co \Nerve BG\to \NerveAc \MPcat$. Further, Morse homology gives
a functor $\MorseFunc\co \NerveAc\MPcat\to \ChainComplexes$, defined by
$\MorseFunc(\metric,f)=C_*^{\Morse}(\metric,f)$, the Morse
complex of the cylindrical pair $(\metric,f)$, and for
$\sigma\in \NerveAc\MPcat$ a simplex of dimension $n>0$,
$\maxMorseFunc(\sigma)$ is defined by Formula~\eqref{eq:FloerFuncDef} but
using the Morse moduli spaces in place of the Floer moduli spaces.

The composition is a diagram
$\MorseFunc\circ \BGSection\co \Nerve BG\to \ChainComplexes$. Taking
hypercohomology gives
\begin{equation}\label{eq:equi-morse-coho}
  \eH[G]^*(M)\coloneqq H^*(\MorseFunc\circ \BGSection).
\end{equation}


\section{Computations}\label{sec:computations}
\subsection{The discrete case}\label{sec:discrete}
Let $G$ be a discrete group acting on $(M, L_0, L_1)$ satisfying the
hypotheses in Section~\ref{subsec:hypotheses}. In particular, if $G$
is not finite then we require that every $G$-twisted loop has Maslov
index $0$. In this section, we prove the following theorem.

\begin{theorem}\label{thm:finite-same} 
  For $\Ring = \FF_2$, the $G$-equivariant Floer cohomology of
  $(M,L_0,L_1)$ constructed in
  Definition~\ref{def:equivariant-cohomology}, agrees with the
  construction of $G$-equivariant Floer cohomology given in our
  previous paper~\cite{HLS:HEquivariant}.
\end{theorem}

In our previous paper~\cite{HLS:HEquivariant} the construction of $G$-equivariant Floer cohomology was carried out for finite $G$, but all of the arguments hold also for the general discrete case. Furthemore, note that the construction of this paper lifts the construction of
equivariant cohomology given in our previous paper~\cite{HLS:HEquivariant}, described there
only over $\FF_2$-coefficients, to $\Ring$-coefficients for an arbitrary ring $\Ring$.

In the special case that there is a $G$-equivariant cylindrical pair $(J,H)$ which is generic (that is, in the presence of equivariant transversality), one may also consider the $G$-equivariant cohomology of the chain complex $\CF(L_0^H,L_1)$ with the induced action of $G$ by chain maps. This equivariant cohomology is the module $\Ext_{\Ring[G]}(\CF(L_0^H,L_1), \Ring)$, where $\Ring$ is the trivial $\Ring[G]$ module. (This is the theory used in~\cite[Section 2b and 3]{SeidelSmith10:localization}.) In our previous paper, we showed the the new construction of $G$-equivariant Lagrangian Floer cohomology given there is identified with this simpler construction when a generic $G$-equivariant $(J,H)$ exists~\cite[Section 4]{HLS:HEquivariant}. It follows that the same is true of the $G$-equivariant Floer cohomology constructed in Definition~\ref{def:equivariant-cohomology}:

\begin{corollary} 
  Suppose there exists a generic cylindrical pair $(J,H)$ which is
  invariant under the action of $G$, that is, such that
  $g(J,H) = (J,H)$ for all $g \in G$. Then the $G$-equivariant Floer
  cohomology constructed in Definition
  \ref{def:equivariant-cohomology} agrees with the $G$-equivariant
  cohomology $\Ext_{\Ring[G]}(\CF(L_0^H,L_1),\Ring)$.
\end{corollary}

\begin{proof} 
  This follows from Theorem \ref{thm:finite-same} and an adaptation of
  the proof of \cite[Proposition 4.5]{HLS:HEquivariant}, originally
  proved over $\Field_2$-coefficients, to $\Ring$-coefficients. Including
  the signs proceeds without incident.
\end{proof}

Turning to the proof of Theorem~\ref{thm:finite-same}, we begin
with a comparison of the spaces of almost complex structures
appearing in this paper and the previous paper
\cite{HLS:HEquivariant}. In the previous paper, the main construction
was carried out for Lagrangians $L_0$ and $L_1$ intersecting
transversely, so there was no need to include the data of a
Hamiltonian function. (In \cite[Section 3.6]{HLS:HEquivariant} we did
consider the case of non-transverse intersections.) We
considered the topological category $\ol{\JSpace}$ with objects
consisting of cylindrical almost complex structures $J$. The spaces of
morphisms $\ol{\JSpace}(J_a,J_b)$ consisted of multi-story semicylindrical
almost complex structures with limits $J^{-\infty}=J_a$,
$J^{\infty}=J_b$, up to an equivalence relation analogous to that in
Section~\ref{sec:J-space}.  In this paper we have used the space $\GJHspace$ of elements
$(g,(\wt{J}, \wt{H}))$ with $(\wt{J}, \wt{H})^{-\infty} = (J,H)$ and
$(\wt{J},\wt{H})^{\infty} = g(J,H)$. Because $G$ is discrete, there is a continuous map
$\eta\co\GJHspace \to \coprod_{J,J'\in\ob(\ol{\JSpace})}
{\ol{\JSpace}}(J,J')$, sending $(g,(\wt{J},\wt{H}))\in\GJHspace$ to $\wt{J}\in{\ol{\JSpace}}(J,gJ)$.

With this in mind, we turn our attention to the functor
$\BGSection \co \Nerve B G \to \NerveAc \GJHcat$ from
Section~\ref{sec:construction}. Recall that $BG$ is the smooth category with one object $o$ and $\Hom(o,o)=G$, with composition given by
$g \circ h = hg$. There is exactly one smooth map $[0,1]^n \rightarrow G$ for each $n$ and each element $g \in G$, the constant map. Ergo $n$-simplices of the smooth nerve $\Nerve B G$
correspond to ordered $n$-tuples $(g_n, \dots, g_1)$ of elements of
$G$. For example, a triple $(g_3,g_2,g_1)$ corresponds to a square of maps as in the left-hand side of Example~\ref{exam:3-cell-in-nbg}, with cube maps as follows:

\begin{itemize}
\item $\cubemap{\sigma}_{01} = g_1$, $\cubemap{\sigma}_{12} = g_2$, $\cubemap{\sigma}_{23} = g_3$, $\cubemap{\sigma}_{02} = g_1g_2$, $\cubemap{\sigma}_{13} = g_2g_3$, $\cubemap{\sigma}_{03} = g_1g_2g_3$.

\item $\cubemap{\sigma}_{012}$ is the constant path at $g_1g_2$, $\cubemap{\sigma}_{123}$ is the constant path at $g_2g_3$, $\cubemap{\sigma}_{023} = \cubemap{\sigma}_{013}$ are both the constant path at $g_1g_2g_3$.

\item $\cubemap{\sigma}$ is the constant map from $[0,1]^2$ to $g_1g_2g_3$.
\end{itemize}

 Consider the diagram
$\BGSection \co \Nerve BG \to \NerveAc \GJHcat$ described
in Lemma \ref{lem:sec-exists}. The functor $\BGSection$ associates to
every element in $G$ a point
$\BGSection(g)=(g,(\wt{J},\wt{H}))$ with
$(\wt{J}, \wt{H})^{-\infty} = (J,H)$ and
$(\wt{J},\wt{H})^{\infty} = g(J,H)$. Because $\BGSection$ respects
the degeneracy maps in $\Nerve BG$ and
$\NerveAc \GJHcat$, $\BGSection(e) = (e, (\wt{J},\wt{H}))$ has $\wt{J}(v)=J$ and $\wt{H}(v)=H$ for all $v \in \mathbb R$. Furthermore, to each $n$-simplex in $\Nerve BG$, or
equivalently each sequence of elements $(g_n,\dots,g_1)$,
$\BGSection$ associates a smooth $(n-1)$-dimensional cube of
semicylindrical pairs
\[
  \BGSection(g_n,\dots, g_1) \co [0,1]^{n-1} \to \GJHspace.
\]

For example, $\BGSection$ associates to the triple $(g_3,g_2, g_1)$ a map from the square $[0,1]^2$ into $\GJHspace$; this is a special case of Example~\ref{eg:ngj-3cell}.

Observe that in the case that $L_0\pitchfork L_1$, we may choose
$\BGSection$ so that $\BGSection(o)=(J,H)$ has $H \equiv 0$ and for
each simplex $\sigma\in \Nerve_nBG$,
$\cubemap{\BGSection(\sigma)}=\bigl(g_{\vec{\sigma}},(\wt{J},\wt{H})_{\vec{\sigma}}\bigr)\co
[0,1]^{n-1} \to \GJHspace$ has $\wt{H}\equiv 0$. Therefore the same is
true of the maps $\BGSection(g_n, \dots, g_1)$ described above. Since
all of our Hamiltonians are trivial, to condense notation we will let
\[
\cM(\phi; \wt{J}_{\sigma})=\cM(\phi; \wt{J}_{\sigma}, \wt{H}_{\sigma}).
\]
 
Next we express the information above in the language of our
previous paper~\cite{HLS:HEquivariant}. Our construction in that paper
used the notion of a homotopy coherent functor. We will require a superficially slightly different (although
equivalent) definition to the one appearing there. As in
Section~\ref{sec:background}, this is an adaptation of Vogt~\cite{Vogt73:hocolim}.

\begin{definition}
  Let $\Cat$ be a small category. A homotopy coherent $\Cat$-diagram in
  $\ol{\JSpace}$ consists of:
  \begin{itemize}
  \item For each object $x$ of $\Cat$, a cylindrical almost complex
    structure $F(x)$, and
  \item For each integer $n\geq 1$ and each sequence
    $x_0\stackrel{f_1}{\longrightarrow}\cdots\stackrel{f_n}{\longrightarrow}
    x_n$ of composable morphisms a continuous
    map $F(f_n,\dots,f_1)\co [0,1]^{n-1}\to\ol{\JSpace}$,
\end{itemize}    
such that for all $(t_{n-1},\dots,t_{1})\in[0,1]^{n-1}$,
\[
  F(f_n,\dots,f_1)(t_{n-1},\dots,t_{1})^{-\infty} =F(x_0) 
  \qquad\text{and}\qquad
  F(f_n,\dots,f_1)(t_{n-1},\dots,t_{1})^{\infty} =F(x_n),
\]
and such that
  \begin{gather*}
    \begin{split}
      F(f_n,\dots,f_2,\Id)(t_{n-1},\dots,t_{1})&=F(f_n,\dots,f_2)(t_{n-1},\dots,t_{2})\\
      F(\Id,f_{n-1},\dots,f_1)(t_{n-1},\dots,t_{1}) &= F(f_{n-1},\dots,f_1)(t_{n-2},\dots,t_{1})
    \end{split}\\
    \begin{split}
      F(f_n,\dots,f_{i+1},&\Id,f_{i-1},\dots,f_1)(t_{n-1},\dots,t_{1}) \\
      &= F(f_n,\dots,f_{i+1},f_{i-1},\dots,f_1)(t_{n-1},\dots,t_i \cdot t_{i+1},\dots,t_{1})\\
      F(f_n,\dots,f_1)&(t_{n-1},\dots,t_{i+1},1,t_{i-1},\dots,t_{1}) \\
      &= F(f_n,\dots,f_{i+1}\circ f_i,\dots,f_1)(t_{n-1},\dots,t_{i+1},t_{i-1},\dots,t_{1})\\
      F(f_n,\dots,f_1)&(t_{n-1},\dots,t_{i+1},0,t_{i-1},\dots,t_{1})\\
      &=
      [F(f_n,\dots,f_{i+1})(t_{n-1},\dots,t_{i+1})]\circ [F(f_i,\dots,f_1)(t_{i-1},\dots,t_{1})].
    \end{split}
  \end{gather*}
\end{definition}

The ordering of the factors in $[0,1]^{n-1}$ above differs from the definition given in our previous paper \cite[Definition 3.3]{HLS:HEquivariant}. This
choice is for compatibility with
the signs in Section~\ref{sec:background}. 

Now let $\BCat G$ be the category with a single object $o$ and
$\Hom(o,o)=G$, and let $\ECat G$ be the
category with an object for every element of $G$ and a unique morphism
between any two objects. Let $G$ act on $\ECat G$ on the left. 
(This is not quite standard, but necessary in order to match the 
composition in $\BCat G$.)   There is an identification between $n$-simplices
in $\Nerve B G$ and sequences of $n$-composable morphisms in
$\BCat G$, since both correspond to $n$-tuples $(g_n, \dots, g_1)$ of
elements in $G$. Sequences of $n$-composable morphisms in $\ECat G$ are
of the form
\[
g \xrightarrow{g_1} gg_1 \xrightarrow{g_2} gg_1g_2 \to \dots \xrightarrow{g_n} gg_1\cdots g_n.
\]
Call this sequence $(g; g_n, \dots, g_1)$. 
We can use $\BGSection$ to give a $G$-equivariant homotopy coherent diagram
$F\from\ECat G \to \ol{\JSpace}$ as follows. For every $g \in G$, let
$F(g)=g\BGSection(o) = gJ$. For every sequence of morphisms
$(g; g_n, \dots, g_1)$ in $\ECat G$, let $F(g; g_n,\dots, g_1)
=g(\eta(\BGSection(g_n,\dots, g_1)\circ (\vec{t}\mapsto \vec{1}-\vec{t})))\co [0,1]^{n-1}\to \ol{\JSpace}$.
Here, $\vec{t}\mapsto \vec{1}-\vec{t}$ denotes reflection of each $t_i$-coordinate in
$[0,1]^{n-1}$. (Compare Remark~\ref{rem:max-vs-smooth-max}.)
Because $\BGSection$ respects the simplicial relations in $\Nerve B G$,
$F$ gives a homotopy coherent diagram. 

We will also need the notion of a homotopy coherent diagram of chain
complexes from a small category $\Cat$. Once again, this is a slight modification of
the definition in the previous paper in order to better align
with the signs of Section~\ref{sec:background}.

\begin{definition}\label{def:ho-coh-diag-cxs}
  Let $\Cat$ be a small category. Then a homotopy coherent
  $\Cat$-diagram in chain complexes consists of:
 \begin{itemize}
 \item For each object $x$ of $\Cat$, a chain complex $T(x)$.
 \item For each $n\geq 1$ and each sequence 
   $x_0\stackrel{f_1}{\longrightarrow}\cdots\stackrel{f_n}{\longrightarrow} x_n$
   of composable morphisms a chain map $T(f_n,\dots,f_1)\co
   I_*^{\otimes (n-1)}\otimes T(x_0)\to T(x_n)$
 \end{itemize}
 such that
 \begin{multline*}
   T(f_n,\dots,f_1)(t_{n-1}\otimes\dots\otimes t_1)\\
   =
   \begin{cases}
     T(f_n,\dots,f_2)(t_{n-1}\otimes\dots\otimes t_2 \otimes \pi(t_{1})) & f_1=\Id\\
     T(f_n,\dots,f_{i+1},f_{i},\dots,f_1)(t_{n-1} \otimes \dots \otimes m(t_{i}\otimes t_{i-1}) \otimes \dots \otimes t_1) & f_i=\Id, 1<i<n\\
     T(f_{n-1},\dots,f_{1})(\pi(t_{n-1})\otimes t_{n-2} \otimes\dots\otimes t_1) & f_n=\Id\\
     T(f_n,\dots,f_{i+1}\circ f_i,\dots,f_1)(t_{n-1}\otimes\dots \otimes t_{i+1}\otimes t_{i-1}\otimes \dots\otimes t_{1}) & t_i=\{1\}\\
     [T(f_n,\dots,f_{i+1})(t_{n-1}\otimes\dots\otimes t_{i+1})]\circ [T(f_i,\dots,f_1)(t_{i-1}\otimes\dots\otimes t_{1})] & t_i=\{0\}
   \end{cases}
 \end{multline*}
 as maps $T(x_0)\to T(x_n)$.
 \end{definition}

 \begin{lemma}
   A homotopy coherent diagram of chain complexes is determined by the
   complexes $T(x)$ and the maps
   \[
     T_{f_n,\dots,f_1} =
     T(f_n,\dots,f_1)(\{0,1\}\otimes\cdots\otimes\{0,1\})\co T(x_0)\to
     T(x_n)
   \]
   which satisfy compatibility conditions
   \begin{enumerate}[label=(hc-\arabic*)]
   \item $T_{\Id_{x_0}} = \Id_{T(x_0)}$.
   \item $T_{f_n,\dots,f_1} = 0$ if $n>1$ and any $f_i = \Id$.
   \item\label{item:hc3}
     $\partial \circ T_{f_n,\dots,f_1} + (-1)^{n}T_{f_n,\dots,f_1}
     \circ \partial = \sum \limits_{0<i<n} (-1)^{n-i-1} [T_{f_n,\dots,
       f_{i+1} \circ f_i, \dots,f_1}- T_{f_n,\dots,f_{i+1}} \circ
     T_{f_i,\dots,f_1}].$
   \end{enumerate}
 \end{lemma}
\begin{proof}
  The only non-obvious part of this assertion is the signs in the last
  equation, which we now verify. Observe that
\begin{align*}
\partial&\bigl( T_{f_n,\dots, f_1}(z)\bigr)=\partial \bigl( T(f_n,\dots,f_1)(\{0,1\}\otimes \cdots \otimes \{0,1\}\otimes z)\bigr) \\
&= T(f_n,\dots,f_1)\bigl(\partial (\{0,1\}\otimes \cdots \otimes \{0,1\}\otimes z)\bigr) \\
&= (-1)^{n-1}T(f_n,\dots,f_1)(\{0,1\}\otimes\cdots\otimes \{0,1\}\otimes \partial z) \\ 
& \qquad +\sum_{i=1}^{n-1} (-1)^{n-1+i}[T(f_n,\dots, f_1)(\{0,1\}\otimes \cdots \otimes \{1\} \otimes \cdots \otimes \{0,1\}\otimes z)\\
&\qquad - T(f_n,\dots,f_1)(\{0,1\}\otimes \cdots \otimes \{0\} \otimes \cdots \otimes \{0,1\}\otimes z)]\\
&= (-1)^{n-1}T_{f_n,\dots, f_1}(\partial z)\\ 
&\qquad + \sum_{i=1}^{n-1} (-1)^{n-i+1}[T(f_n,\dots, f_{i+1}\circ f_i,\dots, f_1)(\{0,1\}\otimes\cdots\otimes\{0,1\}\otimes z)\\
&\qquad  - T(f_n,\dots,f_{i+1})(\{0,1\}\otimes \cdots \otimes \{0,1\} \otimes T(f_{i},\dots, f_1)(\{0,1\}\otimes\dots\otimes\{0,1\}\otimes z))]\\
&= (-1)^{n-1}T_{f_n,\dots, f_1} \circ \partial (z) + \sum_{0<i<n} (-1)^{n-i-1} [T_{f_n,\dots, f_{i+1} \circ f_i, \dots,f_1}- T_{f_n,\dots,f_{i+1}} \circ T_{f_i,\dots,f_1}].
\end{align*}
(In the third line, $t_i$ is the term that may be
$\{1\}$ or $\{0\}$, and all other $t_j$ are $\{0,1\}$.) This completes
the proof.
\end{proof}

In our previous paper \cite[Section 3]{HLS:HEquivariant} we obtained a
$G$-equivariant homotopy coherent $\ECat G$-diagram $T'$ in chain
complexes by letting $T'(g) = \CF(L_0,L_1; F(g))=\CF(L_0, L_1; gJ)$ and 
$T'(g;g_n,\dots,g_1) \co T'(g) \to T'(gg_1\cdots g_n)$ be
\begin{equation}\label{eq:T-map}
T'(g;g_1,\dots,g_n)(x) =  
\sum_{y \in L_0 \cap L_1}
\sum_{\substack{\phi \in \pi_2(x,y) \\ \mu(\phi)= 1-n}} 
\#\cM(\phi;F(g;g_n,\dots,g_1))y
\end{equation}
on generators $x\in L_0 \cap L_1$ and extending linearly. (The
formulas here are slightly different from our previous paper, as we
are using a left action of $G$ here and used a right action in the
prequel.)  
The moduli space in Formula~\eqref{eq:T-map} is the union of the moduli spaces of
$F(g;g_n,\dots,g_1)(v)$-holomorphic curves over points $v$ in the
$(n-1)$-cube.
Equivalently, we can
define a homotopy coherent $\BCat G$-diagram in chain complexes by
setting $T(o) = \CF(L_0,L_1; J)$ and
\[
T(g_n,\dots, g_1)(x)=  
\sum_{gy \in L_0 \cap L_1} 
\sum_{\substack{\phi \in \pi_2(x,gy) \\ \mu(\phi)=1-n}} 
\#\cM(\phi;F(g_n,\dots,g_1))y
\]
on generators $x \in L_0 \cap L_1$ and extending linearly. Here $g =
g_n \cdots g_1$.

Let us compare this to the data we get by applying the Floer functor
to $\BGSection$. First, notice that since the $n$-simplices in
$\Nerve B G$ correspond
to sequences of $n$-composable morphisms $(g_n,\dots, g_1)$ in
$\BCat G$, the data of the functor
$\FloerFunc \circ \BGSection \co \Nerve BG \to
\Complexes$ is exactly the data of a homotopy coherent $\BCat G$-diagram
in chain complexes. (Note that the $i\th$ degeneracy map corresponds to dropping $g_{i+1}$.) Condition~\ref{item:hc3}
corresponds to Equation~\eqref{eq:chain-cx-cat} in the definition of
the infinity category of chain complexes.

\begin{lemma}\label{lemma-coherent-diagrams-same} 
  The homotopy coherent $\BCat G$-diagram in chain complexes induced
  by $\FloerFunc \circ \BGSection \co \Nerve BG \to \Complexes$
  is exactly the homotopy coherent $\BCat G$-diagram $T$ defined using
  the maps determined by the diagram of multi-story almost complex
  structures $F$.
\end{lemma}

\begin{proof}
  This is immediate from the definitions.
\end{proof}

Before the proof of Theorem \ref{thm:finite-same},
we recall one final definition. Given a homotopy coherent $\Cat$-diagram of
chain complexes, we recall that the homotopy colimit is the
following.

\begin{definition}\label{def:old-hocolim-cxs}
  Given a homotopy coherent $\Cat$-diagram of chain complexes,
  the \emph{homotopy colimit} of $T$ is defined by
  \begin{equation}\label{eq:old-chain-hocolim-def}
  \hocolim T = \bigoplus_{n\geq 0}
  \bigoplus_{x_0\stackrel{f_1}{\longrightarrow}\cdots\stackrel{f_n}{\longrightarrow}x_n}
  I_*^{\otimes n}\otimes T(x_0)/\sim,
  \end{equation}
  where the coproduct is over $n$-tuples of composable morphisms in
  $\Cat$ and the case $n=0$ corresponds to the objects
  $x_0\in\ob(\Cat)$.  The equivalence relation $\sim$ is given by
  \begin{multline*}
  (f_n,\dots,f_1;t_n\otimes\dots\otimes t_1;x)\\
  \sim
  \begin{cases}
    (f_n,\dots,f_2;t_n\otimes\dots\otimes t_2 \otimes \pi(t_1);x) &  f_1=\Id\\
    (f_n,\dots,f_{i+1},f_{i-1},\dots,f_1;t_n\otimes\dots\otimes m(t_{i},t_{i-1})\otimes\dots \otimes t_1;x) & f_i=\Id,\ i>1\\
    (f_n,\dots,f_{i+1}\circ f_i,\dots,f_1;t_n\otimes\dots\otimes t_{i+1}\otimes t_{i-1} \otimes\dots \otimes t_1;x) & t_i=\{1\},\ i<n\\
    (f_{n-1},\dots,f_1;t_{n-1}\otimes\dots\otimes t_{1};x) & t_n=\{1\}\\
    (f_n,\dots,f_{i+1};t_{n}\otimes\dots\otimes t_{i+1};T(f_i,\dots,f_1)(t_{i-1}\otimes\dots\otimes t_{1} \otimes x)) & t_i=\{0\}.
  \end{cases}
  \end{multline*}
  The differential is induced by the tensor product differential in
  Formula~\eqref{eq:old-chain-hocolim-def}.
\end{definition}

\begin{proof}[Proof of Theorem~\ref{thm:finite-same}] 
  The homotopy colimit of $\FloerFunc \circ \BGSection$, in the sense
  of Definition \ref{def:hocolim}, is the complex
  \[
    C_*=\hocolim (\FloerFunc \circ \BGSection) = \bigoplus_{n\geq 0}
    \bigoplus_{\sigma \in \Nerve BG}
    I_*^{\otimes n}\otimes \maxFloerFunc(\BGSection(\sigma))(x_0)/\sim
  \]
  where $\sim$ is as in Definition~\ref{def:hocolim}. The cohomology
  of this complex is the equivariant cohomology of $(M,L_0,L_1)$ in
  the sense of this paper. Notice that this complex is also the
  homotopy colimit of the homotopy coherent $\BCat G$-diagram
  in the sense of Definition~\ref{def:old-hocolim-cxs},
  since the equivalence relations in the two
  definitions are identical in the case that $G$ is discrete. However,
  by Lemma~\ref{lemma-coherent-diagrams-same} this homotopy coherent
  diagram is precisely the diagram $T$, so the complex $C_*$ is also
  equal to
  \[
    \hocolim (T) = \bigoplus_{n\geq 0}
    \bigoplus_{x_0\stackrel{f_1}{\longrightarrow}\cdots\stackrel{f_n}{\longrightarrow}x_n}
    I_*^{\otimes n}\otimes T(x_0)/\sim
  \]
  where $\sim$ is as in Definition~\ref{def:old-hocolim-cxs}. 
  
The $G$-equivariant cohomology of $(M,L_0,L_1)$ in the sense of our
previous paper is the homology of the complex
$\Hom_{\Ring[G]}(\hocolim T',\Ring)$. This is the same as the
cohomology of $\hocolim T$, the homotopy colimit of the $\BCat G$-diagram $T$. Since the cohomology of this
  homotopy colimit is also the $G$-equivariant cohomology of $(M,L_0,L_1)$
  in the sense of this present paper, the two definitions agree.
\end{proof}

\begin{example}
  Continuing with Examples~\ref{eg:coh-or} and~\ref{eg:or-2},
  depending on the choice of $\ZZ/2$-orientation system, there are
  four possible models for $\hocolim(T)$. Since we can replace $x$ by
  $-x$, it suffices to consider the cases that the element of
  $\cM(B;J)$ connecting $x$ to $y$ is oriented positively. The four
  cases then correspond to the orientations over $B_e(x,x,F-B)$ and
  $B_\tau(x,x,0)$. Equivalently, these correspond to whether the
  elements $u_F\in \cM(F;J)$ from $x$ to $y$ and $u_{0,x}\in \cM(0;F(\tau))$ from $x$
  to $x$ are oriented positively or negatively. There is also an
  element $u_{0,y}\in\cM(0;F(\tau))$ from $y$ to $y$, whose
  orientation is easy to determine from Theorem~\ref{thm:gluing}. The four cases are:
  \begin{itemize}
  \item Both $u_F$ and $u_{0,x}$ are oriented positively. Then
    $u_{0,y}$ is oriented positively and $\hocolim(T)$ is the total
    complex of
    \[
      \xymatrix{
        \ZZ\ar[d]_{2} & \ZZ\ar[l]_0\ar[d]_{-2} & \ZZ\ar[d]_{2}\ar[l]_2 & \ZZ\ar[d]_{-2}\ar[l]_0 & \ZZ\ar[d]_{2}\ar[l]_2 & \cdots\ar[l]\\
        \ZZ & \ZZ\ar[l]^0 & \ZZ\ar[l]^2 & \ZZ\ar[l]^0 & \ZZ\ar[l]^2 & \cdots.\ar[l]\\
      }
    \]
  \item $u_F$ is oriented positively and $u_{0,x}$ is oriented
    negatively. Then $u_{0,y}$ is oriented negatively and
    $\hocolim(T)$ is the total complex of
    \[
      \xymatrix{
        \ZZ\ar[d]_{2} & \ZZ\ar[l]_2\ar[d]_{-2} & \ZZ\ar[d]_{2}\ar[l]_0 & \ZZ\ar[d]_{-2}\ar[l]_2 & \ZZ\ar[d]_{2}\ar[l]_0 & \cdots\ar[l]\\
        \ZZ & \ZZ\ar[l]^2 & \ZZ\ar[l]^0 & \ZZ\ar[l]^2 & \ZZ\ar[l]^0 & \cdots.\ar[l]\\
      }
    \]
  \item $u_F$ is oriented negatively and $u_{0,x}$ is oriented positively. Then
    $u_{0,y}$ is oriented negatively and $\hocolim(T)$ is the total
    complex of
    \[
      \xymatrix{
        \ZZ\ar[d]_{0} & \ZZ\ar[l]_0\ar[d]_{0} & \ZZ\ar[d]_{0}\ar[l]_2 & \ZZ\ar[d]_{0}\ar[l]_0 & \ZZ\ar[d]_{0}\ar[l]_2 & \cdots\ar[l]\\
        \ZZ & \ZZ\ar[l]^2 & \ZZ\ar[l]^0 & \ZZ\ar[l]^2 & \ZZ\ar[l]^0 & \cdots.\ar[l]\\
      }
    \]    
  \item Both $u_F$ and $u_{0,x}$ are oriented negatively. Then
    $u_{0,y}$ is oriented positively and $\hocolim(T)$ is the total
    complex of
    \[
      \xymatrix{
        \ZZ\ar[d]_{0} & \ZZ\ar[l]_2\ar[d]_{0} & \ZZ\ar[d]_{0}\ar[l]_0 & \ZZ\ar[d]_{0}\ar[l]_2 & \ZZ\ar[d]_{0}\ar[l]_0 & \cdots\ar[l]\\
        \ZZ & \ZZ\ar[l]^0 & \ZZ\ar[l]^2 & \ZZ\ar[l]^0 & \ZZ\ar[l]^2 & \cdots.\ar[l]\\
      }
    \]    
  \end{itemize}
  In particular, the last case gives the usual Borel equivariant
  (co)homology of $S^1$ with respect to the $\ZZ/2$-action by
  reflection across a line.
\end{example}

\subsection{The case of Morse theory}\label{sec:comp-morse}
The goal of this section is to prove:
\begin{proposition}\label{prop:Morse-is-Borel}
  For $G$ a Lie group acting on a compact, smooth manifold $M$ the
  $G$-equivariant cohomology $\eH[G](M)$ as defined in
  Equation~\eqref{eq:equi-morse-coho} agrees with the Borel
  equivariant cohomology $H^*(M\times_G EG)$.
\end{proposition}
As in Section~\ref{sec:Morse-const}, we either assume that $M$ is
orientable and the action of $G$ on $M$ is orientation-preserving or
we work with coefficients in $\Field_2$ throughout.

The proof of Proposition~\ref{prop:Morse-is-Borel} is in two steps.
First we replace the Morse complex in the definition of $\eH[G](M)$
with the smooth cubical singular chain complex $C_*^{\cube}(M)$ (the analogue of the
usual smooth singular chain complex, but using smooth maps of cubes to $M$ instead
of maps of simplices to $M$), obtaining a map of semi-simplicial sets
$\cubeF\co \Nerve BG\to\Complexes$. We then lift $\cubeF$ to a functor
to $\spaceF$ of simplicial sets and verify that the homotopy colimit of $\spaceF$ satisfies
the defining properties of the Borel construction.

Turning to the details,
we start by defining the map $\cubeF$. Define $\cubeF$ of the
0-simplex in $\Nerve BG$ to be $C_*^{\cube}(M)$. Recall that the smooth cubical chain group $C_k^{\cube}(M)$
is the free abelian group generated by the smooth maps $[0,1]^k\to M$ modulo
degenerate cubes, i.e., cubes which factor through one of the $k$
canonical projections $[0,1]^k\to [0,1]^{k-1}$ (cf.~\cite{Massey80:cubical-singular}).

We now define the map on the higher-dimensional simplices. Given $\sigma$ an $n$-dimensional simplex in $\Nerve BG$, by Convention \ref{convention:highest-maps-in-complexes}, it suffices to define $\maxmap{F}_{\cube}(\sigma)$. Informally,
via the $G$-action on $M$, $\cubemap{\sigma}$ gives a family of maps $M\to M$
parameterized by $[0,1]^{n-1}$, and we define
$\maxmap{F}_{\cube}(\sigma)$ to be the higher chain homotopy associated
to (the inverse of) this family of maps, so
$\maxmap{F}_{\cube}(\sigma)\co C_k^{\cube}(M)\to C_{k+n-1}^{\cube}(M)$. More formally, we have a map $\cubemap{\sigma} \co [0,1]^{n-1} \rightarrow G$. Given a $k$-cube
$(\alpha\co [0,1]^k\to M)\in C_k^{\cube}(M)$ define
$\maxmap{F}_{\cube}(\sigma)(\alpha)\in C_{k+n-1}^{\cube}(M)$ to
be
\[
  \maxmap{F}_{\cube}(\sigma)(\alpha)(x,y)=\cubemap{\sigma}(x)^{-1}\cdot\alpha(y)\eqqcolon (\cubemap{\sigma}_!^{-1}\alpha)(x,y)
\]
where $(x,y)\in [0,1]^{n-1}\times[0,1]^k=[0,1]^{n-1+k}$. (The inverse
is because of our convention for composition; compare, for instance,
Example~\ref{eg:ngj-3cell}.) Note that for degenerate simplices
$\sigma$, $\maxmap{F}_{\cube}(\sigma)$ does not necessarily vanish, so
$F_{\cube}$ does not respect the degeneracy maps.
Extend $\maxmap{F}_{\cube}(\cubemap{\sigma})$ linearly to all of $C_k^{\cube}(M)$.

\begin{lemma}
  The maps $\maxmap{F}_{\cube}$ induce a map of semi-simplicial sets
  $\cubeF\co\Nerve BG\to \ChainComplexes$.
\end{lemma}
\begin{proof}
  The proof is straightforward and is left to the reader.
\end{proof}

Now, fix a Morse-Smale pair, and let $\MorseFunc$ be as in
Section~\ref{sec:Morse-const}.  Our next goal is to define a summed
natural transformation from $\MorseFunc\circ\BGSection$ to $\cubeF$
(Definition~\ref{def:summed-nat-trans}),
i.e., for each $\sigma\in \Nerve BG$ a map
$\eta(\sigma)\co C_*^{\Morse}(M,\metric,f)\to C_*^{\cube}(M)$ satisfying
Equation~\eqref{eq:summed-hocolim-map}, which we repeat here for
convenience:
\begin{equation}\label{eq:nat-transform-main}
\begin{split}
      (-1)^n\bdy\circ \eta(\sigma)-\eta(\sigma)\circ\bdy
      =&\sum_{1\leq\ell\leq n-1}
      (-1)^{\ell-1}\eta(\sigma_{0,\dots,\hat{\ell},\dots,n})-\sum_{1\leq \ell\leq n}
      \eta(\sigma_{\ell,\dots,n})\circ \maxMorseFunc\BGSection(\sigma_{0,\dots,\ell})\\
      &+\sum_{0\leq\ell\leq
          n-1}(-1)^{\ell}\maxmap{F}_\cube(\sigma_{\ell,\dots,n})\circ\eta(\sigma_{0,\dots,\ell}).
    \end{split}
  \end{equation}

We explain two low-dimensional cases of $\eta$ before giving the general case.

The first special case is when $\sigma$ is the unique $0$-simplex, which
we denote $o$, of
$\Nerve BG$. We need to define a chain map
$\eta(o)\from C_*^{\Morse}(M,\metric,f)\to C_*^{\cube}(M)$.
Let $\pModuli(x,y)$ denote the moduli space of parameterized gradient
flows from $x$ to $y$, so $\Moduli(x,y)=\pModuli(x,y)/\RR$, and
$\pModuli(x,y)$ has dimension $\ind(x)-\ind(y)$. There is
an evaluation map $\ev\co \pModuli(x,y)\to M$,
$\ev(\gamma)=\gamma(0)$. The space $\pModuli(x,y)$ has a
compactification $\opModuli(x,y)$ in terms of broken gradient
flows in which one of the flows is parameterized; as sets,
\begin{multline}\label{eq:bdy-pmoduli-top}
  \bdy\opModuli(x,y)=\opModuli(x,y)\setminus\pModuli(x,y)=\coprod_{z\neq
    y}\pModuli(x,z)\times\oModuli(z,y)\ \amalg\ \coprod_{z\neq x}\oModuli(x,z)\times\pModuli(z,y)\\
  \amalg\!\!\coprod_{z,w\not\in\{x,y\}}\oModuli(x,z)\times\pModuli(z,w)\times\oModuli(w,y).
\end{multline}
Further, $\opModuli(x,z)\times\oModuli(z,y)$
and $\oModuli(x,z)\times\opModuli(z,y)$ are subspaces of
$\opModuli(x,y)$, and every point in $\bdy \opModuli(x,y)$ lies in at
least one such subspace.

The orientations of the ascending disks
and descending disks as in Section~\ref{sec:Morse-const} induce
orientations of the spaces $\Moduli(x,y)$ and $\pModuli(x,y)$. Choose the orientations of the
$\Moduli(x,y)$ so that the inclusion
$\oModuli(z,y)\times\oModuli(x,z)\into\bdy\oModuli(x,y)$ is
orientation preserving. For $x\neq y$, orient the space $\pModuli(x,y)$ as
$(-1)^{\ind(y)}\Moduli(x,y)\times\RR$; if $x=y$, orient the point
$\pModuli(x,x)$ as $(-1)^{\ind(x)}$.

Choose a cubical fundamental chain (which we abbreviate as ``cubulation'') 
$[\oModuli(x,y)]$ of $\oModuli(x,y)$ so that the inclusion
$\oModuli(z,y)\times\oModuli(x,z)\hookrightarrow\bdy\oModuli(x,y)$ is
a cubical map. (These choices can be made inductively on $\ind(x)$,
say.) Then choose a cubulation $[\opModuli(x,y)]$ of each
$\opModuli(x,y)$ so that the inclusions
$\oModuli(z,y)\times\opModuli(x,z)\hookrightarrow \bdy\opModuli(x,y)$
and $\opModuli(z,y)\times\oModuli(x,z)\hookrightarrow
\bdy\opModuli(x,y)$ are cubical maps. Define
\[
  \eta(o)(x)=\sum_{\ind(y)=0}\ev_\#[\opModuli(x,y)]
\]
for each $x\in\Crit(f)$.

\begin{lemma}\label{lem:bdy-pmoduli}
  As cubical chains,
  \[
    \bdy\ev_{\#}\bigl[\opModuli(x,y)\bigr]=-\!\!\!\!\!\!\sum_{\ind(z)-\ind(y)=1}\!\!\!\!\!\!\bigl(\#\cM(z,y)\bigr)\ev_{\#}\bigl[\opModuli(x,z)\bigr]+\!\!\!\!\!\!\sum_{\ind(x)-\ind(z)=1}\!\!\!\!\!\!\bigl(\#\cM(x,z)\bigr)\ev_{\#}\bigl[\opModuli(z,y)\bigr].
  \]
\end{lemma}
\begin{proof}[Proof sketch]
  First observe that taking orientations into account
  \[
  \bdy\opModuli(x,y)=\bigcup_z(-1)^{\ind(x)-\ind(z)-1}\opModuli(z,y)\times\oModuli(x,z)\cup(-1)^{\ind(z)-\ind(y)}\oModuli(z,y)\times\opModuli(x,z).
  \]

  Next observe that on the stratum $\opModuli(x,z)\times\oModuli(z,y)$ of
  $\opModuli(x,y)$, the evaluation map $\ev$ factors through the
  projection $\opModuli(x,z)\times\oModuli(z,y)\to \opModuli(x,z)$.
  So, if $\ind(z)-\ind(y)>1$ then every cube in
  $\ev_{\#}[\oModuli(z,y)\times\opModuli(x,z)]$ is
  degenerate. Similarly, if $\ind(x)-\ind(z)>1$ then every cube in
  $\ev_{\#}[\opModuli(z,y)\times\oModuli(x,z)]$ is degenerate. The
  remaining cubes in $\bdy\opModuli(x,y)$ lie in 
  \[
    \Bigl[\hspace{1em}-\hspace{-1.75em}\bigcup_{\ind(z)-\ind(y)=1}\hspace{-1.75em}\Moduli(z,y)\times\opModuli(x,z)\Bigr]\coprod\Bigl[\bigcup_{\ind(x)-\ind(z)=1}\hspace{-1.75em}\opModuli(z,y)\times\Moduli(x,z)\Bigr],
  \]
  and so
  \begin{align*}
  \bdy\ev_{\#}\bigl[\opModuli(x,y)\bigr]&=-\!\!\!\!\!\!\sum_{\ind(z)-\ind(y)=1}\!\!\!\!\!\!\ev_{\#}\bigl[\Moduli(z,y)\times\opModuli(x,z)\bigr]+\!\!\!\!\!\!\sum_{\ind(x)-\ind(z)=1}\!\!\!\!\!\!\ev_{\#}\bigl[\opModuli(z,y)\times\Moduli(x,z)\bigr]
\\
&=-\!\!\!\!\!\!\sum_{\ind(z)-\ind(y)=1}\!\!\!\!\!\!\bigl(\#\cM(z,y)\bigr)\ev_{\#}\bigl[\opModuli(x,z)\bigr]+\!\!\!\!\!\!\sum_{\ind(x)-\ind(z)=1}\!\!\!\!\!\!\bigl(\#\cM(x,z)\bigr)\ev_{\#}\bigl[\opModuli(z,y)\bigr],
  \end{align*}
as desired.
\end{proof}

\begin{lemma}\label{lem:eta-chain-map}
  If $o$ is the $0$-cell in $\Nerve BG$, then the map $\eta(o)$, as
  defined above, is a chain map.
\end{lemma}
\begin{proof}
  Given a critical point $x$ of index $n$,
  \begin{align*}
    \bdy \eta(o)(x)&=\sum_{\ind(y)=0}\bdy \ev_\#[\opModuli(x,y)]\\
    &=-\sum_{\substack{\ind(y)=0\\\ind(z)=1}}(\#\cM(z,y))\ev_\#[\opModuli(x,z)]
    +\sum_{\substack{\ind(y)=0\\\ind(z)=n-1}}(\#\cM(x,z))\ev_\#[\opModuli(z,y)]\\
    &=\eta(o)(\bdy(x)),
  \end{align*}
  where the second equality uses Lemma~\ref{lem:bdy-pmoduli} and the
  third uses the fact that for any critical point $z$ with index $1$,
  $\sum_{\ind(y)=0}\#\cM(z,y)=0$.
\end{proof}

The second warm-up case is when $\sigma$ is a non-degenerate $1$-cell
$g$ of $\Nerve BG$ (where $g$ is a non-identity element of $G$). Then
$\eta(g)$ should be a homotopy between
$\eta(o)\circ \maxMorseFunc\BGSection(g)$ and
$\maxmap{F}_{\cube}(g)\circ \eta(o)$.
The semicylindrical Morse-Smale pair $\BGSection(g)$ produces
parameterized moduli spaces $\pModuli(x,y;\BGSection(g))$ of pairs
$(\gamma,t)$ where $\gamma$ is a gradient flow from $x$ to $gy$ with
respect to the semicylindrical pair $\BGSection(g)$ and $t\in
\RR$. There is an evaluation map $\ev\co
\pModuli(x,y;\BGSection(g))\to M$, $\ev(\gamma,t)=\gamma(t)$.  The
boundary of $\pModuli(x,y;\BGSection(g))$ is a non-disjoint union
\begin{align*}
  \bigcup_{z\neq y}\opModuli(x,z;\BGSection(g))\times\oModuli(z,y)&\cup
  \bigcup_{z}\opModuli(x,z)\times\oModuli(z,y;\BGSection(g))\\
  &\cup
  \bigcup_{z\neq x}\oModuli(x,z)\times\opModuli(z,y;\BGSection(g))\cup
  \bigcup_{z}\oModuli(x,z;\BGSection(g))\times\opModuli(z,y).
\end{align*}
Arguably, we should write the first term as $\bigcup_{z\neq y}\opModuli(x,z;\BGSection(g))\times g\cdot \oModuli(z,y)$ and the last term as $\bigcup_{z}\oModuli(x,z;\BGSection(g))\times g\cdot \opModuli(z,y)$. In particular, in the last term, the evaluation map sends $\gamma\in \opModuli(z,y)$ to $g\cdot\ev(\gamma)$.

Orient the moduli space $\Moduli(x,y;\BGSection(g))$ so
that the inclusions
$-\Moduli(z,y;\BGSection(g))\times\Moduli(x,z)\into\bdy\Moduli(x,y;\BGSection(g))$
and
$\Moduli(z,y)\times\Moduli(x,z;\BGSection(g))\into\bdy\Moduli(x,y;\BGSection(g))$
are orientation preserving (cf.~Equation~\eqref{eq:gluing}).  Orient
$\pModuli(x,y;\BGSection(g))$ as
$(-1)^{\ind(y)}\Moduli(x,y;\BGSection(g))\times\RR$.  Choose
cubulations of the $\opModuli(x,z;\BGSection(g))$ so that the above
inclusions are cubical and define
\[
  \eta(g)(x)=\sum_{\ind(y)=0}g_{\#}^{-1}\ev_\#[\opModuli(x,y;\BGSection(g))].
\]
(Here $g_\#^{-1}$ is the chain map induced by the action of $g^{-1}\from M\to M$; that is, for any cube $\sigma\from[0,1]^k\to\opModuli(x,y;\BGSection(g))$ in the cubulation $[\opModuli(x,y;\BGSection(g))]$, $g_\#^{-1}\ev_\#\sigma$ is the cube $g^{-1}\circ\ev\circ\sigma$.)
The space $\pModuli(x,y;\BGSection(g))$ has dimension
$\ind(x)-\ind(y)+1$, so $\eta(g)(x)$ is a cubical chain of
dimension $\ind(x)+1$.
\begin{lemma}\label{lem:eta-square}
  The map $\eta(g)$ is a homotopy between
  $\eta(o)\circ \maxMorseFunc\BGSection(g)$ and
  $\maxmap{F}_{\cube}(g)\circ \eta(o)$.
\end{lemma}
\begin{proof}
  As in the proof of Lemma~\ref{lem:bdy-pmoduli}, the boundary of
  $\ev_{\#}[\opModuli(x,y;\BGSection(g))]$ is
  \begin{equation}
    \begin{split}
      &-\!\!\!\!\!\!\sum_{\ind(z)-\ind(y)=1}\bigl(\#\Moduli(z,y)\bigr)\ev_\#\bigl[\opModuli(x,z;\BGSection(g))\bigr]
      -\!\!\!\!\!\!
      \sum_{\ind(z)=\ind(y)}\bigl(\#\Moduli(z,y;\BGSection(g))\bigr)\ev_\#\bigl[\opModuli(x,z)\bigr]\\
      &-\!\!\!\!\!\!\sum_{\ind(x)-\ind(z)=1}\bigl(\#\Moduli(x,z)\bigr)\ev_\#\bigl[\opModuli(z,y;\BGSection(g))\bigr]
      +\!\!\!\!\!\!
      \sum_{\ind(z)=\ind(x)}\bigl(\#\Moduli(x,z;\BGSection(g))\bigr)g_\#\ev_\#\bigl[\opModuli(z,y)\bigr].
    \end{split}
  \end{equation}

  Note that for $z\in\Crit(f)$ with $\ind(z)=0$,
  \[
    \sum_{\ind(y)=0}\#\Moduli(z,y;\BGSection(g))=1.
  \]
  Thus, for $x\in\Crit(f)$,
  \begin{align*}
    \maxmap{F}_\cube(g)&(\eta(o)(x))-\eta(o)(\maxmap{F}_\Morse\BGSection(g)(x))\\
    &= g_\#^{-1}\!\!\!\sum_{\ind(y)=0}\ev_{\#}[\opModuli(x,y)]-\!\!\!\sum_{\substack{\ind(y)=0\\\ind(z)=\ind(x)}}\!\!\bigl(\#\Moduli(x,z;\BGSection(g))\bigr)\ev_{\#}[\opModuli(z,y)]\\
    &= g_\#^{-1}\!\!\!\sum_{\substack{\ind(y)=0\\\ind(z)=0}}\bigl(\#\Moduli(z,y;\BGSection(g))\bigr)\ev_{\#}[\opModuli(x,z)]
    -\!\!\!\sum_{\substack{\ind(y)=0\\\ind(z)=\ind(x)}}\!\!\bigl(\#\Moduli(x,z;\BGSection(g))\bigr)\ev_{\#}[\opModuli(z,y)]\\
    &=-g_{\#}^{-1}\!\!\!\sum_{\substack{\ind(y)=0\\\ind(z)=1}}\bigl(\#\Moduli(z,y)\bigr)\ev_{\#}\bigl[\opModuli(x,z;\BGSection(g))\bigr]\\
      &\qquad\qquad-g_{\#}^{-1}\!\!\!\!\!\!\!\!\!\sum_{\substack{\ind(y)=0\\\ind(x)-\ind(z)=1}}\!\!\!\!\bigl(\#\Moduli(x,z)\bigr)\ev_{\#}\bigl[\opModuli(z,y;\BGSection(g))\bigr]
    -\bdy(\eta(g)(x))\\
    &=-g_{\#}^{-1}\sum_{\substack{\ind(y)=0\\\ind(z)=1}}\!\!\!\!\bigl(\#\Moduli(z,y)\bigr)\ev_{\#}\bigl[\opModuli(x,z;\BGSection(g))\bigr]-\eta(g)(\bdy x)-\bdy(\eta(g)(x))\\
    &=-\eta(g)(\bdy(x))-\bdy(\eta(g)(x)),
  \end{align*}
  as desired.
\end{proof}

Finally, we explain the general case. For each $n$-simplex $\sigma$ in
$\Nerve BG$ we have an $n$-simplex $\BGSection(\sigma)$ in
$\Nerve \MPcat$, and in particular an $(n-1)$-cube
$\maxmap{\BGSection(\sigma)}$ in the space of Morse-Smale pairs. Let
\[
  \opModuli(x,y;\maxmap{\BGSection(\sigma)})
  =\bigcup_{v\in[0,1]^{n-1}}\opModuli\bigl(x,y;\maxmap{\BGSection(\sigma)}(v)\bigr)
\]
denote the compactification of the corresponding parameterized moduli
space. As in the definition of $\cubeF$, the $n$-simplex $\sigma$ also
specifies an $(n-1)$-cube $\cubemap{\sigma}$ in $G$. We define
\[
  \eta(\sigma)(x)=
    \sum_{\ind(y)=0}\cubemap{\sigma}^{-1}_\#\ev_\#\bigl[\opModuli(x,y;\maxmap{\BGSection(\sigma)})\bigr]\in
    C_{n+\ind(x)}^{\cube}(M).
\]
Here, $\cubemap{\sigma}_\#^{-1}\ev_\#$ means the map of chains induced by the composition 
\[
  \opModuli\bigl(x,y;\maxmap{\BGSection(\sigma)}(v)\bigr)\stackrel{\ev}{\longrightarrow} M\stackrel{\cubemap{\sigma}(v)^{-1}\cdot }{\longrightarrow}M
\]
for each $v\in[0,1]^{n-1}$. 
The precise chain depends on a choice of cubical fundamental chain
(informally, cubulation) of
$\opModuli\bigl(x,y;\maxmap{\BGSection(\sigma)}\bigr)$. The
boundary of $\opModuli\bigl(x,y;\maxmap{\BGSection(\sigma)}\bigr)$
is
\begin{align*}
  \opModuli(x,y;\maxmap{\BGSection(\sigma)}|_{\bdy[0,1]^{n-1}})&\cup
  \bigcup_z\bigl(\opModuli(x,z;\maxmap{\BGSection(\sigma)})\times\oModuli(z,y)\bigr)\cup
  \bigcup_z\bigl(\oModuli(x,z)\times\opModuli(z,y;\maxmap{\BGSection(\sigma)})\bigr)\\
  &\cup
  \bigcup_z\bigl(\opModuli(x,z)\times\oModuli(z,y;\maxmap{\BGSection(\sigma)})\bigr)\cup
  \bigcup_z\bigl(\oModuli(x,z;\maxmap{\BGSection(\sigma)})\times\opModuli(z,y)\bigr)
  .
\end{align*}
We  choose the cubulation inductively so that each of these boundary
inclusions is a cubical map.

\begin{proposition}
  These definitions make $\eta$ into a summed natural transformation.
\end{proposition}
\begin{proof}
  We need to check that these maps $\eta$ satisfy
  Equation~\eqref{eq:nat-transform-main}; this is similar to the proof
  of Lemma~\ref{lem:eta-square}. The boundary of $\eta(\sigma)(x)$
  (which is the first term in Equation~\eqref{eq:nat-transform-main})
  has the following non-degenerate pieces:
  \begin{enumerate}
  \item The chain
    \[
      \sum_{\substack{\ind(y)=0\\\ind(z)=\ind(x)-1}}\bigl(\#\Moduli(x,z)\bigr)\cubemap{\sigma}^{-1}_\#\ev_{\#}\bigl[\opModuli(z,y;\cubemap{\BGSection(\sigma)})\bigr],
    \]
    i.e., broken flows in which the first flow (from $x$) is
    cylindrical and unparameterized. These correspond to the second
    term in Equation~\eqref{eq:nat-transform-main}.
  \item The chain
    \[
      \sum_{\substack{\ind(y)=0\\\ind(z)=0}}\bigl(\#\Moduli(z,y;\maxmap{\BGSection(\sigma)})\bigr)\sigma_{!}^{-1}\ev_{\#}\bigl[\opModuli(x,z)\bigr]=
      \sum_{\ind(y)=0}\sigma_{!}^{-1}\ev_{\#}\bigl[\opModuli(x,y)\bigr],
    \]
    i.e., broken flows in which the first flow is
    cylindrical and parameterized. Here, if
    $\ev_{\#}\bigl[\opModuli(x,z)\bigr]=\sum n_i\alpha_i\in
    C_k^{\cube}(M)$, so
    $\alpha_i\co [0,1]^k\to M$ and $k=\ind(x)$, then
    \[
      \sigma_{!}^{-1}\ev_{\#}\bigl[\opModuli(x,z)\bigr]=\sum
      n_i\bigl[(s,t)\mapsto \cubemap{\sigma}(s)^{-1}\cdot \alpha_i(t)\bigr]\in C_{k+n-1}^{\cube}(M).
    \]
    (This operation $\sigma_!$ also appears in the definition of $\maxmap{F}_{\cube}$.) These
    correspond to the $\ell=0$ case of the fifth term of
    Equation~\eqref{eq:nat-transform-main}. 
  \item The chain
    \[
      \sum_{\substack{\ind(y)=0\\\ind(z)=\ind(x)+n-1}}\bigl(\#\Moduli(x,z;\maxmap{\BGSection(\sigma)})\bigr)\ev_{\#}\bigl[\opModuli(z,y)\bigr],
    \]
    i.e., broken flows in which the first flow is semicylindrical and
    unparameterized. These correspond to the $\ell=n$ case of the
    fourth term of Equation~\eqref{eq:nat-transform-main}.
  \item The chain
    \[
      \sum_{\substack{\ind(y)=0\\\ind(z)=1}}\bigl(\#\Moduli(z,y)\bigr)\cubemap{\sigma}^{-1}_\#\ev_{\#}\bigl[\opModuli(x,z;\cubemap{\BGSection(\sigma)})\bigr],
    \]
    i.e., broken flows in which the first flow is semicylindrical and
    parameterized. These flows contribute $0$, as in the proof of
    Lemma~\ref{lem:eta-chain-map}, because for any critical point $z$
    with index $1$, $\sum_{\ind(y)=0}\#\Moduli(z,y)=0$.
  \item The boundary of $[0,1]^n$ corresponding to where the $i\th$
    coordinate is $0$. This corresponds to the third term of
    Equation~\eqref{eq:nat-transform-main}.
  \item The boundary of $[0,1]^n$ corresponding to where the $i\th$
    coordinate is $1$. Over this boundary, $\maxmap{\BGSection(\sigma)}$ is a
    $2$-story Morse-Smale pair. This part of the boundary corresponds
    to the rest of the fourth or fifth terms in
    Equation~\eqref{eq:nat-transform-main}, depending on whether the
    first story is unparameterized or parameterized, respectively.
  \end{enumerate}
  Taking into account signs, which are left to the reader, this
  proves the result.
\end{proof}

Recall from Proposition~\ref{prop:summed-nat-chain-map} that a summed natural transformation
$\eta$ from $\MorseFunc\circ \BGSection$ to $\cubeF$ induces a chain map
$f_\eta\co \fathocolim(\MorseFunc\circ\BGSection)\to \fathocolim \cubeF$. 
Also, recall that by Proposition~\ref{prop:fat-hocolim-is-hocolim},
$\hocolim(\MorseFunc\circ\BGSection)\simeq \fathocolim(\MorseFunc\circ\BGSection)$, as modules over $H^*(BG)$. 
\begin{lemma}\label{lem:Morse-is-sing}
  The map
  $(f_\eta)^*\co H^*(\fathocolim \cubeF)\to H^*(\fathocolim(\MorseFunc\circ\BGSection))$ is
  an isomorphism and respects the $H^*(BG)$-module structure.
\end{lemma}
\begin{proof}
  That the map $(f_{\eta})^*$ respects the module structure
  follows from Proposition~\ref{prop:summed-nat-chain-map}. To see that $(f_\eta)^*$
  is an isomorphism, note that $f_{\eta}$ respects the filtration
  $\Filt$ by the dimension of simplices (see Section~\ref{sec:homotopy-colimits}, as well as the proof of Proposition~\ref{prop:fat-hocolim-is-hocolim}). On the $E^1$-page
  of the associated spectral sequence, the induced map
  \[
    (f_{\eta})_*\co \bigoplus_{\sigma\in \Nerve BG} H_*^{\Morse}(M)\to 
    \bigoplus_{\sigma\in \Nerve BG} H_*^{\cube}(M)
  \]
  is a sum of copies of the well-known isomorphism between Morse
  homology and cubical singular homology, induced by sending a
  critical point to a cubulation of its descending manifold. (See also
  the proof of Proposition~\ref{prop:sseq}.) Consequently, the map
  $f_{\eta}$ is itself a quasi-isomorphism, as desired.
\end{proof}

It remains to compute $H^*(\fathocolim \cubeF)$.

\begin{proposition}\label{prop:sing-is-Borel}
  There is an isomorphism
  $H^*(\fathocolim \cubeF)\cong H^*(M\times_G EG)$ of modules over
  $H^*(BG)$.
\end{proposition}

To prove Proposition~\ref{prop:sing-is-Borel} we observe that $\cubeF$ is
essentially given by composing a functor to spaces with the cubical
chains functor. Then, we observe that the homotopy colimit of the map
$\spaceF$ satisfies the defining properties of Borel equivariant
homology: it is invariant under $G$-equivariant maps which are
homotopy equivalences and agrees with the singular homology of the
quotient when the action is free.

As a first step, note that in the definition of $\cubeF$ we can
replace the smooth nerve with the usual homotopy coherent nerve. An
argument analogous to the proof of Lemma~\ref{lem:sm-nerve}
shows that the two resulting chain complexes are quasi-isomorphic.

A left action of the topological group $G$ on a space $M$ induces a
map $G_\bullet\times M_\bullet\to M_\bullet$, where $G_\bullet$
(respectively $M_\bullet$) denotes the singular set of $G$
(respectively $M$). Let $L_g\co M_n\to M_n$ denote the left action by
an $n$-simplex $g\in G_n$.  Let $\SSetsOther$ be the
simplicially-enriched category of simplicial sets and
$\SSets=\Nerve\SSetsOther$ be the quasicategory of
simplicial sets.  There is an induced functor
$\spaceF\co \Nerve BG\coloneqq \Nerve BG_\bullet\to \SSets$ which sends the
unique $0$-simplex to $M_\bullet$ and sends an $n$-simplex
$\sigma\co S[n]\to G_\bullet$ to the map
$S[n]\to \Hom_{\SSetsOther}(M_\bullet,M_\bullet)$
\[
  S[n]\stackrel{\sigma}{\longrightarrow}G_\bullet \stackrel{g\mapsto L_{g^{-1}}}{\longrightarrow}\Hom_{\SSetsOther}(M_\bullet,M_\bullet).
\]

We recall the definition of the homotopy colimit for spaces and
simplicial sets, and some basic properties of it. Given a
quasicategory $\Cat$ and a functor $F\co\Cat\to\Spaces$, the homotopy
colimit of $F$ is
\begin{equation}\label{eq:sp-hocolim-def}
  \hocolim F = \Bigl(\coprod_{n\geq 0}
  \coprod_{\sigma\in \Cat}  I^{n}\times F(\sigma_0)\Bigr)/\sim
\end{equation}
where $I=[0,1]$ and the equivalence relation $\sim$ is generated by
\begin{align*}
  (s_0\sigma;t_{n+1},\dots, t_1;
  x)&\sim(\sigma;t_{n+1},\dots, t_2; x)\\
  (s_i\sigma;t_{n+1},\dots, t_1; x)&\sim
                                     (\sigma;t_{n+1},\dots, \min\{t_{i+1},t_{i}\},\dots,t_1; x)&\text{if }i\geq 1\\
  (\sigma;t_n,\dots, t_1; x)&\sim(d_i\sigma;t_n,\dots, \widehat{t_i},\dots, t_1; x)&\text{if }t_i=\{1\}\\
  (\sigma;t_n,\dots, t_1;x)&\sim(\sigma_{i,\dots,n};t_n,\dots, t_{i+1}, F(\sigma_{0,\dots,i})(t_1,\dots,t_{i-1})(x))
                                   &\text{if } t_i=\{0\}.
\end{align*}
Call the first two kinds of moves in the equivalence relation $\sim$
\emph{degeneracy moves} and the last two \emph{face moves}. Call
face moves from left to right \emph{face deflations} and face moves
from right to left \emph{face inflations}, and define \emph{degeneracy deflations} and \emph{degeneracy inflations} similarly. The fat homotopy colimit is
defined by the same formula~\eqref{eq:sp-hocolim-def} but quotienting
only by the face moves.

If instead $F\co\Cat\to\SSets$ then the homotopy colimit is defined
similarly, but where the $k$-simplices are a quotient 
\[
  \Bigl(\coprod_{n\geq 0}\coprod_{\sigma\in \Nerve_nBG}  I^{n}_k\times F(\sigma_0)_k\Bigr)/\sim
\]
with $I$ being the standard simplicial set for the interval. The
equivalence relation $\sim$ is induced by the case of
spaces. Explicitly, there 
is a multiplication map $m\co I_\bullet\times I_\bullet\to I_\bullet$
induced by $(t_1,t_2)\mapsto \min\{t_1,t_2\}$. The degeneracy moves
become
\begin{align*}
  (s_0\sigma;t_{n+1},\dots, t_1;
  x)&\sim(\sigma;t_{n+1},\dots, t_2; x)\\
  (s_i\sigma;t_{n+1},\dots, t_1; x)&\sim(\sigma;t_{n+1},\dots, m(t_{i+1},t_{i}),\dots,t_1; x)&\text{if }i\geq 1
\end{align*}
where now the $t_i$ are $k$-simplices in $I_\bullet$ and $x$ is a
$k$-simplex in $F(\sigma_0)$. The inclusion $\{1\}\into I$ induces a map
of simplicial sets $\iota\co \{1\}\into I_\bullet$. The first face
move becomes
\begin{align*}
  (\sigma;t_n,\dots, t_1; x)&\sim(d_i\sigma;t_n,\dots,
  \widehat{t_i},\dots, t_1; x)& \text{if }t_i\in\iota(\{1\}).
\end{align*}
To understand the second face move, note that $F(\sigma_{0,\dots,i})$
is an $i$-simplex in $\SSets$, that is, a
functor from the Cordier-Vogt category $\CVcat[i]$ to $\SSetsOther$. In particular,
restricting $F(\sigma_{0,\dots,i})$ to $\Hom_{\CVcat[i]}(0,i)$ gives a
map $\Hom_{\CVcat[i]}(0,i)\times F(\sigma_0)=I_\bullet^{i-1}\times
F(\sigma_0)\to F(\sigma_i)$. This is exactly what is needed for the
second face move,
\begin{align*}
  (\sigma;t_n,\dots, t_1;x)&\sim(\sigma_{i,\dots,n};t_n,\dots, t_{i+1}, F(\sigma_{0,\dots,i})(t_1,\dots,t_{i-1})(x))
                                   &\text{if } t_i\in\iota(\{0\}),
\end{align*}
where the $t_i$ are $k$-simplices in $I_\bullet$, $x$ is a
$k$-simplex in $F(\sigma_0)$, and $\iota$ is now induced by inclusion
of $0$ into $I$.

Call a $k$-simplex $\alpha=(t_1,\dots,t_n)\in I_\bullet^n$
\emph{internal} if no $t_i$ is in the image of $\{0\}$ or $\{1\}$.
\begin{lemma}\label{lem:hocolim-inject}
  Let $F\co\Cat\to\SSets$ be a functor of quasi-categories. Given
  non-degenerate simplices $\sigma,\sigma'\in\Nerve_nBG$, interior
  simplices $\alpha,\in I_\ell^n$, $\alpha',\in I_{k}^m$,
  $\beta\in F(\sigma_0)_\ell$, $\beta'\in F(\sigma'_0)_k$, we have
  \[
    (\sigma;\alpha;\beta)\sim(\sigma';\alpha';\beta')
  \]
  if and only if $\sigma=\sigma'$, $\alpha=\alpha'$, and $\beta=\beta'$.
\end{lemma}
\begin{proof}
  We claim that if
  \[
    (\sigma;\alpha;\beta)\sim (\sigma';\alpha';\beta')
  \]
  are equivalent then we can get from one to the other by first
  applying a sequence of face moves and then a sequence of degeneracy
  moves. To see this, it suffices to see that one can replace a degeneracy move
  followed by a face move with a face move followed by a degeneracy
  move (or the identity map). This is a simple case analysis.

  Next, observe that any simplex $(\tau;\gamma;\delta)$ is equivalent
  to a unique simplex $(\tau';\gamma';\delta')$ with $\gamma'$
  internal by a sequence of face deflations. That is, the order in
  which we apply the face deflations does not matter.
  Further, any sequence of face moves can be factored as a sequence of
  face inflations followed by face deflations. It follows that any
  simplex $(\tau;\gamma=(r_1,\dots,r_p);\delta)$ is equivalent to a
  unique simplex $(\tau';\gamma';\delta')$ with $\gamma'$ internal by
  a sequence of face moves (not just deflations), and this simplex is
  obtained by applying face deflations at the $r_i$ in the image of
  $\{0,1\}$. In particular, this already implies the result if we were
  only quotienting by the face moves.

  Similarly, using the simplicial identities, it is easy to see that if $(\tau;\gamma;\delta)$ is equivalent to $(\tau';\gamma';\delta')$ by a sequence of degeneracy inflations and deflations then one can get from 
  $(\tau;\gamma;\delta)$ to $(\tau';\gamma';\delta')$ by a sequence of
  degeneracy inflations followed by degeneracy deflations. Further,
  any maximal sequence of degeneracy deflations from
  $(\tau;\gamma;\delta)$ yields the same simplex
  $(\tau';\gamma';\delta')$ with $\tau'$ nondegenerate.

  Now, suppose 
  \[
    (\sigma;\alpha;\beta)\sim (\sigma';\alpha';\beta')
  \]
  where $\alpha,\alpha'$ are internal and $\sigma,\sigma'$ are
  non-degenerate. From the previous three paragraphs, there is a
  $(\tau;\gamma;\delta)$ equivalent to $(\sigma;\alpha;\beta)$ by a sequence of
  face deflations and equivalent to $(\sigma';\alpha';\beta')$ by a sequence of
  degeneracy deflations. So, $\delta=\beta'$. Write
  $\gamma=(t_n,\dots,t_{\ell+1};0,t_{\ell-1},\dots,t_1)$ where
  $t_n,\dots,t_{\ell+1}\not\in\iota(\{0\})$. Assume further that
  $t_{i_m},\dots,t_{i_1}\in\iota(\{1\})$ or, in shorthand, $t_{i_j}=1$, where
  $n\geq i_m>\cdots>i_1>\ell$, and the other $t_j$, $j>\ell$, are not in
  $\iota(\{1\})$. Then we have
  \[
    (\sigma;\alpha;\beta)=\bigl(d_{i_1-\ell}\cdots d_{i_m-\ell}\tau_{\ell,\dots,n};t_n,\dots,\widehat{t_{i_m}},\dots,\widehat{t_{i_1}},\dots,t_{\ell+1};F(\tau_{0,\dots,\ell})(t_1,\dots,t_{\ell-1})(\delta)\bigr)
  \]
  Observe that
  \[
    d_{i_1-\ell}\cdots d_{i_m-\ell}\tau_{\ell,\dots,n}=d_{i_1-\ell}\cdots d_{i_m-\ell} d_0^\ell\tau
  \]
  
  Since $\alpha'$ is internal, we must have
  \begin{equation}\label{eq:lotsa-degens}
    \tau=s_0^\ell  s_{i_m+\epsilon_m}\cdots s_{i_1+\epsilon_1} \tau'
  \end{equation}
  for some simplex $\tau'$, where each $\epsilon_i\in \{-1,0\}$. It follows
  from the simplicial relations that $\tau'=\sigma$, which is
  nondegenerate. So, since $\sigma'$ is the unique nondegenerate simplex
  obtained from $\tau$ by degeneracy deflations, $\tau'=\sigma'$, as well, and
  $\alpha'=(t_n,\dots,\hat{t}_{i_m},\dots,\hat{t}_{i_1},\dots,t_{\ell+1})$. So,
  $\sigma=\sigma'$ and $\alpha=\alpha'$. Further,
  Equation~\eqref{eq:lotsa-degens} also implies that $\tau_{0,\dots,\ell}$ is
  fully degenerate, so
  $\beta=F(\tau_{0,\dots,\ell})(t_1,\dots,t_{\ell-1})(\delta)=\beta'$. This
  proves the result.
\end{proof}

\begin{lemma}\label{lem:sp-to-cube}
  There is an isomorphism
  \[
    H_*\hocolim\spaceF\cong H_*\fathocolim\cubeF
  \]
  of modules over $H^*(BG)$.
\end{lemma}
\begin{proof}
  An argument analogous to the proof of
  Proposition~\ref{prop:fat-hocolim-is-hocolim} shows that
  \[
    H_*\hocolim\spaceF\cong H_*\fathocolim\spaceF.
  \]
  From the definitions of the fat homotopy colimit in complexes and
  spaces, there is an evident chain map
  $\fathocolim\cubeF\to
  C_*\fathocolim\spaceF$. Filtering
  Formula~\eqref{eq:sp-hocolim-def} by the
  integer $n$ and considering the corresponding filtration of
  $\fathocolim\cubeF$, the map of $E^1$-pages of the associated
  spectral sequences is an isomorphism. It is straightforward from the
  construction of the maps that the induced maps of cohomology respect
  the $H^*(BG)$-module structure. The result follows.
\end{proof}

Let $N$ be another $G$-space, and let $\spaceF^M$ and $\spaceF^N$ be
the corresponding functors. Given a $G$-equivariant map $M\to N$ there
is an induced map $\hocolim \spaceF^M\to\hocolim\spaceF^N$.
\begin{lemma}\label{lem:equi-map}
  If $f\co M\to N$ is an equivariant map which induces an isomorphism
  on homology then the induced map $\hocolim
  \spaceF^M\to\hocolim\spaceF^N$ also induces an isomorphism on homology.
\end{lemma}
\begin{proof}
  Consider the same spectral sequence as in the proof of
  Lemma~\ref{lem:sp-to-cube} (filtering by $n$).
\end{proof}

There is a map $\pi\co \hocolim \spaceF\to M/G$ defined by
\[
  \pi(\sigma;t_n,\dots, t_1; x)=[x],
\]
the image of $x$ in $M/G$. It is clear that this map is well-defined.
\begin{lemma}\label{lem:fibration}
  Suppose that $G$ acts freely on $M$ and $M\to M/G$ is a
  fibration. (If $G$ is compact, the latter condition is automatic.)
  Then the map $\hocolim \spaceF^M\to M/G$ is a Kan fibration, and the
  fibers of this map are $\hocolim \spaceF^G$.
\end{lemma}
\begin{proof}
  To show $\pi\co \hocolim \spaceF^M\to M/G$ is a Kan fibration, assume we
  are given the following solid arrows
  \[
  \begin{tikzpicture}[xscale=2.5,yscale=1.5]
    \node (simplex) at (0,0) {$\Delta^k$};
    \node (horn) at (0,1) {$\Lambda^k_i$};
    \node (quotient) at (1,0) {$M/G$};
    \node (hocolim) at (1,1) {$\hocolim\spaceF^M$};

    \draw[right hook->] (horn) -- (simplex);
    \draw[->] (horn) -- (hocolim) node[midway,above] {\small $f$};
    \draw[->] (simplex) -- (quotient) node[midway,below] {\small $g$};
    \draw[->] (hocolim) -- (quotient) node[midway,right] {\small $\pi$};
    \draw[->,dashed] (simplex) -- (hocolim) node[midway,anchor=south east] {\small $h$};
  \end{tikzpicture}
  \]
  and we want to construct the lift $h$ shown with a dashed arrow.

  Consider the one-point $G$-space $\pt$. It is immediate from a
  result of Joyal~\cite[Corollary 1.4]{Joyal02:quasi-cat} that $\hocolim\spaceF^{\pt}$ is a
  Kan complex. The map $Q\from M\to \pt$
  induces a map $Q_*\from \hocolim \spaceF^M\to\hocolim
  \spaceF^\pt$. Choose a $k$-simplex $\wt\tau$ in $\hocolim \spaceF^\pt$
  that fills the horn $Q_*\circ f\from \Lambda^k_i\to \hocolim
  \spaceF^\pt$. Recall that the $k$-simplices of $\hocolim \spaceF^\pt$ are
  $\Bigl(\coprod_{n\geq 0} \coprod_{\sigma\in \Nerve_nBG}
  I^{n}_k/\sim\Bigr)$. Write the $I^n$ corresponding to
  $\sigma\in\Nerve_nBG$ as ${}_{\sigma}I^{n}$, and assume $\wt\tau$ is the
  image under $\sim$ of some $\tau \in {}_\sigma I^n_k$.

  Consider the $k$ faces of $\Lambda^k_i$. By applying the map $f$, we
  get $k$ $(k-1)$-simplices in $\hocolim \spaceF^M$; assume they are
  images under $\sim$ of $\alpha_j\in {}_{\mu_j}I^{n_j}_{k-1}\times
  M_{k-1}$, with $0\leq j\leq k,j\neq i$. The map $Q_*$ is simply the
  projection to the first factor, and hence
  $Q_*(\alpha_j)\in{}_{\mu_j}I^{n_j}_{k-1}$ agrees with the face
  $d_j\tau\in {}_{\sigma}I^{n}_{k-1}$---that is, they are related by a
  sequence of equivalence relation atoms $\sim$ for $\hocolim
  \spaceF^\pt$ from Equation~\eqref{eq:sp-hocolim-def}. Apply the
  corresponding sequence of equivalence relation atoms to
  $\alpha_j\in{}_{\mu_j}I^{n_j}_{k-1}\times M_{k-1}$, but in $\hocolim
  \spaceF^M$, to get an equivalent $\beta_j\in {}_{\sigma}I^n_{k-1}\times M_{k-1}$.
  (This sequence will change the projection to $M_{k-1}$ by pointwise
  group actions.) In more detail, if the sequence applies the first
  kind of degeneracy move
  \[
    (s_0\sigma;t_{n+1},\dots,t_1;\pt)\to
    (\sigma;t_{n+1},\dots,t_2;\pt)\quad\text{or}\quad 
    (s_0\sigma;t_{n+1},\dots,t_1;\pt)\leftarrow
    (\sigma;t_{n+1},\dots,t_2;\pt) 
  \]
  to $Q_*(\alpha_j)$, say, then apply the corresponding degeneracy move
  \[
    (s_0\sigma;t_{n+1},\dots,t_1;x)\to
    (\sigma;t_{n+1},\dots,t_2;x)\quad\text{or}\quad 
    (s_0\sigma;t_{n+1},\dots,t_1;x)\leftarrow
    (\sigma;t_{n+1},\dots,t_2;x)
  \]
  to $\alpha_j$. The second degeneracy move and first face move are
  similar. For the second face move, if the sequence applies
  the equivalence
  \[
    (\sigma;t_n,\dots,t_1;\pt)\to (\sigma_{i,\dots,n};t_n,\dots,t_{i+1},\spaceF(\sigma_{0,\dots,i})(t_1,\dots,t_{i-1})(\pt)=\pt)
  \]
  to $Q_*(\alpha_j)$, say, then apply the equivalence
  \[
    (\sigma;t_n,\dots,t_1;x)\to (\sigma_{i,\dots,n};t_n,\dots,t_{i+1},\spaceF(\sigma_{0,\dots,i})(t_1,\dots,t_{i-1})(x))
  \]
  to $\alpha_j$. If the sequence applied the equivalence
  \[
    (\sigma;t_n,\dots,t_1;\pt)\leftarrow (\sigma_{i,\dots,n};t_n,\dots,t_{i+1},\spaceF(\sigma_{0,\dots,i})(t_1,\dots,t_{i-1})(\pt)=\pt)
  \]
  to $Q_*(\alpha_j)$ then apply the equivalence
  \[
    (\sigma;t_n,\dots,t_1;\spaceF(\sigma_{0,\dots,i})(t_1,\dots,t_{i-1})^{-1}(x))\leftarrow (\sigma_{i,\dots,n};t_n,\dots,t_{i+1},x)
  \]
  to $\alpha_j$. (Since $\spaceF$ is given by the group action, it is invertible.)
  
  We claim that applying the second projection to the $\beta_j$ gives
  a horn $g'\from
  \Lambda^k_i\to M$ fitting into the following square
  \[
    \begin{tikzpicture}[xscale=2.5,yscale=1.5]
      \node (simplex) at (0,0) {$\Delta^k$};
      \node (horn) at (0,1) {$\Lambda^k_i$};
      \node (quotient) at (1,0) {$M/G$};
      \node (hocolim) at (1,1) {$M$};
      
      \draw[right hook->] (horn) -- (simplex);
      \draw[->] (horn) -- (hocolim) node[midway,above] {\small $g'$};
      \draw[->] (simplex) -- (quotient) node[midway,below] {\small $g$};
      \draw[->] (hocolim) -- (quotient);
      \draw[->,dashed] (simplex) -- (hocolim);
    \end{tikzpicture}
  \]
  (without the dashed arrow).  To see this, we must verify that the
  maps $\beta_q$ agree on their common faces. For definiteness, fix
  $0\leq p<q\leq k$. We claim that $\bdy_p\beta_q=\bdy_{q-1}\beta_p$.

  By construction, $\bdy_p\beta_q\sim \bdy_{q-1}\beta_p$ and also
  $Q_*(\bdy_p\beta_q)=Q_*(\bdy_{q-1}\beta_p)$. We claim that these two
  conditions imply that $\bdy_p\beta_q=\bdy_{q-1}\beta_p$. By applying
  corresponding sequences face moves, we may assume that
  $\bdy_p\beta_q$ and $\bdy_{q-1}\beta_p$ are internal
  and by applying corresponding sequences of
  degeneracy moves, we may assume $\bdy_p\beta_q$ and
  $\bdy_{q-1}\beta_p$ correspond to the same nondegenerate simplex
  $\sigma$ in $\Nerve_nBG$. It then follows from
  Lemma~\ref{lem:hocolim-inject} that $\bdy_p\beta_q=\bdy_{q-1}\beta_p$,
  as desired.
  
  Now, since $M\to M/G$ is a fibration, there is a simplex $\omega\in
  M_k$ producing the dashed lift. Then (the equivalence class of) the
  $k$-simplex $(\tau,\omega)\in {}_\sigma I^n_k\times M_k$ produces
  the required lift $h \from \Delta^n\to \hocolim\spaceF^M$; this
  establishes that $\pi$ is a Kan fibration.
  
  The fact that the fiber is $\hocolim \spaceF^G$ is immediate from
  the definitions.
\end{proof}

\begin{lemma}\label{lem:contractible}
  For $M=G$ the space $\hocolim \spaceF^G$ is contractible.
\end{lemma}
\begin{proof}[Proof, courtesy of Tyler Lawson]
  Let $\Cat$ be the trivial category, with a single object and only
  the identity morphism. There is a projection
  $\pi\co \Nerve BG\to \Cat$ and an inclusion
  $i\co \Cat\to \Nerve BG$. The homotopy colimit of $\spaceF^G$ is the
  homotopy left Kan extension of $\spaceF^G$ along $\pi$. Let
  $F\co\Cat\to\SSets$ be the functor sending the object of $\Cat$ to
  a one-point space.  From the pointwise construction of homotopy left
  Kan extension (see~\cite[Section 4.3]{Lurie09:topos}), the diagram
  $\spaceF^G$ is the homotopy left Kan
  extension of $F$ along $i$. Since Kan extension respects
  composition, the homotopy colimit of $\spaceF^G$ is weakly
  equivalent to the left Kan extension of $F$ along
  $\pi\circ i=\Id_\Cat$, which is just a one-point space.
\end{proof}

\begin{proof}[Proof of Proposition~\ref{prop:sing-is-Borel}]
  By Lemma~\ref{lem:sp-to-cube}, it suffices to prove that the
  homology of $\hocolim\spaceF$ agrees with the usual Borel
  equivariant homology of $M$. By Lemma~\ref{lem:equi-map}, we may
  replace $M$ by $M\times EG$, so the action of $G$ is free and $M\to
  M/G$ is a fibration. By Lemmas~\ref{lem:fibration}
  and~\ref{lem:contractible}, there is a homotopy equivalence
  $\hocolim\spaceF\to M/G$.  Hence, the homology of $\spaceF$ is
  isomorphic to the homology of $M/G$; this is also the Borel
  equivariant homology of the free $G$-space $M$.
\end{proof}

\begin{proof}[Proof of Proposition~\ref{prop:Morse-is-Borel}]
  By Proposition~\ref{prop:fat-hocolim-is-hocolim}, $H^*(\hocolim (\MorseFunc\circ\BGSection))\cong H^*(\fathocolim (\MorseFunc\circ\BGSection))$. By
Lemma~\ref{lem:Morse-is-sing}, $H^*(\fathocolim (\MorseFunc\circ\BGSection))\cong H^*(\fathocolim \cubeF)$. By
  Proposition~\ref{prop:sing-is-Borel}, $H^*(\fathocolim \cubeF)\cong H^*(M\times_G EG)$.
\end{proof}

\subsection{A spectral sequence}\label{sec:sseq} 
Let $\pi_0(G)$ denote the quotient of $G$ by the connected
component containing the identity. So, $\pi_0(G)$ is a discrete group,
with one element per connected component of $G$. There is an action
of $\pi_0(G)$ on $\HF(L_0^H,L_1)$ defined as follows. Given $[g]\in
\pi_0(G)$ the action of $[g]$ is the composition
\[
  \HF(L_0^H,L_1)\to \HF(L_0^{gH},L_1)\to \HF(L_0^H,L_1)
\]
where the first map is the action of $g$ on $M$, which sends
$L_0^H\cap L_1$ to $L_0^{gH}\cap L_1$, and the second map is the
continuation map associated to the Hamiltonian isotopy $L_0^{gH}\sim
L_0\sim L_0^H$ induced by the inverse of $gH$ and then $H$. 
That this defines an action of $\pi_0(G)$ on $\HF(L_0^H,L_1)$ follows
from Lemma~\ref{lem:is-sset-map} (considering only $1$- and
$2$-morphisms), and is also not hard to prove directly.
Since $\pi_1(BG)=\pi_0(G)$, we can view the action of $\pi_0(G)$ on
$\HF(L_0^H,L_1)$ as a local system over $BG$.

\begin{proposition}\label{prop:sseq}
  Assume $G$ is a compact Lie group, and that we are working with
  Floer complexes with coefficients in a field $\Field$. Then there is
  a spectral sequence with $E^2$-page given by
  $H^*(BG;\HF(L_0^H,L_1;\Field))$ converging to
  $\eHF[G](L_0,L_1;\Field)$. In particular, if $G$ is connected
  then there is a spectral sequence $H^*(BG;\Field)\otimes
  \HF^*(L_0^H,L_1;\Field)\Rightarrow \eHF[G](L_0,L_1;\Field)$.
\end{proposition}

\begin{proof}[Proof of Proposition~\ref{prop:sseq}]
  Let $\chainFb\co\Nerve B G\to\ChainComplexes$ be the functor
  defined in Section \ref{sec:build-diag}. The equivariant cohomology
  $\eHF[G](L_0,L_1;\Field)$ is the cohomology of
  $\hocolim\chainFb$. Generators of $\hocolim\chainFb$ are of the form
  $[\sigma]\otimes \x$ where $\sigma$ is a non-degenerate simplex in
  $\Nerve B G$ and $\x\in\CF(L_0^H,L_1)$.  Recall the (increasing)
  filtration on $\hocolim \chainFb$ defined by setting the filtration
  of $[\sigma]\otimes\x$ to be the dimension of $\sigma$. We
  claim that the $E_2$-term of the associated spectral sequence is
  $H^*(BG;\HF^*(L_0^H,L_1;\Field))$.

  To see this, note that there is another functor, $H$,
  defined by 
  \[
    H(o)=\HF(L_0^H,L_1;\Field)
  \]
  for $o$ the unique object of $BG$;
  $\maxmap{H}(\sigma)=\chainFb(\sigma)_*\co
  \HF(L_0^H,L_1;\Field)\to \HF(L_0^H,L_1;\Field)$ for $\sigma$ a $1$-simplex
  ($1$-morphism) in $\Nerve BG$ (i.e., an element of $G$); and
  $\maxmap{H}(\sigma)=0$ for $\sigma$ an $n$-morphism,
  $n>1$. A small extension of Lemma~\ref{lem:hocolim-geom-realization}
  shows that the hypercohomology of $H$ is exactly the
  cohomology of $BG$ with coefficients in $\HF(L_0^H,L_1;\Field)$.

  The $E^1$-page of the spectral sequence for
  $H^*(\hocolim \chainFb;\Field)$ is $\Hom(\hocolim(H),\Field)$, and the
  $d_1$-differential agrees with the differential on
  $\Hom(\hocolim(H),\Field)$. It follows that the $E^2$-page of the
  spectral sequence is $H^*(\hocolim(H);\Field)\cong H^*(BG;\HF^*(L_0^H,L_1;\Field))$,
  as desired.
\end{proof}

\subsection{The Seidel-Smith example}\label{sec:kh-symp}
In this section, we define $O(2)$-equivariant symplectic Khovanov
homology. Symplectic Khovanov homology was defined by
Seidel-Smith~\cite{SeidelSmith6:Kh-symp}. We will work in Manolescu's Hilbert
scheme formulation~\cite{Manolescu06:nilpotent}; see also our previous
paper~\cite{HLS:HEquivariant} for a brief review. Very briefly, consider a degree $2n$
polynomial $p(z)$ with simple roots and the affine
algebraic surface $S=\{(u,v,z)\in\CC^3\mid u^2+v^2+p(z)=0\}$.
There is an open subspace $\ssspace{n}$ of the Hilbert scheme $\Hilb^n(S)$, and submanifolds
$L_0,L_1\subset \ssspace{n}$ associated to a bridge diagram for a link
$K$, so that the symplectic Khovanov homology
$\KhSymp(K)=\HF(L_0,L_1)$ inside $\ssspace{n}$. There is also an identification
\begin{align*}
  L_0&=\Sigma_{A_1}\times\cdots\times\Sigma_{A_n}\subset (S^{\times n}\setminus\Delta)/(S_n)\subset \Hilb^n(S)\\
  L_1&=\Sigma_{B_1}\times\cdots\times\Sigma_{B_n}\subset (S^{\times n}\setminus\Delta)/(S_n)\subset \Hilb^n(S)
\end{align*}
where each $\Sigma_{A_i}$ and $\Sigma_{B_i}$ is a $2$-sphere in $S$ and $\Delta$
denotes the fat diagonal.

The standard action of $O(2)$ on $(u,v)$ induces an action of $O(2)$
on the algebraic surface $S$ which
preserves the Lagrangians (see, e.g.,~\cite[Section
7.1]{HLS:HEquivariant}). We considered the subgroups $D_{2^n}$ of
$O(2)$ and defined a $D_{2^n}$-equivariant symplectic Khovanov
homology over $\FF_2$. Our goal here is to extend this to a definition
of $O(2)$-equivariant symplectic Khovanov homology, over an arbitrary
ring $\Ring$.

To produce a symplectic form we will use a well-known averaging procedure. For lack of a reference, we note the following lemma:
\begin{lemma}\label{lem:average}
  Let $G$ be a compact Lie group acting on a smooth manifold $M$ and
  $\alpha\in\Omega^k(M)$. Let $\mathit{dg}$ denote Haar measure on $G$. Then the
  exterior derivative satisfies
  \[
    d\left(\int_{g\in G} g^*\alpha \mathit{dg}\right)=\int_{g\in G}g^*(d\alpha)\mathit{dg}.
  \]
\end{lemma}
\begin{proof}
  Define a $k$-form $\wt{\alpha}$ on $G\times M$ by
  $\wt{\alpha}_{g,p}(v_1,\dots,v_k)=\alpha\bigl(g_*(\pi_*(v_1)),\dots,g_*(\pi_*(v_k))\bigr)$. Then
  \[
    \int_{g\in G} g^*\alpha \mathit{dg}=\int_G \mathit{dg}\wedge \wt{\alpha}
  \]
  where the right side denotes integration along the fibers.
  If $\mu\co G\times M\to M$ denotes the action then
  $\mathit{dg}\wedge \wt{\alpha}=\mathit{dg}\wedge
  \mu^*(\alpha)$.
  Further, fiber integration commutes with the exterior derivative
  (see, e.g.,~\cite[Proposition 6.14.1]{BottTu} for the case of
  compact vertical cohomology of vector bundles or~\cite[Section
  VII.5]{GreubHalperinVanstone} in general). Thus, 
  \[
    d\left(\int_{g\in G} \!\!\!\!\!g^*\alpha \thinspace\mathit{dg}\right)
    =d\left(\int_G \mathit{dg}\wedge \mu^*(\alpha)\right)
    =\int_G \mathit{dg}\wedge d(\mu^*(\alpha))
    =\int_G \mathit{dg}\wedge \mu^*(d\alpha)
    =\int_{g\in G}\!\!\!\!\!g^*(d\alpha)\mathit{dg},
  \]
  as desired.
\end{proof}

\begin{lemma}\label{lem:O-invt-form}
  There is a symplectic form $\omega$ on $\ssspace{n}$ which is
  $O(2)$-invariant and which agrees with the product symplectic form
  $\omega_{S}^{\times n}$ outside an open neighborhood of the diagonal
  $\Delta$. Further, $\omega$ is exact and there is an
  $O(2)$-invariant complex structure $I$ compatible with $\omega$ so
  that $(\ssspace{n},I)$ is $I$-convex at infinity.
\end{lemma}
\begin{proof}
  Abouzaid-Smith~\cite[Lemma 5.5]{AbouzaidSmith:arc-alg} construct a
  K\"ahler form $\omega'$ on $\Hilb^n(S)$ which agrees with the
  product symplectic form outside an open neighborhood of the diagonal
  and the restriction of which to $\ssspace{n}$ is exact. (See
  also~\cite[Theorem 1.2]{Manolescu06:nilpotent} and~\cite[Section
  4.2]{SeidelSmith10:localization}.)  We modify their K\"ahler form
  $\omega'$ to be $O(2)$-invariant.  The complex structure on
  $\Hilb^n(S)$ is $O(2)$-invariant, by naturality of the Hilbert
  scheme construction. As in Lemma~\ref{lem:average}, let
  $\mathit{dg}$ denote Haar measure on $O(2)$, and let
  \[
    \omega(v,w)=\int_{g\in O(2)}\omega'(g_*v,g_*w)\mathit{dg}.
  \]
  Using the facts that $\omega'(v,Iv)>0$ for all $v\neq0$ and
  that $I$ is preserved by the $O(2)$-action, it is
  easy to see that $\omega$ is a non-degenerate $2$-form on
  $\Hilb^n(S)$, and $G$-invariance of $dg$ implies $G$-invariance of
  $\omega$. It follows from Lemma~\ref{lem:average} that $\omega$ is
  closed. In fact, applying Lemma~\ref{lem:average} to the primitive
  for $\omega'$, the averaged form $\omega$ is also exact. Finally,
  taking $I$ to be the restriction of the complex structure from
  $\Hilb^n(S)$ (i.e., the same complex structure we have been
  discussing), Manolescu~\cite{Manolescu06:nilpotent} showed that
  $(\ssspace{n},I)$ is biholomorphic to
  Seidel-Smith's~\cite{SeidelSmith6:Kh-symp} original $\ssspace{n}$
  which in turn is an affine variety, hence convex at infinity. 
\end{proof}

The first Chern class $c_1(\ssspace{n})$
vanishes~\cite[p.~501]{SeidelSmith6:Kh-symp} (see also~\cite[Section
6.2]{Manolescu06:nilpotent}), so one can define a squared phase map;
an explicit description of this map is given by
Manolescu~\cite[Section 6.2]{Manolescu06:nilpotent}. Since the
Lagrangians are simply connected, they necessarily admit gradings.
\begin{lemma}\label{lem:Kh-phase}
  The squared phase map on the Lagrangians $L_i$ can be
  chosen to be $O(2)$-equivariant.
\end{lemma}
\begin{proof}
  We will be concrete; a more abstract argument is left to the reader.
  By Equation~\eqref{eq:phase}, it suffices to find a complex volume form $\eta$ so that $\eta^2$ is $O(2)$-equivariant.
  With notation as in Manolescu's paper~\cite[Section
  6.2]{Manolescu06:nilpotent}, note that
  \[
    \frac{dv_j\wedge dz_j}{2u_j}=-\frac{du_j\wedge dz_j}{2v_j}=\frac{du_j\wedge dv_j}{p'(z_j)}.
  \]
  The complex volume form is given by
  \[
    \eta = \prod_{j=1}^n \frac{dv_j\wedge dz_j}{2u_j}
  \]
  on the open set where $z_i\neq z_j$ for all $i\neq j$. In particular, this formula holds on a dense, open subset of $\ssspace{n}$.
  Any element $A$ of $O(2)$ can be written as either
  \[
    \begin{pmatrix}
      a & b\\
      -b & a
    \end{pmatrix}
    \qquad\text{or}\qquad
    \begin{pmatrix}
      a & b\\
      b & -a
    \end{pmatrix}
  \]
  depending on whether the matrix is in $\SO(2)$ or not. Making the
  change of variables $(u_i,v_i)^T=A(x_i,y_i)^T$ gives 
  \begin{align*}
    \eta&=\prod\frac{(-bdx_j+ady_j)\wedge dz_j}{2(ax_j+by_j)}
          =\prod\left(\frac{ax_j+by_j}{ax_j+by_j}\right)\left(\frac{dy_j\wedge dz_j}{2x_j}\right)\\
    \shortintertext{or}
    \eta&=\prod\frac{(bdx_j-ady_j)\wedge dz_j}{2(ax_j+by_j)}
          =\prod\left(\frac{-ax_j-by_j}{ax_j+by_j}\right)\left(\frac{dy_j\wedge dz_j}{2x_j}\right)
  \end{align*}
  depending on whether $\det(A)=1$ or $-1$. That is,
  either $A^*\eta=\eta$ or $A^*\eta=-\eta$, respectively. In either
  case, $\eta^2$ is preserved by $A$. Since preserving $\eta^2$ is a closed condition, the result follows.
\end{proof}

\begin{lemma}\label{lem:Kh-pin-profile}
  Let $P_i$ be the unique $\pin$-structure on $L_i$. Then the $\pin$
  profile associated to $(L_0,O(2),P_0)$ is the same as the $\pin$
  profile associated to $(L_1,O(2),P_1)$.
\end{lemma}

\begin{proof} 
  First, observe that each Lagrangian $L_i$ consists of a product of
  $n$ two-spheres $S^2$. The action of $O(2)$ on each
  $S^2 = \CC \cup \{\infty\}$ factor is given by rotations about the origin
  and reflections across lines through the origin, and the action on $L_i$
  is the product of these actions. This implies that the pin profiles
  associated to $(L_0,O(2),P_0)$ and $(L_1,O(2),P_1)$ must be the same.
\end{proof}

\begin{corollary}
  There is a $O(2)$-orientation for $(\ssspace{n},L_0,L_1)$.
\end{corollary}
\begin{proof}
  By Lemma~\ref{lem:Kh-phase} there is a $O(2)$-invariant square phase
  map, and since the $O(2)$-action has a fixed point (in fact,
  $2^{|K|}$ fixed points) Hypothesis~\ref{hyp:gr-invt} is
  automatically satisfied. By Lemma~\ref{lem:Kh-pin-profile} the
  $\pin$ profiles of the two Lagrangians are the same, so this follows
  from Proposition~\ref{prop:spin-or}.
\end{proof}

\begin{lemma}
  The triple $(\ssspace{n},L_0,L_1)$ satisfies
  Hypotheses~\ref{item:J-1} and~\ref{item:J-2}.
\end{lemma}
\begin{proof}
  By Lemma~\ref{lem:O-invt-form}, $(\ssspace{n},\omega)$ is exact and admits an $O(2)$-invariant, convex at infinity complex structure $I$.
  For notational
  convenience, write $\omega=d\lambda$. Each of the Lagrangians $L_0$ and $L_1$
  is a product of $2$-spheres, hence simply-connected, and
  therefore necessarily exact. This addresses
  Hypothesis~\ref{item:J-2}. Turning to Hypothesis~\ref{item:J-1},
  consider a loop of paths
  \[
  v\co \bigl([0,1]\times S^1,\{0\}\times S^1,\{1\}\times S^1\bigr)\to (\ssspace{n},L_0,L_1).
  \]
  The $\omega$-area of $v$ is
  $\int_{[0,1] \times S^1} v^*\omega = \int_{\{1\} \times S^1}
  v^*\lambda - \int_{\{0\}\times S^1} v^*\lambda = 0$, since
  $\lambda|_{L_i}$ is exact for $i=0,1$. For the second part of the
  hypothesis, consider the bundle $v^*T\ssspace{n}$, which is
  trivializable as a complex vector bundle. Recall that the Maslov
  index of $v$ is twice the relative Chern class of the bundle
  $(v^*T\ssspace{n}, \bigl((v|_{\{0\} \times S^1}^*TL_0)\otimes_{\RR}
  \CC) \amalg ((v|_{\{1\} \times S^1}^*TL_1) \otimes_{\RR} \CC)\bigr)$
  over
  $\bigl([0,1] \times S^1, (\{0\} \times S^1) \amalg (\{1\}\times
  S^1)\bigr)$. (Since $L_i$ is Lagrangian, the bundle
  $v^*TL_i \otimes_{\RR} \CC$ comes with an isomorphism to
  $v^*T\ssspace{n}$.) Since $L_i$ is simply-connected, we may choose a
  nulhomotopy of each $v(\{i\} \times S^1)$ in $L_i$ and use those
  maps to extend $v$ to a map
  $v\co(S^2, \bD^2, \bD^2) \to (M, L_0, L_1)$.
  By excision, we see that the relative
  Chern class of
  $\bigl((v^*T\ssspace{n}, ((v|_{\{0\} \times S^1}^*TL_0)\otimes_{\RR} \CC)
  \amalg ((v|_{\{1\} \times S^1}^*TL_1) \otimes_{\RR} \CC)\bigr)$ over
  $\bigl([0,1] \times S^1, \{0\} \times S^1 \amalg \{1\} \times S^1\bigr)$ is
  identified with the relative Chern class of
  $\bigl(v^*T\ssspace{n}, ((v|_{\bD^2}^*TL_0)\otimes_{\RR} \CC) \amalg
  ((v|_{\bD^2}^*TL_1) \otimes_{\RR} \CC)\bigr)$ over
  $(S^2, \bD^2 \amalg \bD^2)$. Furthermore, the long
  exact sequence for the pair $(S^2,\bD^2\amalg \bD^2)$ implies that the Chern
  classes of a relative bundle over $(S^2, \bD^2\amalg \bD^2)$ vanish
  if and only if the ordinary Chern classes of its restriction to
  $S^2$ vanish. But indeed $c_1(\ssspace{n})$ vanishes, so
  $c_1(v^*T\ssspace{n})=0$ as well. This implies that the Maslov index
  of $v$ is zero.
\end{proof}

\begin{corollary}
  For any coefficient ring $\Ring$, any subgroup $G$ of $O(2)$, and
  any bridge diagram $K$ for a link, there is a $G$-equivariant
  symplectic Khovanov homology $\KhSymp^{G}(K;\Ring)$ of $K$ with
  coefficients in $\Ring$. Up to isomorphism over $H^*(G;\Ring)$,
  $\KhSymp^{G}(K;\Ring)$ is an invariant of the bridge diagram for
  $K$.
\end{corollary}

In the interest of brevity, we will not attempt to prove:
\begin{conjecture}
  For coefficient ring $\Ring$ and subgroup $G$ of $O(2)$, the
  $G$-equivariant symplectic Khovanov homology
  $\KhSymp^{G}(K;\Ring)$ of $K$ is a link invariant.
\end{conjecture}


\bibliographystyle{hamsalpha}
\bibliography{heegaardfloer}
\end{document}
